\documentclass[a4paper,10pt,reqno]{amsart}
\usepackage{a4wide}
\usepackage{amsmath,amsfonts,amssymb,amsthm}
\usepackage{tgpagella}
\usepackage{tikz}
\usetikzlibrary{calc}
\usepackage{enumitem}
\usepackage{hyperref}
\usepackage{float} 
\usepackage{booktabs}

\usepackage{longtable}
\usepackage{caption}
\usepackage{graphicx}
\usepackage{pict2e}

\calclayout

\theoremstyle{plain}
\newtheorem{theorem}{Theorem}[section]
\newtheorem*{theorem*}{Theorem}
\newtheorem{lemma}[theorem]{Lemma}
\newtheorem{corollary}[theorem]{Corollary}

\theoremstyle{definition}
\newtheorem{definition}[theorem]{Definition}
\newtheorem{remark}[theorem]{Remark}
\newtheorem{example}[theorem]{Example}

\definecolor{darkblue}{rgb}{0,0,0.7} % darkblue color
\newcommand{\darkblue}{\color{darkblue}} % darkblue command
\newcommand{\defn}[1]{\emph{\darkblue #1}} % emphasis of a definition

\DeclareMathAlphabet{\mathdutchcal}{U}{dutchcal}{m}{n}
\newcommand{\N}{\mathbb{N}}
\newcommand{\R}{\mathbb{R}}
\newcommand{\A}{\mathdutchcal A}
\newcommand{\AG}{\mathcal{A}}
\newcommand{\F}{\mathdutchcal F}
\renewcommand{\H}{\mathdutchcal H}
\newcommand{\Sh}{\mathdutchcal{Sh}}
\newcommand{\Reg}{\mathdutchcal{R}}
\newcommand{\PosBA}{P_B(\A)}
\newcommand{\Lines}{\mathdutchcal{L}}
\newcommand{\cov}{\mathdutchcal{c}}
\newcommand{\covd}{\mathdutchcal{d}}
\newcommand{\hypint}{\mathdutchcal{h}}
\renewcommand{\P}{\mathbf{P}}
\newcommand{\ce}{\mathbf{c}}
\newcommand{\x}{\mathbf{x}}
\newcommand{\y}{\mathbf{y}}
\newcommand{\m}{\mathbf{m}}
\newcommand{\n}{\mathbf{n}}
\newcommand{\p}{\mathbf{p}}
\newcommand{\ve}{\mathbf{v}}
\newcommand{\z}{\mathbf{z}}
\DeclareMathOperator{\sign}{sign}
\DeclareMathOperator{\Sep}{Sep}
\DeclareMathOperator{\con}{con}
\DeclareMathOperator{\Con}{Con}
\DeclareMathOperator{\inter}{int}
\DeclareMathOperator{\aff}{aff}
\DeclareMathOperator{\conv}{conv}
\DeclareMathOperator{\spa}{span}
\DeclareMathOperator{\pre}{pre}

\makeatletter
\newcommand{\addresseshere}{%
  \enddoc@text\let\enddoc@text\relax
}
\makeatother

\title{Congruence Normality of Simplicial Hyperplane Arrangements via Oriented Matroids}
\date{\today}

\author[M.~Cuntz]{Michael Cuntz}
\address[M.~Cuntz]{Leibniz Universit\"at Hannover, Institut f\"ur Algebra, Zahlentheorie und Diskrete Mathematik, Fakult\"at f\"ur Mathematik und Physik, Welfengarten 1, D-30167 Hannover, Germany}
\email{cuntz@math.uni-hannover.de}
\urladdr{https://www.iazd.uni-hannover.de/de/cuntz}
\author[S.~Elia]{Sophia Elia$^{ \sharp}$}
\address[S.~Elia, J.-P. Labb\'e]{Institut f\"ur Mathematik, Freie Universit\"at Berlin, Arnimallee 2, 14195 Berlin, Germany}
\email{sophiae56@math.fu-berlin.de, labbe@math.fu-berlin.de}
\urladdr{http://page.mi.fu-berlin.de/\{sophiae56\},\{labbe\}}
\thanks{$^{ \sharp}$With the support of the Research Training Group 2434 ``Facets of Complexity''.}

\author[J.-P.~Labb\'e]{Jean-Philippe Labb\'e$^{\flat}$}
\thanks{$^{\flat}$With the support of the DFG Collaborative Research Center TRR~109 ``Discretization in Geometry and Dynamics''.}


\begin{document}
% \sloppy
\begin{abstract}
A catalogue of simplicial hyperplane arrangements was first given by Gr\"unbaum in 1971.
These arrangements naturally generalize finite Coxeter arrangements and also the weak order through the poset of regions.
The weak order is known to be a congruence normal lattice, and congruence normality of lattices of regions of simplicial arrangements can be determined using polyhedral cones called \emph{shards}.

In this article, we update Gr\"unbaum's catalogue by providing normals realizing all known simplicial arrangements with up to 37 lines and key invariants.
Then, we add structure to this catalogue by determining which arrangements \emph{always}/\emph{sometimes}/\emph{never} lead to congruence normal lattices of regions.
To this end, we use oriented matroids to recast shards as covectors to determine congruence normality of large hyperplane arrangements.
We also show that lattices of regions coming from finite Weyl groupoids of any rank are always congruence normal.

\end{abstract}

\maketitle

\section{Introduction}

% Simplicial Hyperplane Arrangements
The first catalogue of simplicial hyperplane arrangements of rank~$3$ appeared in 1971~\cite{grunbaum_1971}.
This catalogue included three infinite families and 90 sporadic arrangements.
Since then, the catalogue has changed: certain arrangements have been found to be isomorphic while some new arrangements have also been found, bringing the number of sporadic arrangements to 95~\cite{grunbaum_2009,cuntz_greedy_2020}.
The list is known to be complete for arrangements with up to 27 lines~\cite{cuntz_simplicial_2012}.
The current updated list along with several invariants is available in the present article, see Table~\ref{tab:invariants} in Section~\ref{sec:invariants} and Appendices~\ref{app:lists} and~\ref{app:wd}.
The following questions are still open:
\emph{Is the list complete? Is there a finite list at all?}
In order to answer these questions, it is natural to search for structures lurking behind the list.
Simplicial arrangements can be thought of as generalizations of finite Coxeter arrangements.
Furthermore, they correspond to normal fans of simple zonotopes \cite[Theorem~7.16]{ziegler_lectures_1995}. 
What other combinatorial/geometric/algebraic structures regulate the list?

%Finite Weyl groupoids
It turns out that finite Weyl groupoids provide an algebraic justification for around half of the sporadic arrangements.
Finite Weyl groupoids are algebraic structures generalizing Weyl groups that were introduced to better understand the symmetries of Nichols algebras and related Hopf algebras~\cite{heckenberger_weyl_2006,heckenberger_geometric_2011,cuntz_finite_2015}.
Each Weyl groupoid originates from the data of a ``Cartan graph'', leading to a so-called ``root system''.
In turn, these root systems generalize the usual notion of root system of a Weyl group.
Notably, they form the set of normals of certain simplicial hyperplane arrangements.
Finite Weyl groupoids of rank~$3$ have been classified and account for $53$ simplicial arrangements~\cite{cuntz_crystallographic_2011}. 

%Poset of regions
Posets of regions form a family of combinatorial structures that encode detailed information on hyperplane arrangements and the adjacency of regions.
For simplicial arrangements, the posets of regions are always lattices, no matter what the \emph{base region} of the poset is~\cite[Theorem~3.4]{bjoerner_hyperplane_1990}.
Reading showed that simpliciality can be weakened to \emph{tightness}---which is a connectivity condition on facets of regions---in order to obtain lattices~\cite[Chapter~9]{reading_lattice_2016} (see Lemma~\ref{lem:bez} below).
Once again, simplicial arrangements through their lattices of regions provide generalizations, this time of the weak order of finite Coxeter groups.
Unlike in the Coxeter case, a simplicial arrangement may lead to several non-isomorphic lattices of regions.
Apart from being lattices, much less is known about the poset of regions of simplicial arrangements.

%Lattice congruences
Lattice congruences of the weak order of Coxeter arrangements generate several objects of study.
For example, the permutahedron is perhaps the most studied example of a simple zonotope, that comes from the braid arrangement, or Coxeter arrangement of type $A$.
The corresponding poset of regions is the weak order of the symmetric group and is a lattice.
Moreover, Tamari and Cambrian lattices, generalized permutahedra, and associahedra are all related to lattice congruences of the weak order~\cite{reading_cambrian_2006,postnikov_permutahedra_2009,hohlweg_permutahedra_2011}.
In particular, in type~$A$ and $B$, every lattice congruence leads to a polytope~\cite{pilaud_quotientopes_2019,padrol_shard_2020}.
To which extent do these constructions extend to general simplicial arrangements?
We focus here on two important properties used to study lattice congruences and shard polytopes: \emph{congruence normality} and \emph{congruence uniformity}.
Coxeter arrangements are congruence normal and uniform \cite{caspard_cayley_2004}.
Congruence uniform lattices admit a bijection between their join-irreducible elements and the join-irreducible elements in the lattice of lattice congruences. 
Congruence uniform lattices are thus particularly nice lattices as they allow one to more easily study the lattice of congruences.
Reading characterized congruence uniformity of posets of regions using \emph{tightness} and \emph{shards} (i.e.\ pieces of hyperplanes)~\cite[Corollary~9-7.22]{reading_lattice_2016}.
Reading also showed that supersolvable hyperplane arrangements have congruence uniform posets of regions for some canonical choice of base region~\cite{reading_lattice_2003}.
Congruence uniform lattices admit a combinatorial construction whose geometric aspects in this context have yet to be explored in detail.

In this article, we determine congruence uniformity and normality of posets of regions of simplicial hyperplane arrangements of rank~$3$ and draw several conclusions.
To do so, we approach posets of regions through the oriented matroids naturally associated to the normals of the hyperplane arrangements, which are presented in Appendix~\ref{app:lists}.
Covectors of the oriented matroid can be used to encode the ``facial weak order'' of simplicial hyperplane arrangements~\cite{dermenjian_facialweak_2018}.
Here, we use covectors and the \emph{intersection} operation as our main tools to elevate Reading's characterization of congruence uniformity to the level of oriented matroids (see Theorem~\ref{thm:covector_forcing} and Corollary~\ref{cor:covector_forcing}).
Namely, we introduce \emph{shard covectors}---which are covectors with some ``$*$'' entries---and show they are in bijection with shards (see Theorem~\ref{thm:shard_bijection}).

This approach led to the following results.
If a set of normals admits a \emph{root poset} with respect to a base region, then its lattice of regions is congruence normal (see Theorem~\ref{thm:constructible}).
In particular, the posets of regions of hyperplane arrangements coming from finite Weyl groupoids are always congruence normal and congruence uniform (see Corollary~\ref{cor:cn_weyl}).
This result provides a new proof that finite Coxeter arrangements are obtainable through a finite sequence of interval doublings (i.e. congruence uniform) \cite[Theorem~6]{caspard_cayley_2004}.
We further classify the known rank-$3$ simplicial arrangements according to whether their posets of regions are always or sometimes or never congruence normal (see Table~\ref{tab:cn_part}).
The approach through covectors gives a way to determine congruence normality of posets of regions without the data of the poset or resorting to polyhedral objects (i.e.\ shards).
Notably, this classification could not have been carried out through the computation of the posets of regions due to their large size.
Hence, this framework provides an oriented matroid approach to study congruence normality and uniformity for large posets of regions.
As an interesting outcome of this classification, five arrangements have exceptional behavior.
Two of the five arrangements are always congruence normal: the non-crystallographic arrangement corresponding to the Coxeter group $H_3$ and its point-line dual arrangement which has 31 hyperplanes.
The three other arrangements are never congruence normal: they have yet to show any connection to other known structures.
Furthermore, we provide instructive examples which give deeper insight into congruence uniformity for posets of regions.
We verified that within supersolvable simplicial arrangements (by \cite[Theorem~1.2]{cuntz_supersolvable_2019} these are the arrangements in 2 of the 3 infinite families) only four are always congruence normal and all others are only sometimes congruence normal, see Theorems~\ref{thm:cn_f2} and~\ref{thm:cn_f3}.
The algorithms used to carry out the verifications and the data to construct known simplicial hyperplane arrangements are available as a \texttt{Sage}-package \cite{cn_hyperarr}.

The article is structured as follows.
In Section~\ref{sec:prelim}, we present the necessary background notions on lattice congruences, posets of regions, congruence normality and uniformity and the theory of shards.
In Section~\ref{sec:cong_norm}, we recast shards and the forcing relation using covectors.
In Section~\ref{sec:simplicial}, we present the result of the application of the approach of Section~\ref{sec:cong_norm} to the known rank-$3$ simplicial hyperplane arrangements.
In Section~\ref{sec:invariants}, we present combinatorial and geometric invariants of the known rank-3 simplicial hyperplane arrangements with up to 37 hyperplanes. 
In Appendix~\ref{app:lists}, we give normals to realize each of these arrangements.
Finally, in Appendix~\ref{app:wd}, we give a wiring diagram description for these arrangements.
\medskip

{\bf Acknowledgements.}
The authors would like to express their gratitude to Vincent Pilaud and Julian Ritter for important discussions leading to the results in this paper.

\section{Preliminaries}
\label{sec:prelim}
We use the following notation: $\N=\{0,1,2,\dots\}$, $d,m\in\N\setminus\!\{0\}$, and ${[m]:=\{1,2,\dots,m\}}$.
We use bold faced $\n,\p,\x,$ etc. to denote vectors in the real Euclidean space $\R^{d}$ equipped with the usual dot product \mbox{$\R^d \times \R^d\rightarrow~\R$}.
Let $\P$ denote a finite, ordered set of vectors.
The linear span of $\P$ is denoted $\spa(\P)$, its affine hull by $\aff(\P)$, and its convex hull by $\conv(\P)$.
To ease reading, we often abuse notation and write for instance $\spa(\x_1,\x_2)$ instead of $\spa(\{\x_1,\x_2\})$.
The orthogonal complement of a linear subspace $A\subseteq \R^d$ is denoted~$A^\top$.
The relative interior of a subset~$\P$ of $\R^d$ is denoted by $\inter(\P)$.

In Section~\ref{ssec:latt_cong}, we review the notion of a lattice congruence.
In Section~\ref{ssec:poset_region}, we define hyperplane arrangements and posets of regions.
In Sections~\ref{ssec:cong_normal} and~\ref{ssec:cong_unif}, we discuss the notions of congruence normality and uniformity.
Finally, in Section~\ref{ssec:cong_shards}, we describe Reading's characterization of congruence uniformity for tight hyperplane arrangements using \emph{shards}.
The material presented in this section is mostly based on material treated in the book chapter \cite[Chapter~9]{reading_lattice_2016}. 

\subsection{Lattice congruences}
\label{ssec:latt_cong}

Let $L=(P;\wedge,\vee)$ be a finite lattice, where $P$ is a poset $(P,\leq)$.
An element $j \in L$ is \defn{join-irreducible} if $j$ covers a unique element~$j_\bullet\in L$. 
Similarly, an element $m\in L$ is \defn{meet-irreducible} if $m$ is covered by a unique element $m^\bullet\in L$.
We denote the subposet of join-irreducible elements of a lattice~$L$ by~$L_{\vee}$ and the subposet of meet-irreducible elements by $L_{\wedge}$.
An \defn{order ideal} of a poset $P$ is a subposet ${Q \subseteq P}$ that satisfies $x \in Q$ and $y \leq x  \Rightarrow y \in Q$.
The order ideals of a poset $P$ can be ordered by containment to get the \defn{poset of order ideals} denoted $\mathcal{O}(P)$.
When $L$ is self-dual, join- and meet-irreducible elements are canonically in bijection.
The dual map therefore allows one to refine statements involving~$L$ and its irreducible elements.
Join-irreducible elements (and dually meet-irreducible elements) and posets of order ideals are very useful to understand finite distributive lattices.

\begin{lemma}[{\cite[Theorem~17.3]{birkhoff_combination_1933}}]
\label{lem:fund_thm_dist}
Let $L$ be a lattice, $L_{\vee}$ be its subposet of join-irreducible elements, and $\mathcal{O}(L_{\vee})$ be the poset of order ideals of $L_{\vee}$. 
If $L$ is finite and distributive, then $L$ is isomorphic to $\mathcal{O}(L_{\vee})$.
\end{lemma}

Recall that cosets of a normal subgroup $N \trianglelefteq G$ determine a \emph{congruence} relation, and lead to a quotient group $G/N$, which is the image of the map sending an element to its coset.
Analogously, in lattice theory, intervals play the roll of cosets, and under certain conditions, they form to a \emph{quotient lattice}.
In this case, the equivalence relation is called a lattice congruence. 
For a thorough discussion on congruences and quotient lattices, we refer the reader to \cite[Chapter~9-5 and 9-10]{reading_lattice_2016} and the references therein.

\begin{definition}[Lattice congruence, see e.g. {\cite[Proposition~9-5.2]{reading_lattice_2016}}]
\label{def:lat_cong}
An equivalence relation on the elements of a lattice~$L$ is a \defn{lattice congruence} if the following three conditions are satisfied:
\begin{enumerate}
	\item Every equivalence class is an interval.
	\item The map $\pi_\downarrow$ sending each element to the minimal 
		element in its equivalence class is order-preserving.
	\item The map $\pi_\uparrow$ sending each element to the maximal
		element in its equivalence class is order-preserving.
\end{enumerate}
Given a lattice congruence, the images of $\pi_\downarrow$ and $\pi_\uparrow$ are sublattices, i.e.\ the join and meet operations are preserved on the equivalence classes, and they are referred to as \defn{quotient lattices}.
\end{definition}

Lattice congruences of a lattice can be numerous and the relations between them may be challenging to describe.
In spite of that, the set of lattice congruences on a lattice $L$ may be partially ordered by refinement.
The equivalence relation with singleton classes is the smallest lattice congruence and its associated quotient lattice is the lattice itself.
Furthermore, the equivalence relation with a unique class is the coarsest lattice congruence whose associated quotient lattice has exactly one element.
It turns out that under this partial order, the set of lattice congruences forms a distributive lattice which is called the \defn{lattice of congruences} and is denoted by $\Con(L)$~\cite{funayama_distributivity_1942}. 
The lattices of congruences we consider here are finite and therefore complete.
Consequently, given any set of relations, there is a smallest lattice congruence which contains these relations \cite[Proposition 9-5.13]{reading_lattice_2016}.
This makes it possible to define two important congruences related to join- and meet-irreducible elements.
Consider a join-irreducible element $j\in L_{\vee}$, then there is a smallest lattice congruence $\con_\vee(j)$ such that $j$ and~$j_\bullet$ are equivalent.
Similarly, for a meet-irreducible element $m$, there is a smallest lattice congruence $\con_{\wedge}(m)$ such that $m$ and the unique element $m^\bullet$ that covers it are equivalent.
In this case, we say that the congruence $\con_\vee$ \defn{contracts} $j$, and that $\con_\wedge$ contracts $m$.
As $\Con(L)$ is finite and distributive, we may use Lemma~\ref{lem:fund_thm_dist} to obtain that $\Con(L)$ is isomorphic to $\mathcal{O}(\Con(L)_\vee)$.
That is to say that a congruence is determined by an order ideal of join-irreducible congruences, i.e., by the join-irreducibles it contracts~\cite[Corollary 9-5.15]{reading_lattice_2016}.

\begin{definition}
Let $\con_\vee: L_\vee \rightarrow \Con(L)$ be the map that sends a join-irreducible element $j \in L_{\vee}$ to the smallest lattice congruence in $\Con(L)$ such that $j \equiv j_\bullet$.
Dually, the map $\con_{\wedge}$ is similarly defined for meet-irreducible elements.
\end{definition}

The image of the map $\con_\vee$ is $\Con(L)_{\vee}$, i.e., the congruence $\con_\vee(j)$ is join-irreducible in $\Con(L)$ and for every join-irreducible congruence $\alpha$ in $\Con(L)$, there exists a join-irreducible $j\in L_\vee$ such that $\con_\vee(j)=\alpha$ \cite[Proposition 9-5.14]{reading_lattice_2016}.
It may happen that two distinct join-irreducibles give rise to the same congruence, i.e.\ that $\con_\vee$ is not injective, leading to an equivalence relation on join-irreducible elements in $L_\vee$.
Through the map $\con_\vee$, these equivalence classes of join-irreducible elements in $L$ are in bijection with join-irreducible congruences of $L$.

\subsection{Poset of regions of a real hyperplane arrangement}
\label{ssec:poset_region}

A \defn{(real) hyperplane}~$H$ is a co\-di\-men\-sion-1 affine subspace in $\R^d$:
\[
H := \{\x\in\R^d~:~\n\cdot\x=a \text{ for some } \n \in \R^d \text{ and } a\in\R\}.
\]
The vector $\n$ is called the \defn{normal} of $H$.
A \defn{finite hyperplane arrangement} $\A$ is a finite non-empty set of $m$ hyperplanes. 
If $a=0$ for all hyperplanes in $\A$, then the hyperplane arrangement is called \defn{central}.
In this case, the hyperplanes are completely determined by their normals.
We denote the hyperplanes in $\A$ by $H_1,\dots,H_m$ and often reuse their indices to refer to objects canonically related to them.
The \defn{rank} of $\A$ is the dimension of the linear span of the normal vectors of the hyperplanes in $\A$.
The complement of the arrangement in the ambient space $(\R^d \setminus \bigcup_{i\in[m]}H_i)$ is disconnected, and the closures of the connected components are the \defn{regions} of the arrangement.
The set of regions of $\A$ is denoted by $\Reg(\A)$.
A region is called \defn{simplicial} if the normal vectors of its facet-defining hyperplanes are linearly independent. 
A hyperplane arrangement is \defn{simplicial} if every region in its complement is simplicial. 
To proceed further, a \defn{base region} $B$ of $\A$ is chosen.
For each hyperplane $H_i \in \A$, we fix a normal vector~$\n_i\in\R^d$ such that $\n_i\cdot \x <0$, for all $\x\in B$.
Given a region~$R$ of $\A$, the \defn{separating set} $\Sep_B(R)$ of $R$ is the set of hyperplanes $H_i\in\A$ such that $\n_i\cdot\x>0,$ for all $\x\in R$. 
The separating set of a region is the set of hyperplanes that separate it from the base region $B$.

\begin{definition}[{Poset of regions, $\PosBA$}]
Let $\A$ be a hyperplane arrangement with base region $B$.
The \defn{poset of regions} $\PosBA$ of $\A$ with base region $B$ is the partially ordered set $(\Reg(\A),\leq)$ such that
\[
R_1 \leq R_2 \text{ if and only if } \Sep_B(R_1) \subseteq \Sep_B(R_2),
\]
for all $R_1,R_2\in \Reg(\A)$.
\end{definition}

An \defn{upper facet} of a region $R\in\Reg(\A)$ is a facet of $R$ which corresponds to a cover relation of $R$ in~$\PosBA$.
A hyperplane arrangement is \defn{tight} with respect to $B$ when the upper facets of every region intersect pairwise along a codimension-$2$ face, i.e.\ they are neighbors in the facet-adjacency graph.
When a hyperplane arrangement $\A$ is tight with respect to every base region, we say that $\A$ is tight.
For convenience, when a hyperplane arrangement is tight, we also call the corresponding posets of regions tight.
The usual definition of tightness also requires the dual statement to hold. 
As poset of regions are self-dual, we have restricted the statement to upper facets.
The following lemma is a refinement of \cite[Theorem~3.4]{bjoerner_hyperplane_1990}.

\begin{lemma}
\label{lem:bez}
Let $\A$ be a finite, central hyperplane arrangement with base region~$B$.
\begin{enumerate}
\item If $\A$ is tight with respect to $B$, then $\PosBA$ is a lattice. \cite[Theorem~9-3.2]{reading_lattice_2016}
\item If $\A$ is simplicial, then $\A$ is tight. \cite[Proposition~9-3.3]{reading_lattice_2016}
\end{enumerate}
\end{lemma}

Reading developed an approach to study congruences of lattices of regions that is thoroughly described in \cite[Chapter~9]{reading_lattice_2016}.
In particular, for posets of regions, tightness is equivalent to \emph{semidistributivity} \cite[Theorem~9-3.8]{reading_lattice_2016} (see Section~\ref{ssec:cong_unif} for the definition of semidistributivity).
Furthermore, in order to describe the interplay between join-irreducible elements, the combinatorial notion of ``polygonality'' of a lattice is used; in the case of posets of regions, this notion is equivalent to the notion of tightness \cite[Theorem~9-6.10]{reading_lattice_2016}. 
Using the polygonality property, it is possible to describe which join-irreducibles \emph{force} other ones to be contracted. 
This forcing relation can then be read off from the hyperplane arrangement using pieces of hyperplanes called \emph{shards} (see Definition~\ref{def:shards} in Section~\ref{ssec:cong_shards}).
The interest in the notion of tightness lies in the fact that being tight and having \emph{acyclicity} on shards characterizes congruence uniformity, see Theorem~\ref{thm:charac_uniform} in Section~\ref{ssec:cong_shards}.

Throughout this article, we restrict our study to finite, central, and tight hyperplane arrangements, so that the posets of regions are guaranteed to be complete lattices regardless of the choice of base regions. 
We refer the reader to \cite[Chapter~9-3, 9-6]{reading_lattice_2016} for further details on tightness and polygonality.

\subsection{Congruence normality}
\label{ssec:cong_normal}

\begin{definition}[{Congruence normality, \cite[Section~1, p.400]{day_congruence_1994}}]
Let $L$ be a lattice, $L_\vee\subseteq L$ be the subposet of join-irreducible elements of $L$, and $L_\wedge$ be the subposet of meet-irreducible elements of $L$.
The lattice $L$ is \defn{congruence normal} if
\[
j\leq m\quad \text{ implies }\quad \con_\vee(j) \neq \con_\wedge(m),
\]
for all $j\in L_\vee$, and $m\in L_\wedge$.
A hyperplane arrangement is called \defn{congruence normal} if its lattices of regions are congruence normal for every choice of base region. 
\end{definition}

Equivalently, finite congruence normal lattices are exactly the lattices obtained from a one-element lattice by a sequence of \emph{doublings} of \emph{convex sets} \cite[Section~3]{day_congruence_1994}, see also \cite[Theorem~3-2.39]{adaricheva_classes_2016} and \cite{geyer_generalized_1994}.
The following example illustrates a local condition showing how a lattice may fail to be congruence normal.

\begin{example}
\label{ex:L3_non_cn}
Consider the lattice $L_3$ with the Hasse diagram illustrated in Figure~\ref{fig:L3_non_cn}.
The element~$c$ is join-irreducible, and the smallest congruence $\con_\vee(c)$ such that $b\equiv c$ is illustrated on the right-hand side.

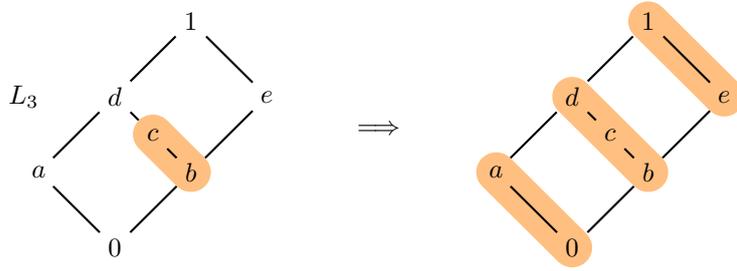
\begin{figure}[H]
\begin{center}
\begin{tabular}{c@{\hspace{1cm}$\Longrightarrow$\hspace{1cm}}c}
\begin{tikzpicture}%
[baseline=(c),
arete/.style={thick},
vertex/.style={inner sep=3pt},
eq_class/.style={line join=round,line width=15pt,color=orange,draw opacity=0.5,line cap=round}]

\def\width{1}
\def\height{1}

\coordinate (0) at (0,0);
\coordinate (a) at (-\width,\height);
\coordinate (b) at (\width,\height);
\coordinate (c) at  (\width/2,1.5*\height);
\coordinate (d) at  (0,2*\height);
\coordinate (e) at  (2*\width,2*\height);
\coordinate (1) at  (\width,3*\height);

\draw[eq_class] (b) -- (c);

\node (L3) at (-1.2*\width,2*\height) {$L_3$};

\node[vertex] (v0) at (0) {$0$};
\node[vertex] (va) at (a) {$a$};
\node[vertex] (vb) at (b) {$b$};
\node[vertex] (vc) at (c) {$c$};
\node[vertex] (vd) at (d) {$d$};
\node[vertex] (ve) at (e) {$e$};
\node[vertex] (v1) at (1) {$1$};

\draw[arete] (vd) -- (v1) -- (ve) -- (vb) -- (vc) -- (vd) -- (va) -- (v0) -- (vb);

\end{tikzpicture}
&
\begin{tikzpicture}%
[baseline=(c),
arete/.style={thick},
vertex/.style={inner sep=3pt},
eq_class/.style={line join=round,line width=15pt,color=orange,draw opacity=0.5,line cap=round}]

\def\width{1}
\def\height{1}

\coordinate (0) at (0,0);
\coordinate (a) at (-\width,\height);
\coordinate (b) at (\width,\height);
\coordinate (c) at  (\width/2,1.5*\height);
\coordinate (d) at  (0,2*\height);
\coordinate (e) at  (2*\width,2*\height);
\coordinate (1) at  (\width,3*\height);

\draw[eq_class] (0) -- (a);
\draw[eq_class] (b) -- (d);
\draw[eq_class] (e) -- (1);

\node[vertex] (v0) at (0) {$0$};
\node[vertex] (va) at (a) {$a$};
\node[vertex] (vb) at (b) {$b$};
\node[vertex] (vc) at (c) {$c$};
\node[vertex] (vd) at (d) {$d$};
\node[vertex] (ve) at (e) {$e$};
\node[vertex] (v1) at (1) {$1$};

\draw[arete] (vd) -- (v1) -- (ve) -- (vb) -- (vc) -- (vd) -- (va) -- (v0) -- (vb);

\end{tikzpicture}
\end{tabular}
\end{center}
\caption{The Hasse diagram of the lattice $L_3$ which is not congruence normal and the equivalence classes of $\con_\vee(c)=\con_\wedge(c)$.}
\label{fig:L3_non_cn}
\end{figure}

\noindent
Following Definition~\ref{def:lat_cong}, setting $b\equiv c$ forces the lattice to project onto a three-element chain.
By order-reversing symmetry, the smallest congruence such that $c\equiv d$ is the same as the smallest congruence such that $b\equiv c$.
Since $c$ is also meet-irreducible, we get $\con_\wedge(c)=\con_\vee(c)$ and since $c\leq c$, this lattice is not congruence normal.
\end{example}

This example complements Reading's example of forcing of polygons nicely, see e.g.~\cite[Example~9-6.6]{reading_lattice_2016} and the exercise on congruence normality of polygonal lattices \cite[Exercice 9.55]{reading_lattice_2016}.
The intervals $[0,d]$ and $[b,1]$ intersect on more than one cover and removing~$c$ from~$L_3$ makes it congruence normal.
Unfortunately, such local obstructions may not be used on lattices of regions of a hyperplane arrangement.
The corresponding Hasse diagrams are isomorphic to the 1-skeleta of the associated zonotopes, and two polygons as in the example may not intersect along more than one cover relation for convexity reasons.
As we shall see in Example~\ref{ex:A10_60}, there are non-congruence normal lattices of regions.

\subsection{Congruence uniformity}
\label{ssec:cong_unif}

A lattice is \defn{join-semidistributive} if for $x,y,z \in L$, 
\[
x \vee y = x \vee z \text{ implies } x \vee (y \wedge z) = x \vee y.
\]
It is \defn{meet-semidistributive} if
\[
x \wedge y = x \wedge z \text{ implies }x \wedge (y \vee z) = x \wedge y.
\]
A lattice that is both join-semidistributive and meet-semidistributive is called \defn{semidistributive}.

\begin{definition}[Congruence uniformity, {\cite[Definition~4.1]{day_characterizations_1979}}]
Let $L$ be a finite lattice.
If the maps $\con_\vee$ and $\con_\wedge$ are injective, then $L$ is called \defn{congruence uniform}.
\end{definition}

Congruence uniformity describes the lattice of congruences of the involved lattice through the map $\con_\vee$.
If $L$ is a finite congruence uniform lattice, then the map $\con_\vee$ gives a order-preserving bijection between $L_\vee$ and $\Con(L)_\vee$.
Lemma~\ref{lem:fund_thm_dist} then permits to study the whole of $\Con(L)$.
Congruence uniformity is a stronger condition than congruence normality in that it should be obtained from a one-element lattice by a sequence of \emph{doublings of intervals} \cite[Theorem~5.1]{day_characterizations_1979}.

\begin{theorem}[{\cite[Section~2]{day_congruence_1994}}]
A finite lattice is congruence uniform if and only if it is both congruence 
normal and semidistributive.
\end{theorem}

\begin{corollary}
\label{cor:norm_unif}
A tight poset of regions $\PosBA$ is congruence normal if and only if it is congruence uniform.
\end{corollary}

\begin{proof}
By Lemma~\ref{lem:bez}, the poset of regions $\PosBA$ of a $\A$ is a finite lattice, independent of the choice of base region $B$.
Furthermore, $\A$ is tight with respect to $B$ if and only if $\PosBA$ is semidistributive \cite[Theorem~9-3.8]{reading_lattice_2016}.
\end{proof}

\begin{remark}
\hfill
\begin{enumerate}
\item Since lattices of regions are self-dual, it suffices to verify the injectivity of $\con_\vee$ to determine whether they are congruence uniform. 
\item Semidistributivity can be described using sublattice avoidance~\cite[Theorem~3-1.4]{adaricheva_classes_2016}. 
The six sublattices obstructing semidistributivity are illustrated in Figure~\ref{fig:L3_non_cn} and~\ref{fig:forbidden_sublattices}.
Four out of the six non-semidistributive lattices are not congruence normal ($L_3$, $L_4$, $L_5$, and $M_3$) and share the property that two polygons share more than 1 cover.
Nevertheless, semidistributivity is neither necessary nor sufficient to obtain congruence normality: $L_1$ and $L_2$ are congruence normal but not semidistributive and Example~\ref{ex:A10_60} gives a poset of regions which is semidistributive but not congruence normal.

\begin{figure}[H]
\begin{tikzpicture}
\def\width{6}
\def\height{0}

\node at (-\width/2,\height) {\begin{tikzpicture}%
[baseline=(c),
arete/.style={thick},
point/.style={circle,fill=black,inner sep=1.5pt}]

\def\width{0.8}
\def\height{0.8}

\coordinate (0) at (      0,         0);
\coordinate (a) at (-\width,   \height);
\coordinate (b) at ( \width,   \height);
\coordinate (c) at (-\width, 2*\height);
\coordinate (d) at (      0, 2*\height);
\coordinate (e) at ( \width, 2*\height);
\coordinate (1) at (      0, 3*\height);

\node at (0,-1/2*\height) {$L_2$};

\node[point] (v0) at (0) {};
\node[point] (va) at (a) {};
\node[point] (vb) at (b) {};
\node[point] (vc) at (c) {};
\node[point] (vd) at (d) {};
\node[point] (ve) at (e) {};
\node[point] (v1) at (1) {};

\draw[arete] (v1) -- (vc) -- (va) -- (v0) -- (vb) -- (ve) -- (v1) -- (vd) -- (va) ;
\draw[arete] (vb) -- (vd);
\end{tikzpicture}};
\node at ( \width/2,\height) {\begin{tikzpicture}%
[baseline=(c),
arete/.style={thick},
point/.style={circle,fill=black,inner sep=1.5pt}]

\def\width{1/2}
\def\height{1/2}

\coordinate (0) at (        0,        0);
\coordinate (a) at (-2*\width,2*\height);
\coordinate (b) at ( 2*\width,2*\height);
\coordinate (c) at (        0,2*\height);
\coordinate (d) at (  -\width,  \height);
\coordinate (1) at (        0,4*\height);

\node at (0,-1/2) {$L_5$};

\node[point] (v0) at (0) {};
\node[point] (va) at (a) {};
\node[point] (vb) at (b) {};
\node[point] (vc) at (c) {};
\node[point] (vd) at (d) {};
\node[point] (v1) at (1) {};

\draw[arete] (v0) -- (va) -- (v1)  -- (vb) -- (v0);
\draw[arete] (vd) -- (vc) -- (v1);
\end{tikzpicture}};
\node at (  -\width,      0) {\begin{tikzpicture}%
[baseline=(c),
arete/.style={thick},
point/.style={circle,fill=black,inner sep=1.5pt}]

\def\width{0.8}
\def\height{0.8}

\coordinate (0) at (      0,         0);
\coordinate (a) at (-\width,   \height);
\coordinate (b) at (      0,   \height);
\coordinate (c) at ( \width,   \height);
\coordinate (d) at (-\width, 2*\height);
\coordinate (e) at ( \width, 2*\height);
\coordinate (1) at (      0, 3*\height);

\node at (0,-1/2*\height) {$L_1$};

\node[point] (v0) at (0) {};
\node[point] (va) at (a) {};
\node[point] (vb) at (b) {};
\node[point] (vc) at (c) {};
\node[point] (vd) at (d) {};
\node[point] (ve) at (e) {};
\node[point] (v1) at (1) {};

\draw[arete] (v0) -- (va) -- (vd) -- (v1) -- (ve) -- (vc) -- (v0) -- (vb) -- (vd) ;
\draw[arete] (vb) -- (ve);
\end{tikzpicture}};
\node at (        0,      0) {\begin{tikzpicture}%
[baseline=(c),
arete/.style={thick},
point/.style={circle,fill=black,inner sep=1.5pt}]

\def\width{1/2}
\def\height{1/2}

\coordinate (0) at (        0,        0);
\coordinate (a) at (-2*\width,2*\height);
\coordinate (b) at ( 2*\width,2*\height);
\coordinate (c) at (        0,2*\height);
\coordinate (d) at (  -\width,3*\height);
\coordinate (1) at (        0,4*\height);

\node at (0,-1/2) {$L_4$};

\node[point] (v0) at (0) {};
\node[point] (va) at (a) {};
\node[point] (vb) at (b) {};
\node[point] (vc) at (c) {};
\node[point] (vd) at (d) {};
\node[point] (v1) at (1) {};

\draw[arete] (v0) -- (va) -- (vd) -- (v1)  -- (vb) -- (v0) -- (vc) -- (vd);
\end{tikzpicture}};
\node at (   \width,      0) {\begin{tikzpicture}%
[baseline=(c),
arete/.style={thick},
point/.style={circle,fill=black,inner sep=1.5pt}]

\def\width{1/2}
\def\height{1/2}

\coordinate (0) at (        0,        0);
\coordinate (a) at (-2*\width,2*\height);
\coordinate (b) at ( 2*\width,2*\height);
\coordinate (c) at (        0,2*\height);
\coordinate (1) at (        0,4*\height);

\node at (0,-1/2) {$M_3$};

\node[point] (v0) at (0) {};
\node[point] (va) at (a) {};
\node[point] (vb) at (b) {};
\node[point] (vc) at (c) {};
\node[point] (v1) at (1) {};

\draw[arete] (v0) -- (va) -- (v1)  -- (vb) -- (v0) -- (vc) -- (v1);
\end{tikzpicture}};
\end{tikzpicture}
\caption{Five of the six sublattices that obstruct semidistributivity, the sixth is $L_3$ illustrated in Figure~\ref{fig:L3_non_cn}}
\label{fig:forbidden_sublattices}
\end{figure}
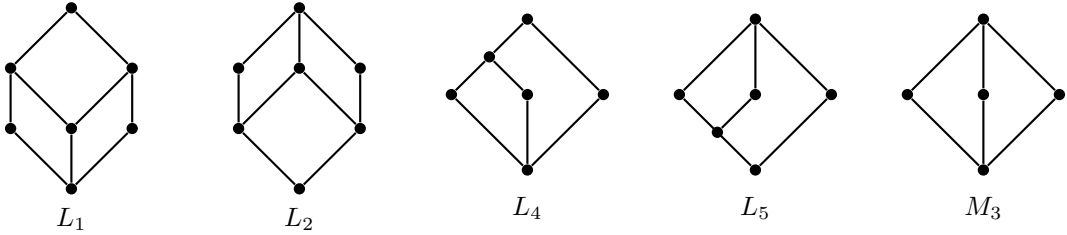
\item Considering two polygons in a polygonal lattice, and verifying congruence normality as in Example~\ref{ex:L3_non_cn}, one realizes that $M_3$, $L_3$, $L_4$, and $L_5$ should be avoided.
For poset of regions, this comes as no surprise as polygonality, tightness and semidistributivity are equivalent \cite[Theorem~9-3.8 and~9-6.10]{reading_lattice_2016}. 
In general, what is the relation between polygonal and semidistributive lattices?
\end{enumerate}
\end{remark}

\begin{example}[{\cite[Figure~5]{reading_lattice_2003} and \cite[Exercise~9.69]{reading_lattice_2016}}]
\label{ex:A10_60}
Figure \ref{fig:A10_60} illustrates the stereographic projection of the simplicial hyperplane arrangement $\A(10,60)_3$ in $\R^3$ with 10 hyperplanes through the intersection of 5 hyperplanes which are mapped to lines.
This arrangement is $\mathcal{A}(10,1)$ in Gr\"unbaum's list \cite[p.2-3]{grunbaum_2009}, see Section~\ref{sec:simplicial}. 

\begin{figure}[H]
\begin{center}
\includegraphics[width=3.5in]{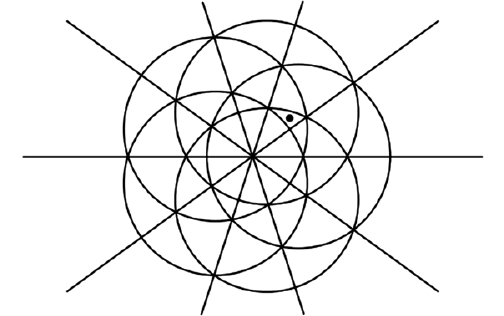}
\end{center}
\caption{The simplicial hyperplane arrangement $\A(10,60)_3=\F_2(10)$ whose lattice of regions with the marked base region is not congruence normal}
\label{fig:A10_60}
\end{figure}

\noindent
The lattice of regions with respect to the base region marked by a black dot is thus semidistributive.
In Example~\ref{ex:cycle_in_shards}, we use shards to demonstrate that this arrangement is not congruence normal, hence not uniform by Corollary~\ref{cor:norm_unif}.
It is the smallest known simplicial hyperplane arrangement of rank three with that property.
\end{example}

Examples~\ref{ex:L3_non_cn} and~\ref{ex:A10_60} illustrate failures to be congruence normal.
Example~\ref{ex:A10_60} is particularly interesting in that it does not fail to be congruence normal because of forbidden sublattices blocking semidistributivity.

\subsection{Congruence normality of simplicial hyperplane arrangements through shards}
\label{ssec:cong_shards}

Reading characterized congruence uniformity of posets of regions via two conditions, the first one is tightness and the second is phrased using pieces of hyperplanes called \emph{shards}.
When the arrangement is central, these pieces are polyhedral cones defined through certain subarrangements.

\begin{definition}[{Rank-2 subarrangements and their basic hyperplanes, see~\cite[Definition~9-7.1]{reading_lattice_2016}}]
Let $\A$ be a hyperplane arrangement with base region $B$, and let $1\leq i<j\leq m$.
The set 
\[
\A|_{i,j}:=\{H\in \A~:~H\supset (H_i\cap H_j)\}
\]
is called a \defn{rank-2 subarrangement} of $\A$.
The two facet-defining hyperplanes of the region of~$\A|_{i,j}$ that contains $B$ are called the \defn{basic} hyperplanes of $\A|_{i,j}$.
\end{definition}

\begin{definition}[{Shards, see~\cite[Definition~9-7.2]{reading_lattice_2016}}]
\label{def:shards}
Let $H_i\in\A$ and set
\[
\pre(H_i):=\{H_k\in\A~:~H_k \text{ is basic in }\A|_{i,k}\quad \text{ and }\quad H_i \text{ is not basic in }\A|_{i,k}\}.
\]
The restriction of $\pre(H_i)$ to the hyperplane $H_i$ breaks $H_i$ into closed regions called \defn{shards}.
We denote shards by capital Greek letters such as $\Sigma,\Theta,\Upsilon,$ etc.
The hyperplane of $\A$ that contains a shard~$\Sigma$ is denoted by $H_\Sigma$.
We write $\Sigma^i$ to indicate that it is contained in $H_i$.
Hyperplanes in $\pre(H_i)$ are said to \defn{cut} the hyperplane $H_i$.
\end{definition}

\begin{example}[Example~\ref{ex:A10_60} continued]
\label{ex:shards}
Figure \ref{fig:shards} illustrates the 29 shards obtained from the base region marked with a dot.

\begin{figure}[H]
\begin{center}
\begin{tikzpicture}
\node at (0,0) {\includegraphics[width=4in]{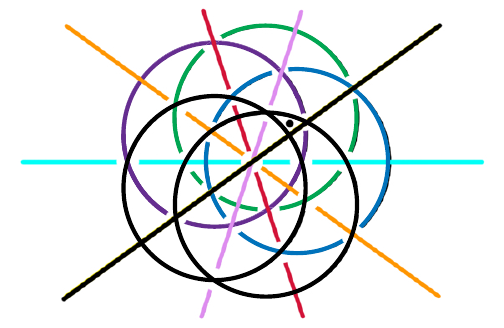}};
% labels
\node at (-2.9,-.8) {$H_1$};
\node at (1.9,-2.3) {$H_2$};
\node at (-.7,-3) {$H_3$};
\node at (-3.4,-2.9) {$H_4$};
\node at (-2.9,1.1) {$H_5$};
\node at (1.3,-3) {$H_6$};
\node at (4.6,.3) {$H_7$};
\node at (3.8,-2.3) {$H_8$};
\node at (2.9,-.8) {$H_9$};
\node at (1.8,2.5) {$H_{10}$};
\node at (.7,3) {$\Sigma$};
\node at (-1.3,-3) {$\Sigma$};
\node at (-1.3,3) {$\Sigma'$};
\node at (.6,-3) {$\Sigma''$};
\end{tikzpicture}
\end{center}
\caption{The shards of the simplicial hyperplane arrangement $\A(10,60)_3=\F_2(10)$ with respect to the dotted region}
\label{fig:shards}
\end{figure}
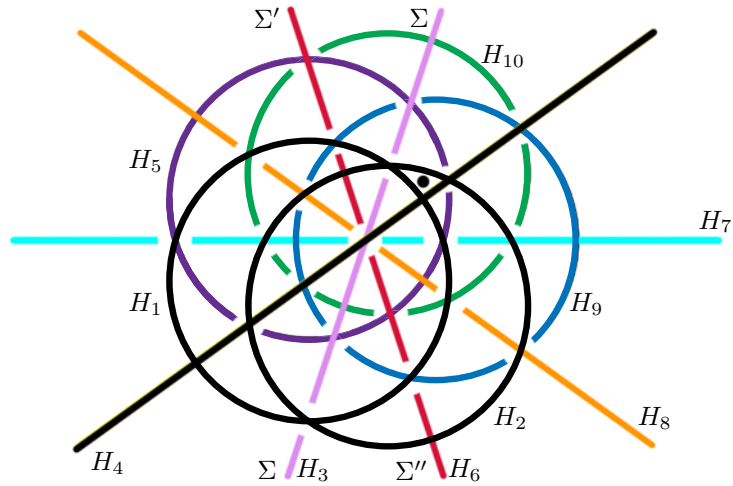
\noindent
On the one hand, due to the particular choice of projection, it is necessary to distinguish whether two unbounded straight line-segments lying on a common line form 1 or 2 shards.
For example, the unbounded line-segments on the line labeled $H_3$ form one shard $\Sigma$, and the unbounded line-segments on the line labeled $H_6$ form 2 distinct shards $\Sigma'$ and $\Sigma''$.
On the other hand, it is possible to solve this by changing the projection to obtain only circles, though simultaneously losing symmetry.
\end{example}

The following directed graph records the cutting relation among hyperplanes.

\begin{definition}[{Directed graph $\H_B(\A)$ \cite[Section~3]{reading_lattice_2004}}]
\label{def:directed_graph}
Let $\H_B(\A)$ be the directed graph whose vertices are the hyperplanes of the 
arrangement $\A$, and whose oriented edges are such that
\begin{center}
$H_i \rightarrow H_j$\qquad if and only if\qquad $H_i \in \pre(H_j)$.
\end{center}
\end{definition}

The following directed graph keeps track of the cutting relation along with the ``geometric proximity'' between shards.

\begin{definition}[Shard digraph, see {\cite[Section~3]{reading_lattice_2004}\cite[Definition~9.7.16]{reading_lattice_2016}}]
\label{def:shard_dig}
Let $\Sh_B(\A)$ be the directed graph on the shards of $\A$ 
such that
\[
\Sigma^i \rightarrow \Sigma^j \quad \text{if and only if } \quad 
\begin{array}{l}
\bullet\hspace{5pt} H_{\Sigma^i} \rightarrow H_{\Sigma^j}\text{ in } \H_B(\A) \text{ and} \\[0.25em]
\bullet\hspace{5pt} \Sigma^i \cap \Sigma^j \text{ has dimension $d-2$.}
\end{array}
\]
\end{definition}

The following theorem gives a characterization of congruence uniformity in terms of the directed graph on shards.

\begin{theorem}[{\cite[Corollary~9-7.22]{reading_lattice_2016}}]
\label{thm:charac_uniform}
Let $\A$ be a hyperplane arrangement with a base region~$B$.
The poset of regions $\PosBA$ is a congruence uniform lattice if and only if $\A$ is tight with respect to~$B$ and $\Sh_B(\A)$ is acyclic.
In this case, $\Sh_B(\A)$ is isomorphic to the Hasse diagram of $\Con(\PosBA)_\vee$. 
\end{theorem}

By Corollary~\ref{cor:norm_unif}, the theorem implies that acyclicity of the directed graph on shards~$\Sh_B(\A)$ characterizes the normality and uniformity of tight posets of regions $\PosBA$.

\begin{example}[{Example~\ref{ex:shards} continued}]
\label{ex:cycle_in_shards}
Let $\Sigma^6,\Theta^{10},\Upsilon^8$, and $\Xi^9$ be the shards illustrated in Figure \ref{fig:cycle_in_shards}.
The directed graph on shards contains the cycle $\Sigma^6 \rightarrow \Theta^{10} \rightarrow \Upsilon^8 \rightarrow \Xi^9 \rightarrow \Sigma^6$. 
Thus, for this choice of base region, the lattice of regions is not congruence normal.
\end{example}

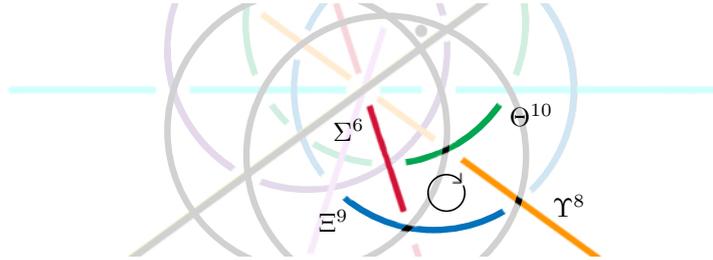
\begin{figure}[H]
\begin{center}
\begin{tikzpicture}
\node at (0,0) {\includegraphics[width=4in]{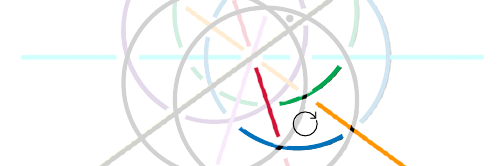}};
\node at (-.2,0) {$\Sigma^6$};
\node at (-.4,-1.2) {$\Xi^9$};
\node at (2.2,.2) {$\Theta^{10}$};
\node at (2.7, -1) {$\Upsilon^8$};
\end{tikzpicture}
\end{center}
\caption{A cycle in the shards of the simplicial hyperplane arrangement from Example~\ref{ex:A10_60}.}
\label{fig:cycle_in_shards}
\end{figure}

\section{Congruence normality through restricted covectors}
\label{sec:cong_norm}

In this section, we recast shards as certain restricted covectors---which we call \emph{shard covectors}---in the point configuration dual to the arrangement~$\A$.
We then describe how to detect cycles in $\Sh_B(\A)$ using shard covectors. 
This reduces the verification of congruence normality for tight posets of regions to its simplest combinatorial expression, one that does not require the entire poset nor the usage of polyhedral objects.
Furthermore, it is possible to express an obstruction to congruence normality for tight hyperplane arrangements.

In Section~\ref{ssec:restrict_inter}, we introduce restricted covectors and the intersection operation.
In Section~\ref{ssec:aff_point}, we define affine point configurations and their lines.
In Section~\ref{ssec:shards_as_cov}, we interpret shards as covectors.
In Section~\ref{ssec:forcing_on_cov}, we translate the forcing relation on shards into the language of covectors.
Finally, in Section~\ref{ssec:obstruction} we describe examples of obstructions to congruence normality in terms of restricted covectors.

\subsection{Restricted covectors and the intersection operation}
\label{ssec:restrict_inter}

For standard references on covectors and oriented matroids, we refer the reader to the books \cite{oriented_matroids,de_loera_triangulations_2010}.

\begin{definition}[Covector and restricted covector]
Let $\P=\{\p_i\}_{i\in[m]}$ be an ordered set of vectors in $\R^d$.
A \defn{covector} on $\P$ is a vector of signs $(\cov_i)_{i\in[m]}\in \{0,+,-\}^m$ defined as
\[
\cov:=(\sign(\ce\cdot \p_i + a))_{i\in[m]},
\]
where $\ce\in\R^d$ and $a\in\R$.
Given a subset $\mathbf{U} \subseteq \P$ and a covector $\cov$ on $\P$, the \defn{restricted covector}~$\cov|_{\mathbf{U}}$ with respect to $\mathbf{U}$ is equal to $\cov$ on the entries $\{j~:~\p_j\in \mathbf{U}\}$ and contains a ``$*$'' symbol in every other entry.
\end{definition}

Intuitively, a restricted covector ``forgets'' about certain hyperplanes while keeping them encoded. 
Similarly, reversing the roles of $\ce$ and $\p_i$ above, covectors may be thought of as \emph{sign evaluations} of a certain vector $\x$ with respect to a set of vectors:

\begin{definition}[Sign evaluation of a vector]
Let $\P=\{ \p_i\}_{i\in[m]}$ be an ordered set of vectors in $\R^d$ and $\x\in\R^d$.
The \defn{sign evaluation} of $\x$ with respect to $\P$ is the covector
\[
\cov_\P(\x):=\left(\sign( \p_i \cdot \x)\right)_{i\in[m]}.
\]
\end{definition}

Inspired by the composition operation on \emph{vectors} (i.e.\ affine dependences) of oriented matroids \cite[Chapter 3]{oriented_matroids}, we define an intersection operation on restricted covectors. 

\begin{definition}[Intersection of restricted covectors]
The commutative \defn{intersection operation} $\cap$ from \newline
${\{0,+,-,*\}\times\{0,+,-,*\}}$ to $\{0,+,-,*\}$
is defined as
\begin{center}
\begin{tabular}{c@{\hspace{2cm}}c@{\hspace{2cm}}c}
$+\cap+:=+$, & $+\cap-=-\cap +:=0$, & $-\cap-:=-$,\\[0.5em]
$0\cap \varepsilon=\varepsilon\cap 0:=0$, & $*\cap \varepsilon=\varepsilon\cap *:=\varepsilon$, & \\
\end{tabular}
\end{center}
where $\varepsilon\in\{0,+,-,*\}$.
Let $\cov,\covd \in \{0,+,-,* \}^m$ be two restricted covectors, then their \defn{intersection} $\cov\cap \covd$ is the vector of signs $(\cov_i\cap \covd_i)_{i\in[m]}$.
\end{definition}

The vector of signs $(\cov_i\cap \covd_i)_{i\in[m]}$ is not necessarily a covector, though it nevertheless records the information of the sign evaluation of points in an intersection.
It is possible to interpret this intersection operation using subsets of the real numbers.
That is, if one replaces the four symbols $0,+,-,*$ respectively by the sets $\{0\},\R_{\geq0},\R_{\leq0},\R$, and consider their intersections, we get exactly the same results.
The associativity of this operation then follows easily.

\subsection{Affine point configurations and lines}
\label{ssec:aff_point}
We use duality to pass from a hyperplane arrangement $\A$ in~$\R^d$ with a base region $B$ to an acyclic point configuration~$\A_B^*$, see \cite[Section 1.2]{oriented_matroids} for more detail.
Indeed, the normals $\{\n_i \}_{i \in [m]}$ are oriented so that the linear hyperplane orthogonal to $\ve_B\in\inter(B)$ separates them from the base region $B$, i.e.\ $\ve_B\cdot \n_i<0$, for all $i\in[m]$, making the set $\{\n_i\}_{i\in[m]}$ acyclic.

\begin{definition}[Affine point configuration relative to a base region]
Let $\A$ be a hyperplane arrangement in $\R^d$, $B\in\Reg(\A)$, and $\ve_B\in\inter(B)$.
Let 
\[
\mathbb{A}_B := \{\x \in \R^d~:~\ve_B\cdot \x = -1 \},
\]
and associate the point $\p_i:=-\frac{1}{\ve_B\cdot \n_i}\cdot \n_i \in \mathbb A_B\subset \R^d$ to the normal $\n_i$.
The ordered set of vectors $\{\p_i\}_{i\in[m]}$ is the \defn{affine point configuration of $\A$ relative to the base region $B$} and is denoted~$\A_B^*$.
\end{definition}

Choosing a different normal vector $\ve_B\in\inter(B)$ yields an affine point configuration which is projectively equivalent to $\A_B^*$.
Hence, up to projective transformation, this construction does not depend on the choice of $\ve_B$.

\begin{definition}[Lines of a point configuration, $\Lines(\P)$]
Let $\P=\{\p_i\}_{i\in[m]}$ be an ordered set of vectors in~$\R^d$ .
A subset of $\P$ consisting of all the points that lie on the affine hull of two distinct points of $\P$ is called a \defn{line}.
The set of lines of $\P$ is denoted by $\Lines(\P)$.
\end{definition}

\begin{lemma}
\label{lem:lines}
Let $\A$ be a hyperplane arrangement in $\R^d$ with base region $B$, $\ell\in\Lines(\A_B^*)$, and $\p_i$ and~$\p_j$ be the two vertices of the segment $\conv(\ell)$.
\begin{enumerate}[label=\roman{enumi}),ref=\ref{lem:lines}~\roman{enumi})]
\item The lines in $\Lines(\A_B^*)$ are in bijection with the rank-2 subarrangements of $\A$.
\label{lem:lines_i}
\item The hyperplanes $H_i$ and $H_j$ are the basic hyperplanes of the rank-2 subarrangement corresponding to $\ell$.
\label{lem:lines_ii}
\end{enumerate}
\end{lemma}

\begin{proof}
\emph{i)} 
Let $\A':=\{H_i~:~i\in\mathcal{I}\}$, for some $\mathcal{I}\subseteq[m]$.
The subarrangement $\A'$ is a rank-2 subarrangement if and only if
\[
\dim\left(\bigcap_{i \in \mathcal I} H_i\right) = d-2\quad\text{and}\quad \dim\left(\bigcap_{i \in \mathcal{I}\cup\{j\}} H_i\right) < d-2, \text{ for every } j \notin \mathcal{I}. 
\]
Equivalently,
\[
\dim\left(\spa(\n_i~:~i \in \mathcal{I})\right)= 2 \quad\text{and}\quad \dim\left(\spa(\{\n_j\}\cup\{\n_i~:~i \in \mathcal{I}\})\right) > 2, \text{ for every } j \notin \mathcal I.
\]
By passing to the affine point configuration in the affine space$~\mathbb A_B$, the above statement is equivalent to $\{\p_i~:~i \in \mathcal{I}\} \in \Lines(\A_B^*)$.
Thus the map sending a rank-2 subarrangement $\A'$ to the line $\{\p_i~:~i \in \mathcal{I} \}$ is a bijection.

\emph{ii)}
Let $B|_{i,j}$ be the region of $\A|_{i,j}$ that contains $B$:
\[
B|_{i,j} = \{\x \in \R^d~:~\p_k \cdot \x \leq 0,\text{ for all } \p_k\in\ell\},
\]
by part \emph{i)}.
Let $\p_k$ be the normal of a facet $F$ of $B|_{i,j}$ and $\x$ be contained in the relative interior of $F$ so that $\p_k\cdot \x =0$.
Since $\p_i$ and $\p_j$ are the vertices of $\conv(\ell)$, we have $\p_k = \lambda_k \p_i +(1-\lambda_k)\p_j$, for some $0 \leq \lambda_k \leq 1$.
Then
\[
0 = \x \cdot \p_k = \lambda_k (\x \cdot \p_i) + (1-\lambda_k)(\x \cdot \p_j).
\]
As $\p_i\cdot\x \leq 0$ and $\p_j\cdot\x\leq 0$, the above equality implies that $\p_k$ must be $\p_i$ or $\p_j$. 
\end{proof}

\subsection{Shards as restricted covectors}
\label{ssec:shards_as_cov}
Let $\A$ be a tight hyperplane arrangement with respect to a base region $B$ and $\A^*_B$ be its associated affine point configuration.
Every shard $\Sigma$ of $\A$ has a corresponding unique join-irreducible region $J_\Sigma$ \cite[Proposition~9-7.8]{reading_lattice_2016}.
In the lattice of regions, $J_\Sigma$ is the meet of all regions $R$ such that
\[
H_\Sigma \in \Sep(R)\quad \text{ and }\quad R \cap \Sigma \text{ has dimension }d-1.
\]
The next lemma shows how $\pre(H_\Sigma)$ and $\Sep(J_\Sigma)$ yield a description of the shard as the intersection of half-spaces.
It is originally stated for simplicial arrangements, though the same holds true for tight hyperplane arrangements.

\begin{lemma}[{see \cite[Lemma 3.7]{reading_lattice_2004}}]
\label{lem:shards_Hdesc}
A shard $\Sigma$ has the following description:
$$
\Sigma = \bigg \{ \x \in H_\Sigma \, \bigg \lvert 
	\begin{array}{l} \n_i \cdot \x \geq 0 \text{  if  } H_i \in \pre(H_\Sigma) \cap \Sep(J_\Sigma) \\
	                 \n_i \cdot \x \leq 0 \text{  if  } H_i \in \pre(H_\Sigma) \setminus \Sep(J_\Sigma)
	 \end{array} \bigg \}.
$$
\end{lemma}

To interpret shards on a hyperplane $H_i$ as covectors, we restrict to a certain subconfiguration containing $\p_i$.

\begin{definition}[Subconfiguration localized at a point]
Let $\p_i\in\A_B^*$.
The \defn{subconfiguration~$\A_{B,i}^*$ of $\A_B^*$ localized at $\p_i$} contains $\p_i$ and the vertices of the convex hulls of lines of $\A_B^*$ that contain~$\p_i$ in their interior. 
\end{definition}

Lemma~\ref{lem:lines_ii} and Definition~\ref{def:directed_graph} imply the following lemma.

\begin{lemma}
\label{lem:subconf}
The subconfiguration $\A_{B,i}^*$ satisfies
\[
\A_{B,i}^*=\{\p_i\}\cup\{\p_j~:~H_j\in \pre(H_i)\}.
\]
\end{lemma}

\begin{definition}[Shard covectors of a point]
Let $\p_i\in\A_B^*$.
A \defn{shard covector of $\p_i$} is a restricted covector~$\sigma^i=\cov|_{\A_{B,i}^*}$ with respect to $\A_{B,i}^*$ such that
\begin{itemize}
\item $\sigma^i_j=*$\quad if and only if\quad $\p_j\not\in \A_{B,i}^*$, and
\item the restriction of $\sigma^i$ to the subconfiguration $\A_{B,i}^*$ is a covector with exactly one zero in position ``$i$''.
\end{itemize}
\end{definition}

\begin{example}
\label{ex:A6_24}
In Figure~\ref{fig:A6_24}, the left image illustrates the affine point configuration $\A(6,24)^*$ for the rank-$3$ braid arrangement with $6$ hyperplanes. 
The right image illustrates the subconfiguration of $\A(6,24)^*$ localized at $\p_6$, $\A(6,24)^*_{6}$. 

\begin{figure}[H]
\begin{center}
\begin{tabular}{c@{\hspace{2cm}}c}
\begin{tikzpicture}
	[scale=0.75,
	 point/.style={circle,fill=black,inner sep=1.5pt},
	 line/.style={black!20,line width=2pt}]

\coordinate (1) at (0.000000000000000,0.000000000000000);
\coordinate (2) at (2.00000000000000,3.46410161513775);
\coordinate (3) at (4.00000000000000,0.000000000000000);
\coordinate (4) at (1.00000000000000,1.73205080756888);
\coordinate (5) at (3.00000000000000,1.73205080756888);
\coordinate (6) at (2.00000000000000,1.15470053837925);

\draw[line] (5) -- (1) -- (2) -- (3) -- (4);

\node[point,label=left:{$\p_1$}]  at (1) {};
\node[point,label=above:{$\p_2$}] at (2) {};
\node[point,label=right:{$\p_3$}] at (3) {};
\node[point,label=left:{$\p_4$}]  at (4) {};
\node[point,label=right:{$\p_5$}] at (5) {};
\node[point,label=below:{$\p_6$}] at (6) {};

\end{tikzpicture}
&
\begin{tikzpicture}
	[scale=0.75,
	 point/.style={circle,fill=black,inner sep=1.5pt},
	 line/.style={black!20,line width=2pt}]

\coordinate (1) at (0.000000000000000,0.000000000000000);
% \coordinate (2) at (2.00000000000000,3.46410161513775);
\coordinate (3) at (4.00000000000000,0.000000000000000);
\coordinate (4) at (1.00000000000000,1.73205080756888);
\coordinate (5) at (3.00000000000000,1.73205080756888);
\coordinate (6) at (2.00000000000000,1.15470053837925);

\draw[line] (1) -- (5);
\draw[line] (3) -- (4);
% \draw[line,dashed,very thick] (1) -- (2);
% \draw[line,dashed,very thick] (3) -- (2);

\node[point,label=left:{$\p_1$}]  at (1) {};
% \node[point,label=above:{$\p_2$},black!20] at (2) {};
\node[point,label=right:{$\p_3$}] at (3) {};
\node[point,label=left:{$\p_4$}]  at (4) {};
\node[point,label=right:{$\p_5$}] at (5) {};
\node[point,label=below:{$\p_6$}] at (6) {};

\end{tikzpicture}
\end{tabular}
\end{center}
\caption{The point configuration $\A(6,24)^*$ for the rank-3 braid arrangement and the subconfiguration $\A(6,24)^*_{6}$ localized at $\p_6$}
\label{fig:A6_24}
\end{figure}
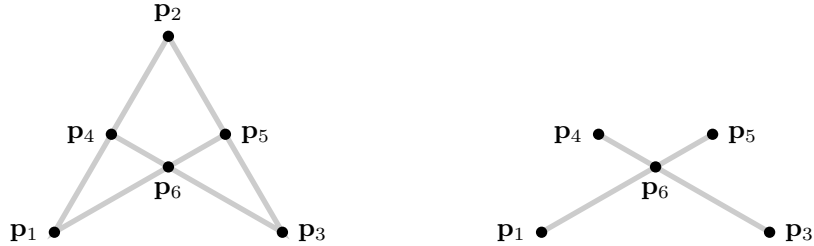

\noindent
There are two pairs of oppositely signed shard covectors of $\p_6$:
\begin{align*}
	\sigma^{6,+} = (+,*,+,-,-,0),\qquad\theta^{6,+}  = (+,*,-,+,-,0),\\
	\sigma^{6,-} = (-,*,-,+,+,0),\qquad\theta^{6,-}  = (-,*,+,-,+,0).
\end{align*}
It is possible to obtain these shard covectors by drawing a line through $\p_6$ in $\A(6,24)^*_{6}$, and choosing a positive and a negative side.
Rotating the line about $\p_6$ in all possible directions, and recording the sign evaluations of the points in $\A(6,24)^*_{6}$ relative to the line exhausts all possibilities. 
\end{example}

We now associate a restricted covector to each shard using the sign evaluation of vectors.
Let~$\Sigma^i$ be a shard contained in hyperplane $H_i$, and let $\x\in\inter(\Sigma^i)$. 
Using Lemma~\ref{lem:shards_Hdesc} and~\ref{lem:subconf}, we get
\[
\cov_{\A^*_{B,i}}(\x) := \left(\cov_{\A^*_{B,i}}(\x)_j=
\begin{cases}
0 & \text{if } j=i,\\
+ & \text{if } H_j\in\pre(H_i)\cap\Sep(J_\Sigma) \\
- & \text{if } H_j\in\pre(H_i)\setminus\Sep(J_\Sigma) \\
\end{cases}
\right)_{j\in[m]\text{ and }\p_j\in\A^*_{B,i}}.
\]
Completing this sign evaluation to the configuration $\A^*_B$, we get the restricted covector
\[
\sigma^i := \left(\sigma^i_j=
\begin{cases}
\cov_{\A^*_{B,i}}(\x)_j & \text{if } H_j\in\pre(H_i)\cup\{H_i\}\\
* & \text{if } H_j\notin\pre(H_i)\cup\{H_i\}
\end{cases}
\right)_{j\in[m]}.
\]
This restricted covector is independent of the choice of vector $\x\in\inter(\Sigma^i)$, thanks to Lemma~\ref{lem:shards_Hdesc}, and only depends on the choice of base region $B$.

\begin{theorem}
\label{thm:shard_bijection}
Let $\A$ be a tight hyperplane arrangement with respect to a base region $B$.
The map sending a shard~$\Sigma^i$ to the shard covector~$\sigma^i$ gives a bijection between the shards of $\A$ with base region~$B$ and the shard covectors of $\A^*_B$.
\end{theorem}

\begin{proof}
\emph{Injectivity.}
Suppose $\sigma^i = \theta^i$ for two shards $\Sigma^i$ and $\Theta^i$, for some $i\in[m]$.
By the definition of $\sigma^i$ and~$\theta^i$, the shard covectors are obtained from some points $\x \in \inter(\Sigma^i)$ and $\y \in \inter(\Theta^i)$ and
\[
\sign(\n_j \cdot \x) = \sign(\n_j \cdot \y) \quad\text{for every } j \text{ such that }  H_j\in \pre(H_i) \cup \{H_i\}.
\]
By Lemma~\ref{lem:shards_Hdesc}, an H-description of the shard is given by the sign evaluation of any of its points in the relative interior with respect to the hyperplanes in $\pre(H_i)$. 
Because $\Sigma^i$ and $\Theta^i$ are both shards on hyperplane~$H_i$, and the sign evaluation of $\x$ and $\y$ agree on all normals in $\pre(H_i)$, $\Sigma^i$ and $\Theta^i$ must be the same. 

\emph{Surjectivity}.
Let $\cov$ be a shard covector with a unique zero at position $i \in [m]$. 
Considered as a sign evaluation, there is an $\x\in\R^d$ such that $\cov=\cov_{\A^*_B}(\x)|_{\A_{B,i}^*}$.
The linear hyperplane with normal $\x$ separates the normal vectors in $\pre(H_i)$ as $\cov$ dictates. 
Thus $\x$ is a point in the relative interior of a shard $\Sigma^i$ of $H_i$ such that ${\sigma^i =\cov}$.
\end{proof}

By Theorem~\ref{thm:shard_bijection}, there is a unique shard covector associated to every shard.
We therefore use lowercase Greek letters $\sigma^i$ to denote the unique shard covector corresponding to a shard~$\Sigma^i$.

\subsection{Forcing relation on covectors}
\label{ssec:forcing_on_cov}
In this section, we use Theorem~\ref{thm:shard_bijection} and interpret the shard digraph~$\Sh_B(\A)$ using shard covectors of $\A^*_B$.
In Definition~\ref{def:shard_dig}, the first condition to get an edge $\Sigma^i\rightarrow\Sigma^j$ translates to the shard covectors of $\p_j$ having a $+$ or $-$ at position~``$i$''.
The second condition requires one to interpret the dimension of intersection of two shards using shard covectors.
To do so, we define \defn{line covectors} of two hyperplanes.  

\begin{definition}[Line covector]
Let $\ell\in\Lines(\A_B^*)$.
A \defn{line covector} of $\ell$ is a covector $\hypint^\ell$ on $\A_B^*$ such that 
\begin{center}
$\hypint^{\ell}_k=0$\quad if and only if\quad $\p_k\in\ell$.
\end{center}
\end{definition}

Line covectors record possible sign evaluations of non-zero points in the intersection of two hyperplanes with respect to~$\A_B^*$.
They come in oppositely signed pairs which we denote by~$\hypint^{\ell,+}$ and $\hypint^{\ell,-}$. 
In the case of rank-3 hyperplane arrangements, these covectors are actually cocircuits of the oriented matroid.
For higher-rank hyperplane arrangements, the set of $0$-indices of a line covector gives a flat of rank~$2$ in the underlying matroid.

\begin{example}[Example~\ref{ex:A6_24} continued]
Let $\ell=\{\p_1,\p_5,\p_6\}$.
Since $\mathbb{A}_B$ has dimension $2$, the line $\ell$ has exactly two line covectors.
From Figure~\ref{fig:A6_24}, we deduce that the line covectors of $\ell$ are:
\[
\begin{aligned}
	\hypint^{\ell,+} &= (0,+,-,+,0,0), \\
	\hypint^{\ell,-} &= (0,-,+,-,0,0). \\
\end{aligned}
\]
\end{example}

\begin{lemma}
\label{lem:codim2}
Let $\A_B^*=\{\p_i\}_{i\in[m]}$ be an affine point configuration, $1\leq i<j\leq m$, $\ell$ be the line spanned by $\p_i$ and $\p_j$, and $\hypint^\ell$ be a line covector of $\ell$.
The set ${\{\, \x\in (H_i \cap H_j)~:~\cov_{\A_B^*}(\x)=\hypint^\ell\, \}}$ has dimension~$d-2$.
\end{lemma}

\begin{proof}
Let $\x \in H_i \cap H_j$ with $\cov_{\A_B^*}(\x) = \hypint^\ell$. 
For any $\ve \in \spa(\n_i,\n_j)^\perp$ and $\varepsilon>0$, the $k$-th entry of $\cov_{\A_B^*}(\x + \varepsilon \ve)$ is equal to
\begin{align*}
	\cov_{\A_B^*}(\x + \varepsilon \ve)_k & = \sign{( \x \cdot \n_k + \varepsilon( \ve \cdot \n_k))} \\
				      & =
\begin{cases}
0 & \text{ if } k \in \{i,j \}, \\
\sign(\x \cdot \n_k + \varepsilon(\ve \cdot \n_k)) & \text{ if } k \notin \{i,j\}.\\
\end{cases}
\end{align*}
When $\varepsilon$ is chosen small enough, then
\[
\cov_{\A_B^*}(\x + \varepsilon \ve)_k = \cov_{\A_B^*}(\x)_k = \hypint^\ell_k.
\]
Thus $\dim(\{\x\in (H_i \cap H_j)~:~\cov_{\A_B^*}(\x)=\hypint^\ell\}) = \dim(\spa(\n_i,\n_j)^\perp) = d-2.$
\end{proof}

\begin{example}[Example~\ref{ex:A6_24} continued]
Figure \ref{fig:A6_24_shards} shows a stereographic projection of $\A(6,24)$ broken into shards.
The shards $\Theta^{6,+}$ and $\Sigma^1 = H_1$ are thickened and one sees that~$H_1$ cuts $H_6$.
The shards $\Sigma^1$ and $\Theta^{6,+}$ intersect at a point so there is an oriented edge $\Sigma^1\rightarrow\Theta^{6,+}$ in the shard digraph.

\begin{figure}[H]
\begin{center}
\begin{tikzpicture}
\node at (0,0) {\includegraphics[width=2.5in]{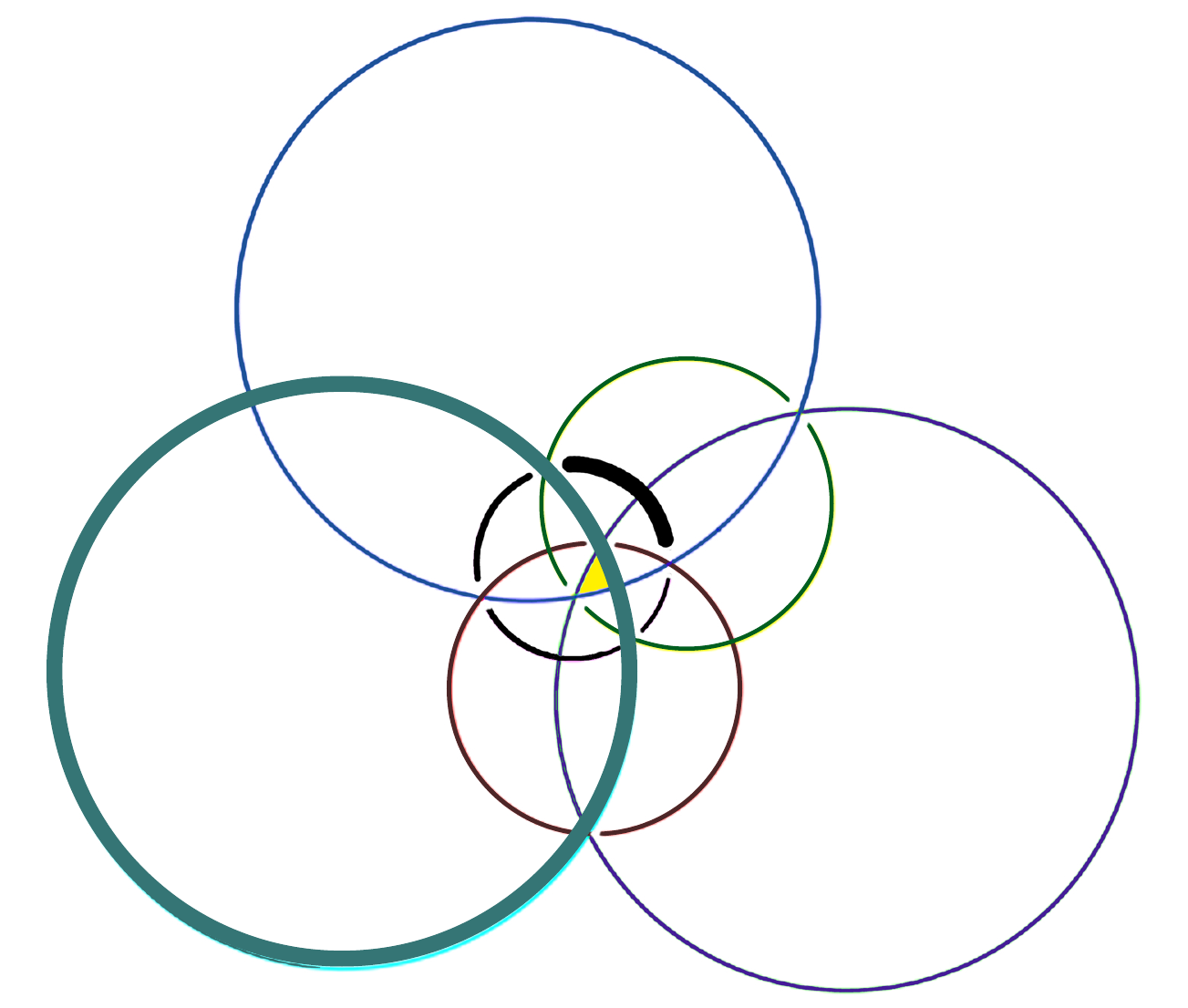}};
\node at (-3.6,-.5) {$H_1 = \Sigma^1$};
\node at (2.5,.4) {$H_2$};
\node at (-1.8,2.3) {$H_3$};
\node at (-1,-1.5) {$H_4$};
\node at (1.5,-.3) {$H_5$};
\node at (-.9,0) {$H_6$};
\node at (.35,.45) {$\Theta^{6,+}$};
\end{tikzpicture}
\end{center}
\caption{The shards of the arrangement $\A(6,24)$ shown via stereographic projection}
\label{fig:A6_24_shards}
\end{figure}
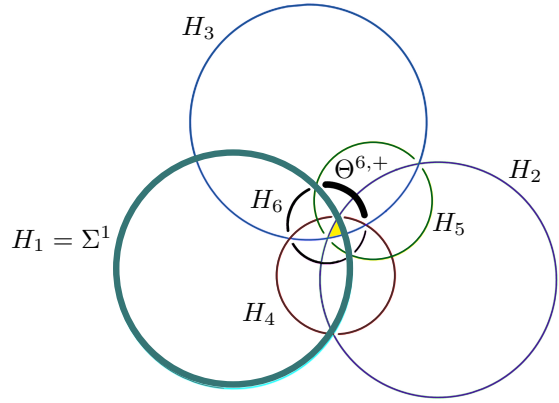

\noindent
This fact translates to a property of the corresponding shards covectors $\sigma^{1}=(0,*,*,*,*,*)$ and $\theta^{6,+}=(+,*,-,+,-,0)$.
Indeed, consider the line $\ell=\{\p_1,\p_5,\p_6\}$ and the line covector $\hypint^{\ell,+}=(0,+,-,+,0,0)$.
Then $\hypint^{\ell,+} \cap \theta^{6,+} \cap \sigma^{1} = (0,+,-,+,0,0)$.
Figure~\ref{fig:shic} illustrates the affine point configuration~$\A(6,24)^*$ along with the three oriented lines describing the involved covectors.

\begin{figure}[H]
\begin{center}
\begin{tikzpicture}
	[ point/.style={circle,fill=black,inner sep=1.5pt},
	 line/.style={black!20,line width=2pt}]

\coordinate (1) at (0.000000000000000,0.000000000000000);
\coordinate (2) at (2.00000000000000,3.46410161513775);
\coordinate (3) at (4.00000000000000,0.000000000000000);
\coordinate (4) at (1.00000000000000,1.73205080756888);
\coordinate (5) at (3.00000000000000,1.73205080756888);
\coordinate (6) at (2.00000000000000,1.15470053837925);

\draw[line,dashed] (1) -- (5);
\draw[line,dashed] (1) -- (2);
\draw[line,dashed] (3) -- (2);
\draw[line,dashed] (3) -- (4);
\draw[line,blue,shorten >= -35mm, shorten <= -10mm] (1) -- (5);
\draw[line,black] (2.00000000000000,1.15470053837925) -- +(2,-2);
\draw[line,black] (2.00000000000000,1.15470053837925) -- +(-2.5,2.5);
\draw[line,teal] (0,0) -- +(-3,3);
\draw[line,teal] (0,0) -- +(1,-1);
\node[point,label=below:{$\p_1$}]  at (1) {};
\node[point,label=above:{$\p_2$}] at (2) {};
\node[point,label=below:{$\p_3$}] at (3) {};
\node[point,label=left:{$\p_4$}]  at (4) {};
\node[point,label=below:{$\p_5$}] at (5) {};
\node[point,label=below:{$\p_6$}] at (6) {};
\node[label = right:{\textcolor{teal}{$-$}}] at (-2,2) {};
\node[label = left:{\textcolor{teal}{$+$}}] at (-2,2) {};
\node[label = left:{\textcolor{black}{$+$}}] at (0,3.15) {};
\node[label = right:{\textcolor{black}{$-$}}] at (0,3.15) {};
\node[label = above left:{\textcolor{blue}{$+$}}] at (4.5,2.625) {};
\node[label = below right:{\textcolor{blue}{$-$}}] at (4.5,2.625) {};
\node[label = below:{$\hypint^{\ell,+}$}]      at (6,3.46) {};
\node[label = left:{$\theta^{6,+}$}] at (-0.5,3.5) {};
\node[label = left:{$\sigma^{1}$}] at (-3,3) {};
\end{tikzpicture}
\end{center}
\caption{The point configuration $\A(6,24)^*$ and hyperplanes describing the covectors $\sigma^1$, $\theta^{6,+}$, and $\hypint^{\ell,+}$, where $\ell=\{\p_1,\p_5,\p_6\}$}
\label{fig:shic}
\end{figure}
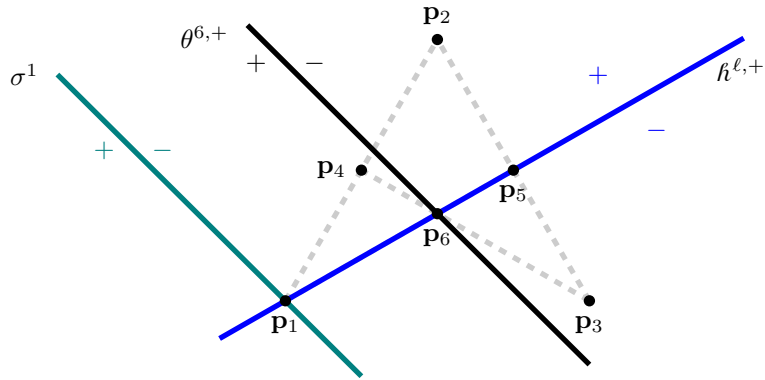

\noindent
It is possible to interpret the fact that the two shards intersect at a point as follows.
Apply a clockwise rotation to the line labeled $\theta^{6,+}$ about the point $\x_6$ until it collides with the line corresponding to $\hypint^{\ell,+}$. 
During the rotation, the line did not cross any points in $\A_{B,6}^*$. 
Similarly, applying the same with the line labeled $\sigma^1$ about $\x_1$ does not cross any points in $\A_{B,1}^*=\varnothing$.
\end{example}

The theorem below shows that the above equality is exactly the necessary and sufficient condition for the two involved shards to have an intersection of dimension~$d-2$.

\begin{theorem}
\label{thm:covector_forcing}
Let $\A_B^*=\{\p_i\}_{i\in[m]}$ be an affine point configuration, $1\leq i<j\leq m$, and let $\ell$ be the line spanned by $\p_i$ and $\p_j$.
Furthermore, let $\Sigma^i \subseteq H_i$ and ${\Theta^j \subseteq H_j}$ be two shards. 
The intersection $\Sigma^i \cap \Theta^j$ has dimension~$d-2$ if and only if there exists a line covector $\hypint^\ell$ such that $\hypint^{\ell}\cap \sigma^i\cap\theta^j= \hypint^\ell$.
\end{theorem}

\begin{proof}
Assume $\dim (\Sigma^i \cap \Theta^j) = d-2$. 
Hence there exists $\x \in \Sigma^i \cap \Theta^j$ such that the sign evaluation $\cov_{\A_B^*}(\x)_k$ equals zero if and only if $\p_k \in \ell$. 
Therefore $\cov_{\A_B^*}(\x)$ is a line covector of $\ell$.
If $\x$ is in the boundary of $\Sigma^i$, then for $\z \in \inter (\Sigma^i), \p_k \in \A_{B,i}^*$, either $\sign(\x \cdot \p_k) = \sign(\z \cdot \p_k)$ or $\sign( \x \cdot \p_k) = 0$.
As $\cov_{\A_B^*}(\x)_k$ equals zero if and only if $\p_k \in \ell$, $\sigma^i_k = \cov_{\A_{B}^*}(\x)_k$ for all $k$ such that $\p_k \in (\A_{B,i}^*\setminus\ell)$.
Likewise, $\theta^j_k = \cov_{\A_{B}^*}(\x)_k$ for all~$k$ such that $\p_k \in (\A_{B,j}^*\setminus\ell)$.
Thus, 
\[
\cov_{\A_B^*}(\x) \cap \sigma^i \cap \theta^j  = \cov_{\A_B^*}(\x).
\]

Assume now that there exists a line covector $\hypint^\ell$ such that ${\hypint^{\ell}\cap \sigma^i\cap\theta^j=\hypint^\ell}$.
Let $S = \{\x \in {(H_i \, \cap\, H_j)}~:\cov_{\A_B^*}(\x) = \hypint^\ell \}$.
By Lemma~\ref{lem:codim2}, $\dim(S)=d-2$.
Let $\x \in S$, $\y \in \inter(\Sigma^i)$, and $\z \in \inter(\Theta^j)$. 
As $\hypint^\ell \cap \sigma^i = \hypint^\ell$, $\sign(\x \cdot \p_k) = \sign(\y \cdot \p_k)$ for all $\p_k \in \A_{B,i}^* \setminus \ell$.
For $\p_k \in \ell$, we have $\x \cdot \p_k = 0$. 
For $0 \leq \lambda \leq 1$, let $\m_\lambda  = (1-\lambda)\y + \lambda \x$.
Then $\cov_{\A_{B,i}^*}(\m_\lambda)_k = \sign(\y \cdot \p_k)$ for all~$k$ such that $\p_k \in \A_{B,i}^*$, and $\lambda \in [0,1)$. 
This shows that $\m_\lambda \in \Sigma^i$ for all $\lambda \in [0,1)$, and thus~$\x$ is contained in $\Sigma^i$.
A similar argument with $\z$ shows that $\x$ is in $\Theta^j$. 
\end{proof}

\begin{corollary}
\label{cor:covector_forcing}
There is a directed arrow $\Sigma^i \rightarrow \Theta^j$ in $\Sh_B(\A)$ if and only if $\theta^j_i\in\{-,+\}$ and there exists a line covector $\hypint^{\ell}$ such that $\hypint^{\ell}\cap \sigma^i\cap\theta^j = \hypint^\ell$. 
\end{corollary}

\subsection{Obstruction to congruence normality}
\label{ssec:obstruction}

\begin{example}[Example~\ref{ex:A10_60} continued]
The normal vectors $\{\n_i\}_{i\in [10]}$ for this configuration can be chosen as follows. 
Let $\tau = \frac{1 + \sqrt{5}}{2}$ and 
$ \n_1=(0, 1, 0), \n_2=(1, 0, 0), \n_3=(1, 1, 0), \n_4=(1, 1, 1), \n_5=(\tau+1, \tau, \tau), \n_6=(\tau+1, \tau+1, 1), \n_7=(\tau+1, \tau+1, \tau), \n_8=(2\tau,2\tau,\tau), \n_9=(2\tau+1,2\tau,\tau), \n_{10}=(2\tau+2, 2\tau+1, \tau+1)$.
Let $B$ be the base region containing the vector $\ve= (-1,-1,-2)$.
Figure \ref{fig:A10_60_star} illustrates $\A(10,60)_{3,B}^{*}{}$ along with four lines describing the shard covectors
\begin{align*}
\begin{array}{ll}
\sigma^6    = (+, *, -, +, *, 0, *, *, -, *), & \theta^8= (-, *, -, +, *, *, *, 0, *, +), \\ 
\upsilon^9  = (*, -, -, *, +, *, *, +, 0, *), & \xi^{10}    = (*, -, *, +, +, -, +, *, -, 0).
\end{array}
\end{align*}

\begin{figure}[H]
\begin{tikzpicture}
	[scale=1.5,
	 point/.style={circle,fill=black,inner sep=1.5pt},
	 line/.style={black!20,line width=1.5pt},
	 rotate=-36.8698976]

\coordinate (A) at (0,0);
\coordinate (B) at (4,3);
\coordinate (C) at (2,3/2);
\coordinate (D) at (0,5);
\coordinate (E) at (0,5/2);
\coordinate (F) at (4/3,8/3);
\coordinate (G) at ($({sqrt(5)},-{sqrt(5)}/2)+(-1,3)$);
\coordinate (H) at ($({2*sqrt(5)/11},-{7*sqrt(5)/22})+(14/11,61/22)$);
\coordinate (I) at ($({-2*sqrt(5)/5},-{4*sqrt(5)/5})+(2,4)$);
\coordinate (J) at ($({-2*sqrt(5)},{7*sqrt(5)/2})+(6,-11/2)$);

% [[A,B,C],[A,G,J],[A,F,I],[A,D,E],[B,G,H],[B,I,J],[B,E,F],[C,D,F,J,H],[C,E,G],[D,G,I],[E,H,I]]$
\draw[line] (A) -- (B) -- (C);
\draw[line] (A) -- (G) -- (J);
\draw[line] (A) -- (F) -- (I);
\draw[line] (A) -- (D) -- (E);
\draw[line] (B) -- (G) -- (H);
\draw[line] (B) -- (I) -- (J);
\draw[line] (B) -- (E) -- (F);
\draw[line] (C) -- (D) -- (F) -- (J) -- (H);
\draw[line] (C) -- (E) -- (G);
\draw[line] (D) -- (G) -- (I);
\draw[line] (E) -- (H) -- (I);

\draw[thick, black!60!green] (H) -- +(5:2.3cm);
\draw[thick, black!60!green] (H) -- +(185:3.5cm);
\path (H) -- +(.4:2.2cm) coordinate (6+);
\node[label={[black!60!green]north east:{+}}] at (6+) {}; 
\path (H) -- +(5:2.3cm) coordinate (sigma6);
\node[label={below:{$\sigma^6$}}] at (sigma6) {}; 

\draw[thick, black!60!blue] (G) -- +(185:3cm);
\draw[thick, black!60!blue] (G) -- +(5:2.5cm);
\path (G) -- +(-.5:2.4cm) coordinate (9+);
\node[label={[black!60!blue]north east:{+}}] at (9+) {}; 
\path (G) -- +(5:2.5cm) coordinate (upsilon9);
\node[label={below:{$\upsilon^9$}}] at (upsilon9) {}; 

\draw[thick, black!60!red] (I) -- +(185:3cm);
\draw[thick, black!60!red] (I) -- +(5:3.1cm);
\path (I) -- +(2:3.1cm) coordinate (10+);
\node[label={[black!60!red]north east:{+}}] at (10+) {}; 
\path (I) -- +(5:3.1cm) coordinate (xi10);
\node[label={below:{$\xi^{10}$}}] at (xi10) {}; 

\draw[thick, black!60!orange] (J) -- +(85:3cm);
\draw[thick, black!60!orange] (J) -- +(-95:2cm);
\path (J) -- +(85:3cm) coordinate (8+);
\node[label={[black!60!orange]left:{+}}]  at (8+)  {};
\path (J) -- +(85:3cm) coordinate (theta8);
\node[label={{$\theta^8$}}]  at (theta8)  {};

\node[point,label=left:{$\p_2$}]  at (A) {};
\node[point,label=right:{$\p_1$}] at (B) {};
\node[point,label=below:{$\p_3$}] at (C) {};
\node[point,label=right:{$\p_4$}]  at (D) {};
\node[point,label=left:{$\p_5$}] at (E) {};
\node[point,label=above right:{$\p_7$}] at (F) {};
\node[point,label=below:{$\p_9$}] at (G) {};
\node[point,label=below:{$\p_6$}]  at (H) {};
\node[point,label=left:{$\p_{10}$}] at (I) {};
\node[point,label=right:{$\p_8$}] at (J) {};

\end{tikzpicture}
\caption{The point configuration $\A(10,60)_{3,B}^*$} 
\label{fig:A10_60_star}
\end{figure}
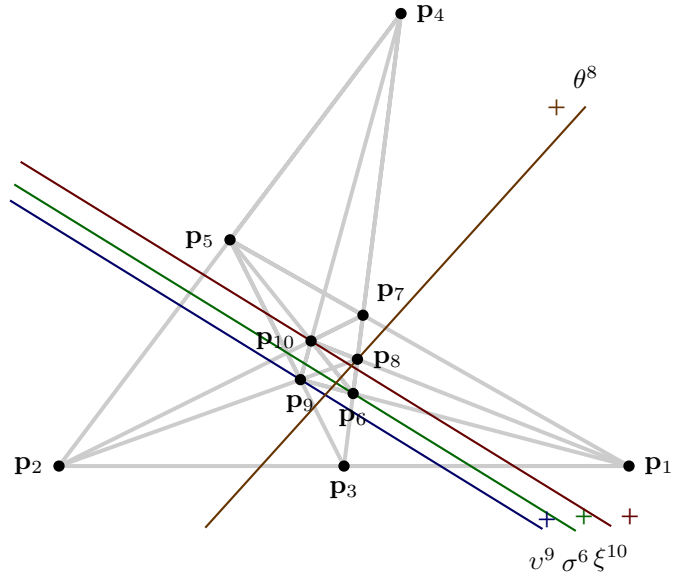

\noindent
Let $\ell_1=\{\p_2,\p_8,\p_9\},\ell_2=\{\p_1,\p_6,\p_9\},\ell_3=\{\p_5,\p_6,\p_{10}\},$ and $\ell_4=\{\p_1,\p_8,\p_{10}\}$ and consider the four line covectors
\begin{align*}
\begin{array}{ll}
\hypint^{\ell_1} = (-,0,-,+,+,-,+,0,0,+), & \hypint^{\ell_2} = (0,-,-,+,+,0,+,+,0,+), \\
\hypint^{\ell_3} = (+,-,-,+,0,0,+,+,-,0), & \hypint^{\ell_4} = (0,-,-,+,+,-,+,0,-,0).
\end{array}
\end{align*}
As $\upsilon^9$ has a ``$+$'' in position 8, $H_8$ cuts $H_9$.
Furthermore, one computes that $\hypint^{\ell_1} \cap\theta^8 \cap \upsilon^9 = \hypint^{\ell_1}$.
By Corollary~\ref{cor:covector_forcing}, there is a directed arrow $\Theta^9\rightarrow\Upsilon^9$ in $\Sh_B(\A)$.
Similar computations reveal that $\theta^8 \rightarrow \upsilon^9 \rightarrow \sigma^6 \rightarrow \xi^{10} \rightarrow \theta^8$ is a cycle in $\Sh_B(\A)$. 
Thus, the poset of regions of $\A(10,60)_3$ with respect to the base region $B$ is not congruence normal.
\end{example}

\begin{example}
\label{ex:min_not_CU}
Removing the hyperplane $H_4$ from the arrangement $\A(10,60)_3$ and taking the base region that contains the vector $\ve=(-1,-1,2)$, one obtains a non-simplicial, tight (hence semidistributive) poset of regions with 52 regions that is not congruence normal as the cycle $\theta^8 \rightarrow \upsilon^9 \rightarrow \sigma^6 \rightarrow \xi^{10} \rightarrow \theta^8$ still occurs in the shard digraph.
Figure~\ref{fig:min_not_CU} illustrates the resulting affine point configuration.
Is there a tight poset of regions which is not congruence normal with at most 8 hyperplanes?
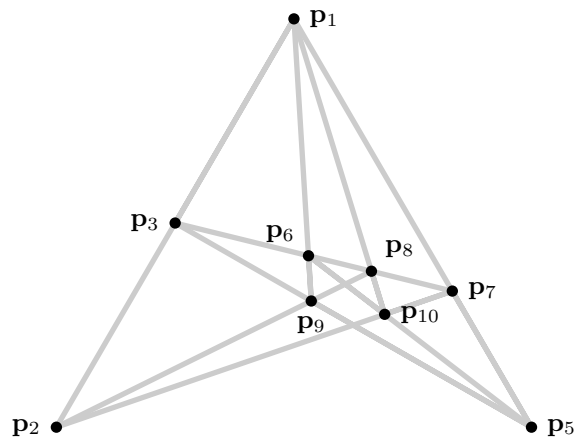
\begin{figure}[H]
\begin{tikzpicture}
	[x={(-0.875cm, 1.08253175473055cm)},
	 y={(2cm, 0cm)},
         scale=1.25,
	 point/.style={circle,fill=black,inner sep=1.5pt},
	 line/.style={black!20,line width=2pt}]

\coordinate (A) at (0,0);
\coordinate (B) at (4,3);
\coordinate (C) at (2,3/2);
\coordinate (D) at (0,5);
\coordinate (E) at (0,5/2);
\coordinate (F) at (4/3,8/3);
\coordinate (G) at ($({sqrt(5)},-{sqrt(5)}/2)+(-1,3)$);
\coordinate (H) at ($({2*sqrt(5)/11},-{7*sqrt(5)/22})+(14/11,61/22)$);
\coordinate (I) at ($({-2*sqrt(5)/5},-{4*sqrt(5)/5})+(2,4)$);
\coordinate (J) at ($({-2*sqrt(5)},{7*sqrt(5)/2})+(6,-11/2)$);

% [[A,B,C],[A,G,J],[A,F,I],[A,D,E],[B,G,H],[B,I,J],[B,E,F],[C,D,F,J,H],[C,E,G],[D,G,I],[E,H,I]]
\draw[line] (A) -- (B) -- (C);
\draw[line] (A) -- (G) -- (J);
\draw[line] (A) -- (F) -- (I);
% \draw[line] (A) -- (D) -- (E);
\draw[line] (B) -- (G) -- (H);
\draw[line] (B) -- (I) -- (J);
\draw[line] (B) -- (E) -- (F);
\draw[line] (C) -- (F);
\draw[line] (C) -- (E) -- (G);
% \draw[line] (D) -- (G) -- (I);
\draw[line] (E) -- (H) -- (I);

\node[point,label=left:{$\p_2$}]  at (A) {};
\node[point,label=right:{$\p_1$}] at (B) {};
\node[point,label=left:{$\p_3$}] at (C) {};
% \node[point,label=right:{$4$}]  at (D) {};
\node[point,label=right:{$\p_5$}] at (E) {};
\node[point,label=right:{$\p_7$}] at (F) {};
\node[point,label=below:{$\p_9$}] at (G) {};
\node[point,label=above left:{$\p_6$}]  at (H) {};
\node[point,label=right:{$\p_{10}$}] at (I) {};
\node[point,label=above right:{$\p_8$}] at (J) {};

\end{tikzpicture}
\caption{An affine point configuration leading to a tight, non-congruence normal hyperplane arrangement}
\label{fig:min_not_CU}
\end{figure}
\end{example}

\begin{example}
\label{ex:cycle_in_H}
It is possible to have cycles in $\H_B(\A)$ while $\Sh_B(\A)$ is acyclic, settling the question raised in \cite[p.~203]{reading_lattice_2003}. 
Figure~\ref{fig:cycle_in_H} shows the affine point configuration of arrangement $\A(14,116)$ with respect to the base region that contains the vector $\approx(0.38, 2.85, -7.85)$.
There is a cycle $H_1 \rightarrow H_4 \rightarrow H_7 \rightarrow H_1$ in~$\H_B(\A)$.
However, this cycle does not lead to any cycle among shards included in these three hyperplanes as $\Sh_B(\A)$ was computed to be acyclic in this case. 

\begin{figure}[H]
\begin{tikzpicture}
	[x={(1.6cm, 0cm)},
	 y={(0cm, 1cm)},
	 scale=1.25,
	 point/.style={circle,fill=black,inner sep=1.5pt},
	 line/.style={black!20,line width=1pt}]

\coordinate (0) at  (2.84240404744278, 1.81787943172609);
\coordinate (1) at  (0.00000000000000, 0.00000000000000);
\coordinate (2) at  (2.50000000000000, 1.59889252318084);
\coordinate (3) at  (2.09220259843842, 1.33808283664909);
\coordinate (4) at  (2.50000000000000, 4.33012701892219);
\coordinate (5) at  (2.50000000000000, 0.79130472169025);
\coordinate (6) at  (2.96237314857923, 0.93765594395167);
\coordinate (7) at  (0.77254248593736, 1.33808283664909);
\coordinate (8) at  (2.13525491562421, 1.65396134562402);
\coordinate (9) at  (2.68237254218789, 1.33808283664909);
\coordinate (10) at (1.63627124296868, 2.83410492778564);
\coordinate (11) at (4.22745751406263, 1.33808283664909);
\coordinate (12) at (1.34836165729158, 2.33543089740679);
\coordinate (13) at (5.00000000000000, 0.00000000000000);

\draw[line] (0) -- (5);
\draw[line] (0) -- (8);
\draw[line] (10) -- (13);
\draw[line] (11) -- (12);
\draw[line] (1) -- (4);
\draw[line] (1) -- (11);
\draw[line] (4) -- (5);
\draw[line] (6) -- (10);
\draw[line] (8) -- (11);
\draw[line] (12) -- (13);
\draw[line] (3) -- (4);
\draw[line] (5) -- (12);
\draw[line] (7) -- (11);
\draw[line] (4) -- (13);
\draw[line] (7) -- (13);
\draw[line] (5) -- (10);
\draw[line] (6) -- (12);
\draw[line] (8) -- (13);
\draw[line,red,line width=2pt,dashed] (4) -- (6);
\draw[line,red,line width=2pt,dashed] (0) -- (1);
\draw[line,red,line width=2pt,dashed] (3) -- (13);

\node[point,label=above right:{$1$}]  at (0) {};
\node[point,label=left:{$2$}]  at (1) {};
\node[point,label=right:{$3$}]  at (2) {};
\node[point,label=below:{$4$}]  at (3) {};
\node[point,label=left:{$5$}]  at (4) {};
\node[point,label=below:{$6$}]  at (5) {};
\node[point,label=below:{$7$}]  at (6) {};
\node[point,label=left:{$8$}]  at (7) {};
\node[point,label=left:{$9$}]  at (8) {};
\node[point,label=right:{$10$}]  at (9) {};
\node[point,label=left:{$11$}]  at (10) {};
\node[point,label=right:{$12$}]  at (11) {};
\node[point,label=left:{$13$}]  at (12) {};
\node[point,label=right:{$14$}]  at (13) {};

\end{tikzpicture}
\caption{The point configuration of arrangement $\A(14,116)$ with respect to the base region containing $\approx(0.38, 2.85, -7.85)$}
\label{fig:cycle_in_H}
\end{figure}
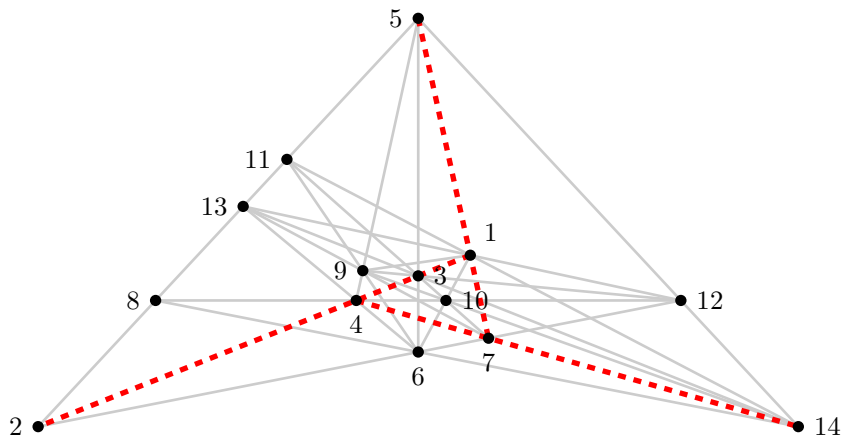
\end{example}

\section{Congruence normality of simplicial hyperplane arrangements}
\label{sec:simplicial}
% As the number of regions of a hyperplane arrangement typically grows exponentially with the number of hyperplanes, its posets of regions becomes infeasible to construct in practice when the number of hyperplanes gets large.
% Consequently, checking if a given poset of regions is obtainable through doublings of convex sets becomes impracticable.
% Using shards to determine congruence normality first involves computing polyhedral cones contained in each hyperplane (i.e.\ regions in restrictions of the hyperplane arrangement) and then pairwise intersections of shards in two hyperplanes $H_i\rightarrow H_j$ in $\H_B(\A)$.
% Getting the dimension of the intersection of two shards involves the computation of a longest chain in a face lattice, which still potentially involves data of exponential size.
% In contrast, the methods developed in Section~\ref{sec:cong_norm} make the determination of congruence normality for posets of regions of rank-$3$ hyperplane arrangements tractable.

As the number of regions of a rank-three hyperplane arrangement grows quadratically with the number of hyperplanes, its poset of regions becomes costly to construct in practice when the number of hyperplanes gets large.
Consequently, checking if a given poset of regions is obtainable through doublings of convex sets becomes impracticable.
Using shards to determine congruence normality first involves computing polyhedral cones contained in each hyperplane (i.e.\ regions in restrictions of the hyperplane arrangement) and then checking the dimensions of intersection for pairs of shards.
In contrast, the methods developed in Section~\ref{sec:cong_norm} make the determination of congruence normality for posets of regions of rank-$3$ hyperplane arrangements tractable and could be extended to higher dimensions given a method for determining the covectors of the oriented matroid.
Additionally, the oriented matroid approach makes it possible and natural to check congruence normality for non-realizable oriented matroids.

One of the motivations for studying congruence normality is to better understand simplicial hyperplane arrangements.
In rank $3$, the number of simplicial hyperplane arrangements is unknown \cite{grunbaum_2009,cuntz_greedy_2020}.
So far, three infinite families and 95 sporadic arrangements have been found.
It is conjectured that there are only finitely many sporadic arrangements.
The largest sporadic arrangement found so far has 37 hyperplanes.
In this section, we apply our reformulation of shards as shard covectors to classify which of the known simplicial hyperplane arrangements of rank $3$ are congruence normal.
This verification was carried out using \texttt{Sage} \cite{sagemath}.
The computations took around 18 hours on $8$ Intel Cores (i7-7700 @3.60Hz).
The verification for each poset of regions was computed independently, for example the cocircuits were recomputed for each reorientation of the set of normals, but the computations of intersections on covectors were cached.
The computation could be further improved by applying the reorientation on cocircuits directly in order to avoid recomputing them. 

Our results are summarized in Table~\ref{tab:cn_part}.
We use the following notation: $\A(m,r)_i$ denotes the $i$-th hyperplane arrangement with $m$ hyperplanes and $r$ regions.
We refer to congruence normality using the acronym \textbf{CN} and \textbf{NCN} for non-congruence normality.
The normals of the $119$ arrangements from the known sporadic arrangements and two of the infinite families are listed in the Appendix~\ref{app:lists} and the corresponding wiring diagrams are listed in Appendix~\ref{app:wd}.
The list includes the sporadic arrangements and the arrangements from the infinite families with at most $37$ hyperplanes.

\begin{table}[H]
\begin{tabular}{c@{\hspace{1.5cm}}c@{\hspace{1.5cm}}c}
$\PosBA$ always \textbf{CN} & $\PosBA$ sometimes \textbf{CN} & $\PosBA$ never \textbf{CN} \\\toprule[1pt]
Rank-$3$ Finite Weyl Groupoids & $\F_2(m)$ ($m\geq 10$) & $\A(22,288)$ \\
(including $\F_2(m)$ ($m\leq 8$) & $\F_3(m)$ ($m\geq 17$) & $\A(25,360)$ \\ 
and $\F_3(m)$ ($m\leq13$)) & 41 arrangements                 & $\A(35,680)$ \\
$\A(15,120)$ & & \\
$\A(31,480)$ \\
$\F_1(m)$ & & \\\bottomrule[1pt]
55 arrangements & 61 arrangements & 3 arrangements \\
see Section~\ref{ssec:always} & see Sections~\ref{ssec:sometimes} and \ref{ssec:infinite} & see Section~\ref{ssec:never} \\ 
and Table~\ref{tab:alwayscn} & and Table~\ref{tab:sometimescn} & and Table~\ref{tab:nevercn}
\end{tabular}
\caption{Classification of rank-$3$ simplicial hyperplane arrangements with at most 37 hyperplanes according to the congruence normality of their posets of regions}
\label{tab:cn_part}
\end{table}

Table~\ref{tab:cn_part} provides material to check the veracity of \cite[Conjecture~145]{padrol_shard_2020}, which postulates the existence of certain polytopes for tight congruence normal arrangements. 
Section~\ref{ssec:always} looks at the arrangements that are always \textbf{CN}, Section~\ref{ssec:sometimes} at the arrangements that are sometimes \textbf{CN}, and Section~\ref{ssec:never} at the arrangements that are never \textbf{CN}.
In Section~\ref{ssec:discuss} we finish by discussing these results and compiling related questions.

\subsection{Always \textbf{CN} simplicial arrangements}
\label{ssec:always}
Fifty-five of the of $119$ arrangements are congruence normal, that is, for any choice of base region, the poset of regions is congruence normal, see Table~\ref{tab:alwayscn}. 

\begin{table}[H]
\resizebox{\textwidth}{!}{
\begin{tabular}{c|c|c|c|c|c}
\multicolumn{6}{c}{Finite Weyl Groupoids}\\\toprule[1pt]
\begin{tabular}{r@{\hspace{1mm}}l}
$\F_2(6)=$  & $\A(6,24)$\\
& $\A(7,32)$         \\
$\F_2(8)=$  & $\A(8,40)$\\
$\F_3(9)=$  & $\A(9,48)$\\
& $\A(10,60)_1$      \\
& $\A(10,60)_2$      \\
& $\A(11,72)$        \\
& $\A(12,84)_1$      \\
& $\A(12,84)_2$      \\
$\F_3(13)=$ & $\A(13,96)_1$\\
\end{tabular} &
\begin{tabular}{c}
$\A(13,96)_2$      \\
$\A(13,96)_3$      \\
$\A(14,112)_1$     \\
$\A(15,128)_1$     \\
$\A(16,144)_1$     \\
$\A(16,144)_2$     \\
$\A(17,160)_1$     \\
$\A(17,160)_2$     \\
$\A(17,160)_3$      \\
\phantom{}
\end{tabular} &
\begin{tabular}{l}
$\A(18,180)_1$      \\
$\A(18,180)_2$      \\
$\A(19,192)_1$      \\
$\A(19,192)_2$      \\
$\A(19,200)_1$      \\
$\A(19,200)_2$      \\
$\A(19,200)_3$      \\
$\A(20,216)$     \\
$\A(20,220)_1$      \\
\phantom{}
\end{tabular} &
\begin{tabular}{l}
$\A(20,220)_2$   \\
$\A(21,240)_1$   \\
$\A(21,240)_2$   \\
$\A(21,240)_3$   \\
$\A(22,264)_1$   \\
$\A(25,336)_1$   \\
$\A(25,336)_2$   \\
$\A(25,336)_3$   \\
$\A(25,336)_4$   \\
\phantom{}
\end{tabular} &
\begin{tabular}{l}
$\A(26,364)_1$   \\
$\A(26,364)_2$   \\
$\A(27,392)_1$   \\
$\A(27,392)_2$   \\
$\A(27,392)_3$   \\
$\A(28,420)_1$   \\
$\A(28,420)_2$   \\
$\A(28,420)_3$   \\
\phantom{}       \\
\phantom{}
\end{tabular} &
\begin{tabular}{l}
$\A(29,448)_1$   \\
$\A(29,448)_2$   \\
$\A(29,448)_3$   \\
$\A(30,476)$     \\
$\A(31,504)_1$   \\
$\A(31,504)_2$   \\
$\A(34,612)_1$   \\
$\A(37,720)_1$   \\
\phantom{}       \\
\phantom{}
\end{tabular}
\end{tabular}}
\begin{tabular}{c|c}
\multicolumn{2}{c}{Others}\\\toprule[1pt]
$H_3=\A(15,120)$ & $H_3^*=\A(31,480)$\\
\end{tabular}
\caption{List of congruence normal rank-3 simplicial arrangements}
\label{tab:alwayscn}
\end{table}

Fifty-three of these arrangements come from finite Weyl groupoids of rank~$3$~\cite{cuntz_finite_2012}. 
Finite Weyl groupoids correspond to (generalized) \emph{crystallographic} root systems.
In the present context, affine point configurations $\A_B^*$ play the role of these root systems.
A root system is \defn{crystallographic} if there exists a choice of normals $\{\n_i\}_{i\in[m]}$ for the hyperplanes such that for any base region, all normals are integral linear combinations of normals to the basic hyperplanes~\cite[Section~1]{cuntz_crystallographic_2011}.
Given a base region $B$, denote the set of rays of $\spa^+(\A_B^*)$ by $\Delta$ and call the elements of $\A_B^*$ the \defn{positive roots}.
A positive root $\p_i\in\A_B^*$ is \defn{constructible} if
\begin{align*}
\n_i\in\Delta \quad\text{or}\quad \n_i=\n_\alpha+\n_\beta,
\end{align*}
where $\alpha,\beta\in\A_B^*$.
We call $\A_B^*$ \defn{additive} if every positive root in $\A_B^*$ is constructible.
If $\A_B^*$ is additive, then it is possible to define the \defn{root poset} $(\A_B^*,\leq)$ by
\[
\p_i \leq \p_j \quad \Longleftrightarrow \quad \n_j - \n_i \in \N\Delta.
\]
The following is a fundamental result about finite Weyl groupoids.

\begin{theorem}[{\cite[Corollary~5.6]{cuntz_crystallographic_2011}} and {\cite[Theorem~2.10]{cuntz_finite_2012}}]
\label{thm:weyl_constr}
A simplicial arrangement $\A$ corresponds to a finite Weyl groupoid if and only if $\A_B^*$ is additive for every choice of base region~$B$.
\end{theorem}

\begin{theorem}
\label{thm:constructible}
Let $\A$ be a tight hyperplane arrangement with respect to a base region $B$.
If $\A_B^*$ is additive, then~$\PosBA$ is congruence normal.
\end{theorem}

\begin{proof}
Via the contrapositive statement, having a cycle in the graph $\H_B(\A)$ is a necessary condition for $\PosBA$ \emph{not} to be congruence normal.
Such a cycle between hyperplanes yields a cycle in the order defining the root poset of $\A_B^*$.
Hence, when $\PosBA$ is not congruence normal, the positive roots $\A_B^*$ do not lead to a root poset.
Thus $\A_B^*$ can not be additive.
\end{proof}

Theorem~\ref{thm:weyl_constr} leads directly to the following corollary, which provides a new proof that finite Coxeter arrangements are congruence normal \cite[Theorem~6]{caspard_cayley_2004}.

\begin{corollary}
\label{cor:cn_weyl}
Let $\A$ be the hyperplane arrangement of a finite Weyl groupoid $\mathcal{W}$.
For any choice of base region~$B$, the lattice of regions $\PosBA$ is congruence normal.
\end{corollary}

There are two additional \textbf{CN} arrangements that do not stem from finite Weyl groupoids.
Arrangement $\A(15,120)$ is the Coxeter arrangement for the Coxeter group $H_3$ and arrangement $\A(31,480)$ is its point-line dual.
As discussed in \cite{cuntz_root_2015}, there is a root poset for $H_3$ supporting the fact that its arrangement is always congruence normal.
The dual arrangement $\A(31,480)$ is also always congruence normal, as we verified directly.
Is there a proof of congruence normality for $\A(31,480)$ using duality with $H_3$?
Example~\ref{ex:cycle_in_H} shows that having a root poset structure on $\A_B^*$ \emph{is not} necessary for $\PosBA$ to be congruence normal.

\subsection{Simplicial arrangements that are sometimes congruence normal}
\label{ssec:sometimes}
Sixty-one of the $119$ arrangements are congruence normal for some base regions and not congruence normal for others, see Table~\ref{tab:sometimescn}.
Among them is the arrangement $\A(10,60)_3$ which appeared in Example~\ref{ex:A10_60}.

\begin{table}[H]
\resizebox{\textwidth}{!}{
\begin{tabular}{c|c|c}
\begin{tabular}{r@{\hspace{1mm}}lrr}
\multicolumn{2}{c}{Name} & \textbf{CN} & \textbf{NCN} \\\toprule[1pt]  
$\F_2(10)=$ & $\A(10,60)_3$   & $40$  & $20$     \\
$\F_2(12)=$ & $\A(12,84)_3$   & $36$  & $48$     \\
            & $\A(13,104)$    & $24$  & $80$     \\
$\F_2(14)=$ & $\A(14,112)_2$  & $28$  & $84$     \\
            & $\A(14,112)_3$  & $72$  & $40$     \\
            & $\A(14,116)$    & $40$  & $76$     \\
            & $\A(15,128)_2$  & $72$  & $56$     \\
            & $\A(15,132)_1$  & $60$  & $72$     \\
            & $\A(15,132)_2$  & $48$  & $84$     \\
            & $\A(16,140)$    & $120$ & $20$     \\
$\F_2(16)=$ & $\A(16,144)_3$  & $32$  & $112$    \\
            & $\A(16,144)_4$  & $84$  & $60$     \\
            & $\A(16,144)_5$  & $108$ & $36$     \\
            & $\A(16,148)$    & $52$  & $96$     \\
$\F_3(17)=$ & $\A(17,160)_4$  & $96$  & $64$     \\
            & $\A(17,160)_5$  & $120$ & $40$     \\
            & $\A(17,164)$    & $76$  & $88$     \\
            & $\A(17,168)_1$  & $48$  & $120$    \\
            & $\A(17,168)_2$  & $48$  & $120$    \\
$\F_2(18)=$ & $\A(18,180)_3$  & $36$  & $144$    \\
            & $\A(18,180)_4$  & $84$  & $96$     \\
\end{tabular}&
\begin{tabular}{r@{\hspace{1mm}}lrr}
\multicolumn{2}{c}{Name} & \textbf{CN} & \textbf{NCN} \\\toprule[1pt]  
            & $\A(18,180)_5$  & $120$ & $60$     \\
            & $\A(18,180)_6$  & $120$ & $60$     \\
            & $\A(18,184)_1$  & $100$ & $84$     \\
            & $\A(18,184)_2$  & $72$  & $112$    \\
            & $\A(19,200)_4$  & $120$ & $80$     \\
            & $\A(19,204)$    & $72$  & $132$    \\
$\F_2(20)=$ & $\A(20,220)_3$  & $40$  & $180$    \\
            & $\A(20,220)_4$  & $120$ & $100$    \\
$\F_3(21)=$ & $\A(21,240)_4$  & $80$  & $160$    \\
            & $\A(21,240)_5$  & $120$ & $120$    \\
            & $\A(21,248)$    & $88$  & $160$    \\
            & $\A(21,252)$    & $36$  & $216$    \\
$\F_2(22)=$ & $\A(22,264)_2$  & $44$  & $220$ \\
            & $\A(22,264)_3$  & $168$ & $96$     \\
            & $\A(22,276)$    & $60$  & $216$    \\
            & $\A(23,296)$    & $112$ & $184$    \\
            & $\A(23,304)$    & $8$   & $296$    \\
            & $\A(24,304)$    & $112$ & $192$    \\
$\F_2(24)=$ & $\A(24,312)$    & $48$  & $264$ \\
            & $\A(24,316)$    & $184$ & $132$    \\
\phantom{}
\end{tabular}&
\begin{tabular}{r@{\hspace{1mm}}lrr}
\multicolumn{2}{c}{Name} & \textbf{CN} & \textbf{NCN} \\\toprule[1pt]  
            & $\A(24,320)$    & $24$  & $296$    \\
            & $\A(25,320)$    & $288$ & $32$     \\
$\F_3(25)=$ & $\A(25,336)_5$  & $48$  & $288$ \\
            & $\A(25,336)_6$  & $48$  & $288$    \\
$\F_2(26)=$ & $\A(26,364)_3$  & $52$  & $312$ \\
            & $\A(26,380)$    & $20$  & $360$    \\
            & $\A(27,400)$    & $48$  & $352$    \\
$\F_2(28)=$ & $\A(28,420)_4$  & $56$  & $364$ \\
            & $\A(28,420)_5$  & $84$  & $336$    \\
            & $\A(28,420)_6$  & $84$  & $336$    \\
            & $\A(29,440)$    & $136$ & $304$    \\
$\F_3(29)=$ & $\A(29,448)_4$  & $56$  & $392$ \\
            & $\A(30,460)$    & $240$ & $220$    \\
$\F_2(30)=$ & $\A(30,480)$    & $60$  & $420$ \\
$\F_2(32)=$ & $\A(32,544)$    & $64$  & $480$ \\
$\F_3(33)=$ & $\A(33,576)$    & $64$  & $512$ \\
$\F_2(34)=$ & $\A(34,612)_2$  & $68$  & $544$ \\
$\F_2(36)=$ & $\A(36,684)$    & $72$  & $612$ \\
$\F_3(37)=$ & $\A(37,720)_2$  & $72$  & $648$ \\
            & $\A(37,720)_3$  & $96$  & $624$    \\
\phantom{}
\end{tabular}
\end{tabular}}
\caption{Simplicial arrangements that are sometimes congruence normal}
\label{tab:sometimescn}
\end{table}

Reading proved that the poset of regions of a supersolvable hyperplane arrangements is congruence normal with respect to a canonical base region~\cite[Theorem~1]{reading_lattice_2003}.
In rank~$3$, the infinite families are exactly the irreducible supersolvable ones \cite[Theorem~1.2]{cuntz_supersolvable_2019}.
However, we show below that~$\F_2(m)$ with $m\geq 10$ and $\F_3(m)$ with $m\geq 17$ always have a base region for which the associated lattice of regions is not congruence normal.

\subsection{Congruence normality for the infinite families}
\label{ssec:infinite}
There are three infinite families of rank-$3$ simplicial hyperplane arrangements \cite{grunbaum_1971}.
The first family, $\F_1(m)$ with $m \geq 3$ is the family of \defn{near-pencils} in the projective plane with $m$ hyperplanes.
The second family, $\F_2(m)$, for even $m \geq 6$ consists of the hyperplanes defined by the edges of the regular $\frac{m}{2}$-gon and each of its $\frac{m}{2}$ lines of symmetry.
The third family, $\F_3(m)$, for $m=4k+1$, $k\geq2$, is obtained from $\F_2(m-1)$ by adding the line at infinity.
Examples of these families are illustrated in Figure~\ref{fig:infinite_families}.

\begin{figure}[H]
\begin{center}
\begin{tabular}{c@{\hspace{2cm}}c@{\hspace{2.5cm}}c }
% first family: near pencil
\begin{tikzpicture}
	[baseline = (D),
	 point/.style={circle,fill=black,inner sep=1.5pt},
	 line/.style={black,line width=.5pt}]
\def \width{2}
\def \height{2}
\coordinate (A) at (0,0);
\coordinate (B) at (-1/2*\width,2/3*\height);
\coordinate (C) at (1/2*\width,2/3*\height);
\coordinate (D) at (1/2*\width,1/3*\height);
\draw[line] (B) -- (C);
\draw[line] (A) -- +(90:\height);
\draw[line] (A) -- +(-90:1/4*\height);
\draw[line] (A) -- +(105:\height);
\draw[line] (A) -- +(-75:1/4*\height);
\draw[line] (A) -- +(75:\height);
\draw[line] (A) -- +(-105:1/4*\height);
\draw[line] (A) -- +(60:\height);
\draw[line] (A) -- +(-120:1/4*\height);
\draw[line] (A) -- +(120:\height);
\draw[line] (A) -- +(-60:1/4*\height);
\node[label=below:{$\F_1(6)$}] at (0,-7/24*\height) {};
\end{tikzpicture}
&
%2nd family : pentagon
\begin{tikzpicture}
	[baseline = (A)
	 point/.style={circle,fill=black,inner sep=1.5pt},
	 line/.style={black,line width=.5pt}]
\def \width{3}
\def \height{3}
\coordinate (A) at (0,0);
\path (A) -- +(0:1/7*\width) coordinate (v1) ;
\path (A) -- +(72:1/7*\width) coordinate (v2);
\path (A) -- +(144:1/7*\width) coordinate (v3);
\path (A) -- +(216:1/7*\width) coordinate (v4);
\path (A) -- +(288:1/7*\width) coordinate (v5);
\draw[line, shorten  <=-10*\width, shorten >=-10*\width] (v1) -- (v2);
\draw[line, shorten  <=-10*\width, shorten >=-10*\width] (v2) -- (v3);
\draw[line, shorten  <=-10*\width, shorten >=-10*\width] (v3) -- (v4);
\draw[line, shorten  <=-10*\width, shorten >=-10*\width] (v4) -- (v5);
\draw[line, shorten  <=-10*\width, shorten >=-10*\width] (v5) -- (v1);
\draw[line, shorten  <=-13*\width, shorten >=-10*\width] (A) -- (v1);
\draw[line, shorten  <=-13*\width, shorten >=-10*\width] (A) -- (v2);
\draw[line, shorten  <=-13*\width, shorten >=-10*\width] (A) -- (v3);
\draw[line, shorten  <=-13*\width, shorten >=-10*\width] (A) -- (v4);
\draw[line, shorten  <=-13*\width, shorten >=-10*\width] (A) -- (v5);
\node[label=below:{$\F_2(10) = \A(10,60)_3$}] at (0,-5/12*\height) {};
\end{tikzpicture}
&
% third infinite family
\begin{tikzpicture}
	[baseline = (A),
	 point/.style={circle,fill=black,inner sep=1.5pt},
	 line/.style={black,line width=.5pt}]
 \def \width{2.5}
 \def \height{2.5}
\coordinate (A) at (0,0);
\path (A) -- +(45:1/5*\width) coordinate (v1) ;
\path (A) -- +(135:1/5*\width) coordinate (v2);
\path (A) -- +(225:1/5*\width) coordinate (v3);
\path (A) -- +(315:1/5*\width) coordinate (v4);
\path (A) -- +(0:1/5*\width) coordinate (e1);
\path (A) -- +(90:1/5*\width) coordinate (e2);
\path (A) -- +(180:1/5*\width) coordinate (e3);
\path (A) -- +(270:1/5*\width) coordinate (e4);
\draw[line, shorten  <=-10*\width, shorten >=-10*\width] (v1) -- (v2);
\draw[line, shorten  <=-10*\width, shorten >=-10*\width] (v2) -- (v3);
\draw[line, shorten  <=-10*\width, shorten >=-10*\width] (v3) -- (v4);
\draw[line, shorten  <=-10*\width, shorten >=-10*\width] (v4) -- (v1);
\draw[line, shorten  <=-15*\width, shorten >=-10*\width] (A) -- (e1);
\draw[line, shorten  <=-15*\width, shorten >=-10*\width] (A) -- (e2);
\draw[line, shorten  <=-15*\width, shorten >=-10*\width] (A) -- (v1);
\draw[line, shorten  <=-15*\width, shorten >=-10*\width] (A) -- (v2);
\node[label=above:{$\infty$}] at (1/2*\width,1/2*\height) {};
\node[label=below:{$\F_3(9) = \A(9,48)$}] at (0,-6/12*\height) {};
\end{tikzpicture}
\end{tabular}
\end{center}
\caption{Arrangements from the three infinite families of simplicial arrangements of rank $3$ drawn in the projective plane}
\label{fig:infinite_families}
\end{figure}
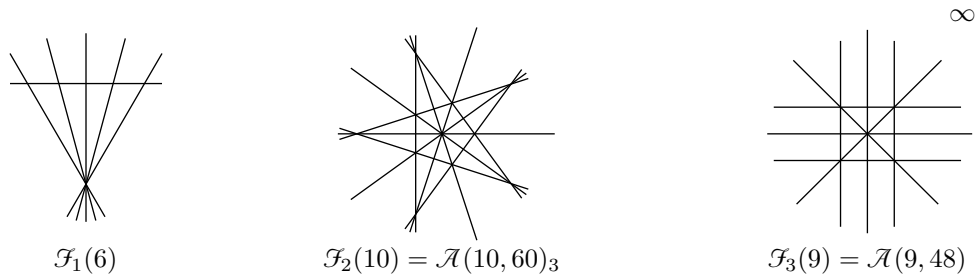

\begin{theorem}
\label{thm:near_pencils}
The near-pencil arrangements of $\F_1(m)$ are congruence normal.
\end{theorem}

\begin{proof}
There is exactly one rank-$2$ subarrangement with at least three hyperplanes.
Thus, for any choice of base region, the length of any path in the directed graph on shards is at most one, so there are no cycles.
\end{proof}

\begin{theorem}\label{thm:cn_f2}
The second family $\F_2(m)$ is sometimes congruence normal for $m \geq 10$. 
\end{theorem}

\begin{proof}
In rank~$3$, the infinite families are exactly the irreducible supersolvable ones, thus there exists a canonical choice of base region such that the poset of regions is congruence normal \cite[Theorem~1.2]{cuntz_supersolvable_2019}.
On the other hand, with respect to a certain choice of base region, there is a guaranteed four-cycle in the shards as demonstrated in Figure~\ref{fig:hexagon}.
The figure shows the arrangement on two projective planes and how some of the hyperplanes intersect at infinity.
Let the base region be bounded by $\mathbf{e}_1$, $\mathbf{e}_2$, and $\mathbf{r}_2$. 
At point~$1$, the hyperplane $\mathbf{e}_5$ is cut by~$\mathbf{r}_5$.
At point~$2$, the hyperplane $\mathbf{r}_6$ is cut by $\mathbf{e}_5$.
At point~$3$, the hyperplane $\mathbf{e}_4$ is cut by $\mathbf{r}_6$.
At point~$4$, the hyperplane $\mathbf{r}_5$ is cut by $\mathbf{e}_4$.
Thus there is a cycle in the shard digraph.
Adapting this procedure when $m\geq14$ similarly provides a $4$-cycle for every member of $\F_2(m)$.
\end{proof}

\begin{figure}[H]
\begin{center}
\resizebox{\textwidth}{!}{
\begin{tikzpicture}
\node at (0,0) {\includegraphics[width=\textwidth]{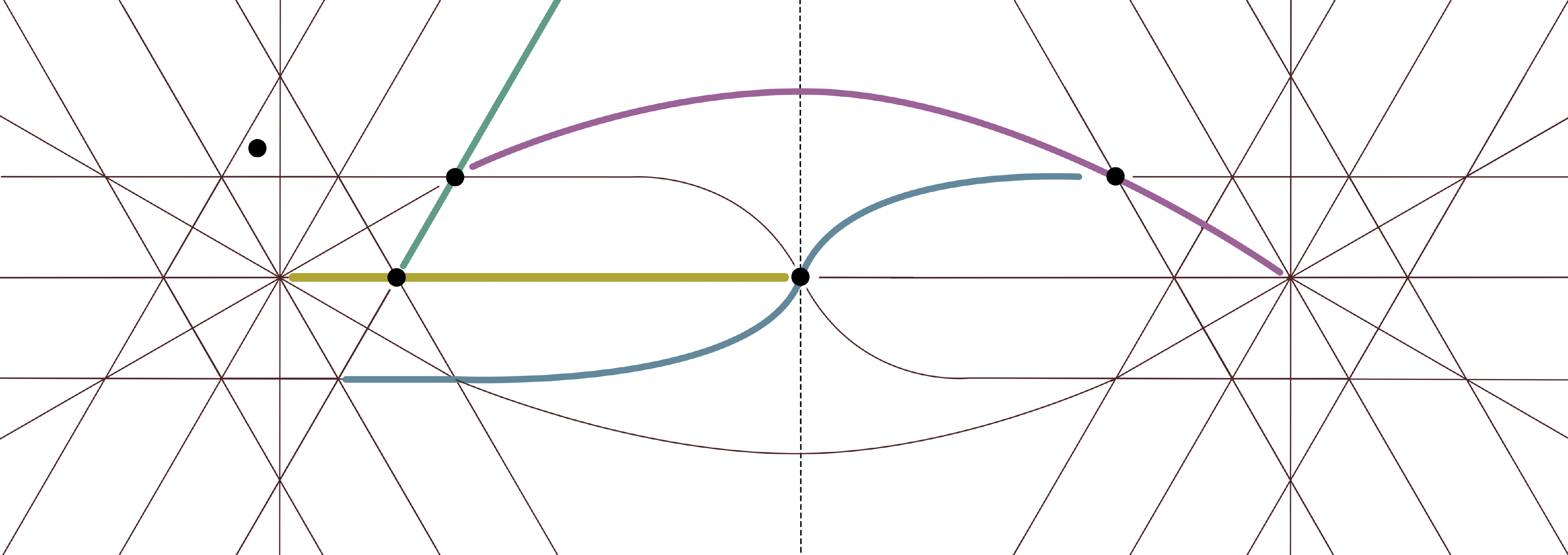}};
\node at (-1,1.2) {e$_1$};
\node at (-1,.2) {r$_5$};
\node at (-1,2) {r$_6$};
\node at (-1,-1.1) {e$_4$};
\node at (-1,-2.3) {r$_4$};
\node at (-2.1,-2.5) {e$_6$};
\node at (-3.3,-2.5) {r$_3$};
\node at (-4.4,-2.5) {e$_3$};
\node at (-5,-2.5) {r$_2$};
\node at (-5.6,-2.5) {e$_5$};
\node at (-6.2,-2.5) {r$_1$};
\node at (-7.3,-2.5) {e$_2$};
% points of intersection
\node at (-3.8,-.3) {1};
\node at (-3.5,1.3)  {2};
\node at (3.4,1.3)   {3};
\node at (.5,-.2)  {4};
% infinity
\node at (.5,2.7)    {$\infty$};
\end{tikzpicture}}
\end{center}
\caption{The simplicial hyperplane arrangement $\A(12,84)_3$ from $\F_2$ whose lattice of regions with the marked base region is not congruence normal}
\label{fig:hexagon}
\end{figure}
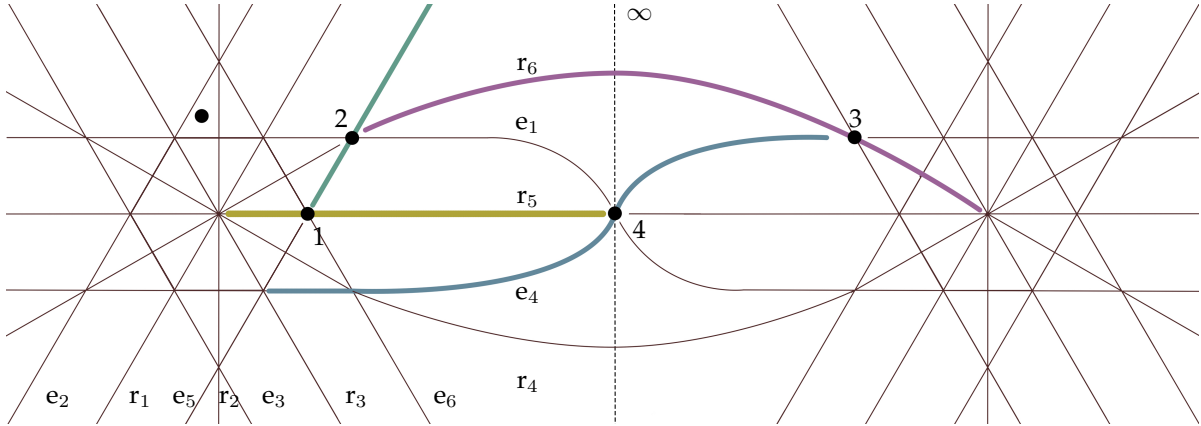

\begin{theorem}\label{thm:cn_f3}
The third family $\F_3(m)$ is sometimes congruence normal for $m \geq 17$. 
\end{theorem}

\begin{proof}
The proof is similar to that of Theorem~\ref{thm:cn_f2}. 
For $m \geq 17$, a four-cycle among shards still occurs, and its location relative to the base region is illustrated in Figure~\ref{fig:octagon} for $m = 17$. 
The line at infinity is included in these arrangements, and one of the intersection points in the cycle occurs in a rank-2 subarrangement that includes the hyperplane at infinity. 
Relative to the plane graph, the cycle involves the same description as a embedded cycle for the family~$\F_2$.
\end{proof}

\begin{figure}[H]
\begin{center}
\resizebox{\textwidth}{!}{
\begin{tikzpicture}
\node at (0,0) {\includegraphics[width=\textwidth]{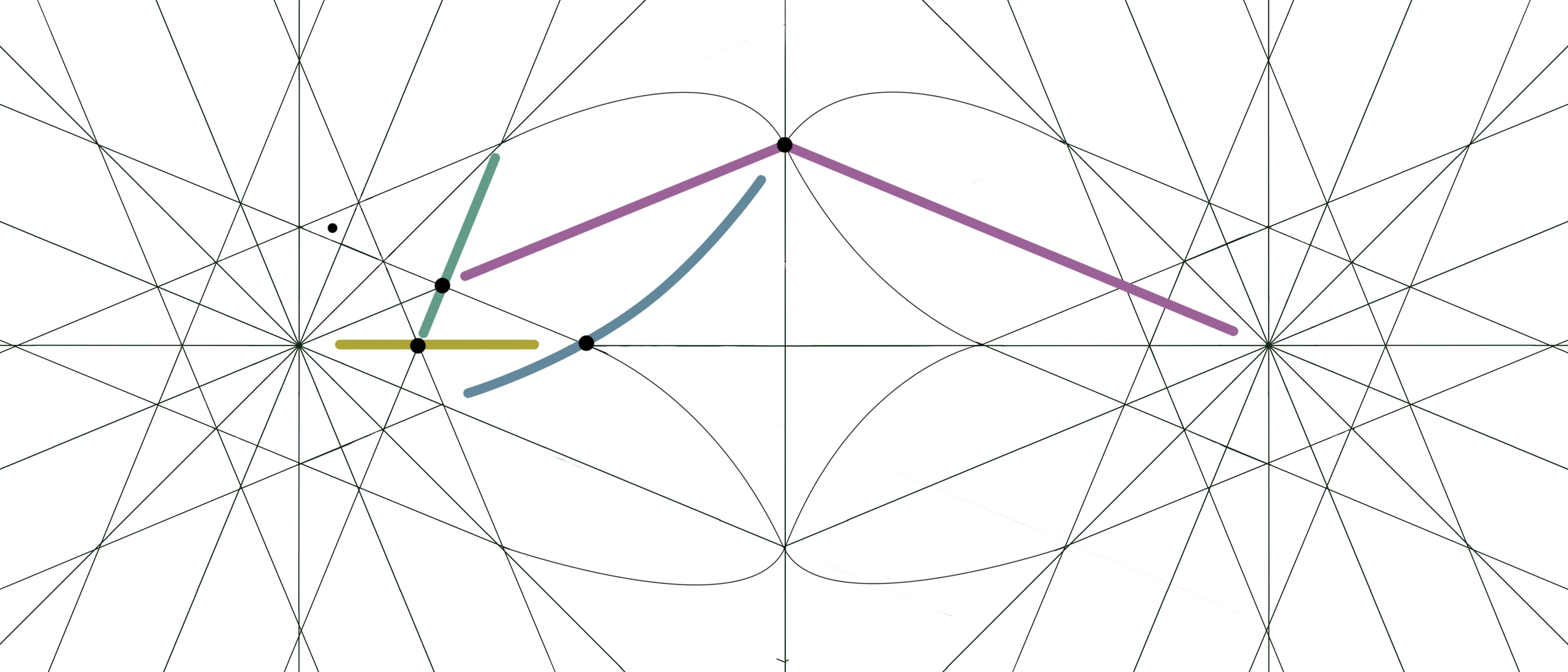}};
% mirror symmetries
\node at (-1,1.3) {r$_8$};
\node at (-1,0) {r$_7$};
\node at (-1,-1.7) {r$_6$};
\node at (-2.3,-2.5) {r$_5$};
\node at (-3.7,-2.5) {r$_4$};
\node at (-4.8,-2.5) {r$_3$};
\node at (-5.8,-2.5) {r$_2$};
\node at (-7.2,-2.5) {r$_1$};
% edges
\node at (-7.2, 2) {e$_1$};
\node at (-7.2,.3) {e$_2$};
\node at (-4.4,2.9) {e$_3$};
\node at (-7.2,2.9) {e$_4$};
\node at (-7.2,-.5) {e$_5$};
\node at (-7.2,-2) {e$_6$};
\node at (-2.6,2.9) {e$_7$};
\node at (-5.1,2.9) {e$_8$};
% points of intersection
\node at (-3.8,-.3) {1};
\node at (-3.6,.7)  {2};
\node at (.2,1.5)   {3};
\node at (-1.9,-.4)  {4};
% infinity
\node at (.3,.6)    {$\infty$};
\end{tikzpicture}}
\end{center}
\caption{The simplicial hyperplane arrangement $\A(17,160)_4$ from $\F_3$ whose lattice of regions with the marked base region is not congruence normal}
\label{fig:octagon}
\end{figure}
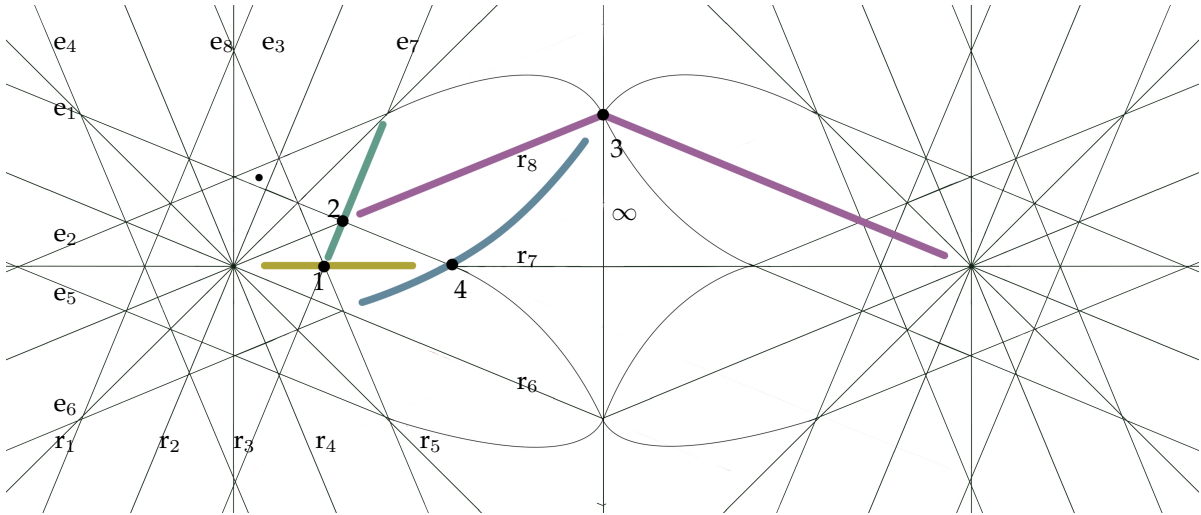

\subsection{Never \textbf{CN} simplicial arrangements}
\label{ssec:never}
Three of the known simplicial arrangements of rank 3 are never congruence normal, see Table~\ref{tab:nevercn}. 
That is, there is no choice of base region such that the lattice of regions is congruence normal.
The first arrangement is an arrangement with 22 hyperplanes with normals 
related to $\sqrt{5}$, see Figure~\ref{fig:not_cn_22}.
The second arrangement has 25 hyperplanes with normals related to $\sqrt{5}$ and is shown in Figure~\ref{fig:not_cn_25}.
The third arrangement is the new sporadic arrangement found in \cite{cuntz_greedy_2020}. 
It is the only known arrangement with $35$ hyperplanes and is illustrated in Figure~\ref{fig:not_cn_35}.
We are not aware of any geometric explanation for the provenance of these arrangements and why they are never congruence normal.

\begin{table}[H]
\begin{tabular}{c|c|c}
\multicolumn{3}{c}{Never congruence normal}\\\toprule[1pt]
$\A(22,288)$ & $\A(25,360)$ & $\A(35,680)$ \\
\end{tabular}
\caption{Simplicial arrangements that are never congruence normal}
\label{tab:nevercn}
\end{table}

\begin{figure}[H]
\begin{tikzpicture}
	[x={(3cm, 0cm)},
	 y={(0cm, 1.2cm)},
	% [x={(1.6cm, 0cm)},
	%  y={(0cm, 1.2cm)},
	 scale=1.5,
	 point/.style={circle,fill=black,inner sep=0.75pt},
	 line/.style={black!20,line width=0.5pt}]

\coordinate (0) at  (2.78235500444083, 2.86296737853363);
\coordinate (1) at  (2.90779740156158, 2.99204418227311);
\coordinate (2) at  (0.00000000000000, 0.00000000000000);
\coordinate (3) at  (2.50000000000000, 4.33012701892219);
\coordinate (4) at  (1.95864408835225, 3.39247107497052);
\coordinate (5) at  (3.81966011250105, 0.00000000000000);
\coordinate (6) at  (5.00000000000000, 0.00000000000000);
\coordinate (7) at  (3.07052017144396, 1.36561724787068);
\coordinate (8) at  (3.57469660655236, 1.58985027592800);
\coordinate (9) at  (3.09737880796660, 2.12227294276199);
\coordinate (10) at (3.05901699437495, 2.16506350946110);
\coordinate (11) at (3.26639215491083, 2.23807810429496);
\coordinate (12) at (3.14194292957320, 2.22375553916077);
\coordinate (13) at (3.01907639967691, 2.20961512267784);
\coordinate (14) at (3.12305898749054, 2.28571836357388);
\coordinate (15) at (3.05696480387029, 2.30768637070836);
\coordinate (16) at (3.10525430923287, 2.34413985987918);
\coordinate (17) at (2.96326962737348, 2.43759490491044);
\coordinate (18) at (2.98420344738630, 2.45481513101885);
\coordinate (19) at (3.00317939994615, 2.47042480920960);
\coordinate (20) at (3.17418082864579, 2.61109112167746);
\coordinate (21) at (3.38999105559213, 2.78861729235476);

\begin{scope}

\clip (1.85,-0.25) rectangle (4.25,3.45);

\draw[line] (1) -- (2);
\draw[line] (3) -- (7);
\draw[line] (4) -- (20);
\draw[line] (0) -- (5);
\draw[line] (0) -- (6);
\draw[line] (0) -- (8);
\draw[line] (0) -- (9);
\draw[line] (0) -- (12);
\draw[line] (0) -- (15);
\draw[line] (3) -- (5);
\draw[line] (4) -- (21);
\draw[line] (1) -- (6);
\draw[line] (1) -- (7);
\draw[line] (1) -- (8);
\draw[line] (1) -- (9);
\draw[line] (1) -- (10);
\draw[line] (1) -- (13);
\draw[line] (2) -- (3);
\draw[line] (2) -- (6);
\draw[line] (2) -- (8);
\draw[line] (2) -- (11);
\draw[line] (2) -- (12);
\draw[line] (2) -- (14);
\draw[line] (2) -- (16);
\draw[line] (2) -- (21);
\draw[line] (3) -- (6);
\draw[line] (3) -- (8);
\draw[line] (3) -- (9);
\draw[line] (3) -- (10);
\draw[line] (3) -- (13);
\draw[line] (4) -- (5);
\draw[line] (4) -- (6);
\draw[line] (4) -- (11);
\draw[line] (4) -- (12);
\draw[line] (4) -- (14);
\draw[line] (4) -- (16);
\draw[line] (5) -- (21);
\draw[line] (5) -- (18);
\draw[line] (5) -- (17);
\draw[line] (5) -- (20);
\draw[line] (5) -- (19);
\draw[line] (6) -- (17);
\draw[line] (6) -- (18);
\draw[line] (6) -- (19);
\draw[line] (7) -- (16);
\draw[line] (7) -- (15);
\draw[line] (7) -- (21);
\draw[line] (7) -- (20);
\draw[line] (7) -- (19);
\draw[line] (8) -- (18);
\draw[line] (8) -- (19);
\draw[line] (8) -- (17);
\draw[line] (9) -- (21);
\draw[line] (9) -- (20);
\draw[line] (10) -- (21);
\draw[line] (10) -- (20);
\draw[line] (11) -- (13);
\draw[line] (11) -- (15);
\draw[line] (11) -- (17);
\draw[line] (13) -- (20);
\draw[line] (13) -- (21);

\node[point]  at (0) {};
\node[point]  at (1) {};
\node[point]  at (2) {};
\node[point]  at (3) {};
\node[point]  at (4) {};
\node[point]  at (5) {};
\node[point]  at (6) {};
\node[point]  at (7) {};
\node[point]  at (8) {};
\node[point]  at (9) {};
\node[point]  at (10) {};
\node[point]  at (11) {};
\node[point]  at (12) {};
\node[point]  at (13) {};
\node[point]  at (14) {};
\node[point]  at (15) {};
\node[point]  at (16) {};
\node[point]  at (17) {};
\node[point]  at (18) {};
\node[point]  at (19) {};
\node[point]  at (20) {};
\node[point]  at (21) {};

\end{scope}

\end{tikzpicture}
\caption{The point configuration $\A(22,288)^*$. Three points are not shown and can be obtained by continuing the line segments. }
\label{fig:not_cn_22}
\end{figure}
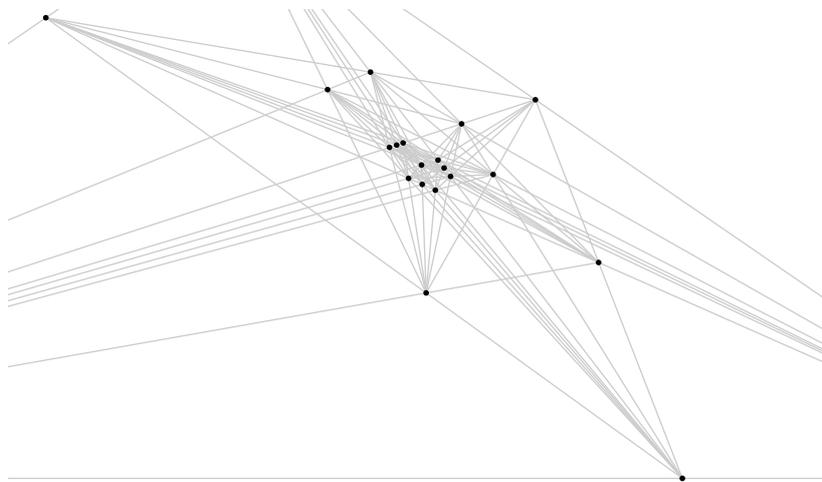

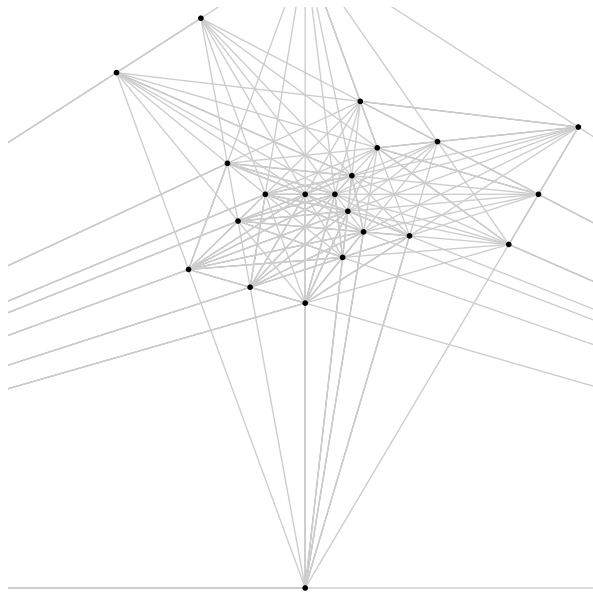
\begin{figure}[H]
\begin{tikzpicture}
	[x={(4cm, 0cm)},
	 y={(0cm, 1.5cm)},
	 scale=1.3,
	 point/.style={circle,fill=black,inner sep=0.75pt},
	 line/.style={black!20,line width=0.5pt}]

\coordinate (0) at  (2.23606797749979, 3.87298334620742);
\coordinate (1) at  (2.61803398874989, 2.80251707688815);
\coordinate (2) at  (2.76393202250021, 2.39363534581848);
\coordinate (3) at  (2.57514161979123, 2.67616567329817);
\coordinate (4) at  (2.60785001480277, 2.56071587030622);
\coordinate (5) at  (2.64754248593737, 2.42061459137964);
\coordinate (6) at  (2.50000000000000, 2.67616567329817);
\coordinate (7) at  (2.02254248593737, 3.50314634611018);
\coordinate (8) at  (2.30327668541684, 2.88675134594813);
\coordinate (9) at  (2.39918693812442, 2.67616567329817);
\coordinate (10) at (2.59463815113608, 2.24702255251213);
\coordinate (11) at (2.33002683322360, 2.49421513157466);
\coordinate (12) at (2.50000000000000, 1.93649167310371);
\coordinate (13) at (2.36067977499790, 2.04440865534831);
\coordinate (14) at (2.20491502812526, 2.16506350946110);
\coordinate (15) at (5.00000000000000, 0.00000000000000);
\coordinate (16) at (0.00000000000000, 0.00000000000000);
\coordinate (17) at (2.50000000000000, 0.00000000000000);
\coordinate (18) at (3.19098300562505, 3.13330934601295);
\coordinate (19) at (3.09016994374947, 2.67616567329817);
\coordinate (20) at (2.50000000000000, 4.33012701892219);
\coordinate (21) at (2.83457635340900, 3.03431837413435);
\coordinate (22) at (3.01502832395825, 2.33543089740679);
\coordinate (23) at (2.63932022500210, 3.30792269124804);
\coordinate (24) at (2.68237254218789, 2.99204418227311);

\begin{scope}

\clip (1.75,-0.15) rectangle (3.25,3.95);

\draw[line] (0) -- (1) -- (2);
\draw[line] (0) -- (3) -- (4) -- (5);
\draw[line] (0) -- (6) -- (10);
\draw[line] (0) -- (7) -- (16) -- (20);
\draw[line] (0) -- (8) -- (11) -- (13) -- (17);
\draw[line] (0) -- (9) -- (12);
\draw[line] (0) -- (15) -- (19) -- (21) -- (23);
\draw[line] (0) -- (22) -- (24);
\draw[line] (1) -- (3) -- (13) -- (24);
\draw[line] (1) -- (4) -- (10) -- (17) -- (23);
\draw[line] (1) -- (5) -- (20);
\draw[line] (1) -- (6) -- (11) -- (16) -- (21);
\draw[line] (1) -- (7) -- (15) -- (22);
\draw[line] (1) -- (8) -- (19);
\draw[line] (1) -- (9) -- (18);
\draw[line] (2) -- (3) -- (7);
\draw[line] (2) -- (4) -- (6) -- (8) -- (15);
\draw[line] (2) -- (5) -- (11) -- (22);
\draw[line] (2) -- (10) -- (13) -- (16) -- (19);
\draw[line] (2) -- (12) -- (18);
\draw[line] (2) -- (17) -- (21);
\draw[line] (2) -- (20) -- (23) -- (24);
\draw[line] (3) -- (6) -- (9) -- (19);
\draw[line] (3) -- (8) -- (22);
\draw[line] (3) -- (10) -- (20);
\draw[line] (3) -- (11) -- (18);
\draw[line] (3) -- (12) -- (23);
\draw[line] (3) -- (14) -- (21);
\draw[line] (4) -- (9) -- (22);
\draw[line] (4) -- (11) -- (19);
\draw[line] (4) -- (12) -- (24);
\draw[line] (4) -- (13) -- (21);
\draw[line] (4) -- (14) -- (16) -- (18);
\draw[line] (5) -- (6) -- (7);
\draw[line] (5) -- (9) -- (15);
\draw[line] (5) -- (10) -- (12) -- (21);
\draw[line] (5) -- (13) -- (18);
\draw[line] (5) -- (14) -- (19);
\draw[line] (5) -- (17) -- (24);
\draw[line] (6) -- (12) -- (17) -- (20);
\draw[line] (6) -- (13) -- (23);
\draw[line] (6) -- (14) -- (24);
\draw[line] (7) -- (8) -- (9) -- (10);
\draw[line] (7) -- (11) -- (12);
\draw[line] (7) -- (14) -- (17);
\draw[line] (7) -- (18) -- (23);
\draw[line] (7) -- (19) -- (24);
\draw[line] (8) -- (14) -- (20);
\draw[line] (8) -- (16) -- (23);
\draw[line] (8) -- (18) -- (21) -- (24);
\draw[line] (9) -- (11) -- (14) -- (23);
\draw[line] (9) -- (13) -- (20);
\draw[line] (9) -- (16) -- (24);
\draw[line] (10) -- (11) -- (15);
\draw[line] (10) -- (14) -- (22);
\draw[line] (12) -- (13) -- (14) -- (15);
\draw[line] (12) -- (16) -- (22);
\draw[line] (15) -- (16) -- (17);
\draw[line] (15) -- (18) -- (20);
\draw[line] (17) -- (18) -- (19) -- (22);
\draw[line] (20) -- (21) -- (22);

\node[point]  at (0) {};
\node[point]  at (1) {};
\node[point]  at (2) {};
\node[point]  at (3) {};
\node[point]  at (4) {};
\node[point]  at (5) {};
\node[point]  at (6) {};
\node[point]  at (7) {};
\node[point]  at (8) {};
\node[point]  at (9) {};
\node[point]  at (10) {};
\node[point]  at (11) {};
\node[point]  at (12) {};
\node[point]  at (13) {};
\node[point]  at (14) {};
\node[point]  at (15) {};
\node[point]  at (16) {};
\node[point]  at (17) {};
\node[point]  at (18) {};
\node[point]  at (19) {};
\node[point]  at (20) {};
\node[point]  at (21) {};
\node[point]  at (22) {};
\node[point]  at (23) {};
\node[point]  at (24) {};

\end{scope}

\end{tikzpicture}
\caption{The point configuration $\A(25,360)^*$. Three points are not shown and can be obtained by continuing the line segments.}
\label{fig:not_cn_25}
\end{figure}

\begin{figure}[H]
\begin{tikzpicture}
	[x={(1.6cm, 0cm)},
	 y={(0cm, 1cm)},
	 scale=3,
	 point/.style={circle,fill=black,inner sep=0.75pt},
	 line/.style={black!20,line width=0.5pt}]

\coordinate (0) at  (3.32768253095497, 1.74265003171565);
\coordinate (1) at  (1.87818397394008, 3.25311006882584);
\coordinate (2) at  (0.00000000000000, 0.00000000000000);
\coordinate (3) at  (2.15548365471666, 1.80238132300618);
\coordinate (4) at  (2.73751258720550, 1.43358819615121);
\coordinate (5) at  (5.00000000000000, 0.00000000000000);
\coordinate (6) at  (1.33921846697968, 2.31959442724331);
\coordinate (7) at  (2.40441358514015, 2.01053259167888);
\coordinate (8) at  (2.50000000000000, 1.78482593493031);
\coordinate (9) at  (2.50000000000000, 1.93649167310371);
\coordinate (10) at (2.96715395795394, 2.11834133483492);
\coordinate (11) at (2.69716897975906, 1.92559085838539);
\coordinate (12) at (2.33345049896733, 2.03112222947948);
\coordinate (13) at (2.30283102024094, 1.92559085838539);
\coordinate (14) at (2.12589566933614, 1.95717465274315);
\coordinate (15) at (2.18385340506974, 2.01053259167888);
\coordinate (16) at (2.24975151703417, 1.74265003171565);
\coordinate (17) at (2.39811018962345, 1.85756816535638);
\coordinate (18) at (2.41881234809646, 1.90488764453019);
\coordinate (19) at (2.32156306604063, 1.86703128316506);
\coordinate (20) at (2.50000000000000, 4.33012701892219);
\coordinate (21) at (2.50000000000000, 1.84497206462827);
\coordinate (22) at (3.83921846697968, 2.01053259167888);
\coordinate (23) at (2.37982380899965, 2.19094986840217);
\coordinate (24) at (2.52377486585338, 1.88615134022947);
\coordinate (25) at (2.65466206810598, 2.44397567869273);
\coordinate (26) at (2.67843693395936, 2.15403390019270);
\coordinate (27) at (3.07450730194470, 1.61006652449696);
\coordinate (28) at (1.62793549097425, 2.81966698181200);
\coordinate (29) at (2.44092434675387, 1.74265003171565);
\coordinate (30) at (2.36594504411141, 1.94390312412404);
\coordinate (31) at (2.27565984571700, 2.01053259167888);
\coordinate (32) at (2.16690099793465, 1.88615134022947);
\coordinate (33) at (2.44604987980490, 1.96714520168799);
\coordinate (34) at (2.35564148800271, 1.76049631439454);

\begin{scope}

\clip (1.3,1.3) rectangle (3.85,3.3);

\draw[line] (0) -- (1) -- (5) -- (10) -- (25);
\draw[line] (0) -- (2) -- (4) -- (22) -- (27);
\draw[line] (0) -- (3) -- (8);
\draw[line] (0) -- (6) -- (7) -- (11) -- (12);
\draw[line] (0) -- (9) -- (15);
\draw[line] (0) -- (13) -- (14) -- (18) -- (24);
\draw[line] (0) -- (16) -- (29);
\draw[line] (0) -- (17) -- (19) -- (21) -- (32);
\draw[line] (0) -- (26) -- (28);
\draw[line] (0) -- (31) -- (33);
\draw[line] (1) -- (2) -- (6) -- (20) -- (28);
\draw[line] (1) -- (3) -- (14);
\draw[line] (1) -- (4) -- (9) -- (23) -- (24);
\draw[line] (1) -- (7) -- (8);
\draw[line] (1) -- (12) -- (17) -- (29) -- (30);
\draw[line] (1) -- (13) -- (19) -- (31) -- (34);
\draw[line] (1) -- (15) -- (16);
\draw[line] (1) -- (21) -- (33);
\draw[line] (1) -- (26) -- (27);
\draw[line] (2) -- (3) -- (7) -- (13);
\draw[line] (2) -- (8) -- (10) -- (11) -- (29);
\draw[line] (2) -- (9) -- (16) -- (17);
\draw[line] (2) -- (12) -- (32);
\draw[line] (2) -- (14) -- (15) -- (23) -- (25);
\draw[line] (2) -- (19) -- (26) -- (33);
\draw[line] (2) -- (24) -- (34);
\draw[line] (3) -- (4) -- (5) -- (6) -- (16);
\draw[line] (3) -- (9) -- (10) -- (18) -- (19);
\draw[line] (3) -- (11) -- (17) -- (24);
\draw[line] (3) -- (12) -- (25);
\draw[line] (3) -- (15) -- (20) -- (32);
\draw[line] (3) -- (21) -- (22);
\draw[line] (3) -- (23) -- (31);
\draw[line] (3) -- (26) -- (30);
\draw[line] (3) -- (27) -- (29) -- (34);
\draw[line] (4) -- (7) -- (21);
\draw[line] (4) -- (8) -- (12) -- (18);
\draw[line] (4) -- (11) -- (20) -- (25) -- (26);
\draw[line] (4) -- (14) -- (34);
\draw[line] (4) -- (15) -- (19) -- (29);
\draw[line] (4) -- (17) -- (28) -- (31);
\draw[line] (5) -- (7) -- (9);
\draw[line] (5) -- (8) -- (13) -- (15) -- (17);
\draw[line] (5) -- (11) -- (23) -- (27) -- (28);
\draw[line] (5) -- (12) -- (24);
\draw[line] (5) -- (14) -- (29);
\draw[line] (5) -- (18) -- (21) -- (30) -- (31);
\draw[line] (5) -- (20) -- (22);
\draw[line] (5) -- (32) -- (34);
\draw[line] (6) -- (8) -- (14) -- (19);
\draw[line] (6) -- (9) -- (31);
\draw[line] (6) -- (10) -- (22) -- (23) -- (26);
\draw[line] (6) -- (13) -- (21) -- (27);
\draw[line] (6) -- (15) -- (24) -- (30);
\draw[line] (6) -- (29) -- (32);
\draw[line] (7) -- (10) -- (14);
\draw[line] (7) -- (15) -- (22) -- (31);
\draw[line] (7) -- (16) -- (19) -- (25) -- (30);
\draw[line] (7) -- (17) -- (20);
\draw[line] (7) -- (18) -- (23) -- (29);
\draw[line] (7) -- (24) -- (28) -- (33);
\draw[line] (7) -- (26) -- (32);
\draw[line] (8) -- (9) -- (20) -- (21);
\draw[line] (8) -- (16) -- (22) -- (34);
\draw[line] (8) -- (23) -- (33);
\draw[line] (8) -- (24) -- (25);
\draw[line] (8) -- (27) -- (32);
\draw[line] (8) -- (28) -- (30);
\draw[line] (9) -- (11) -- (14) -- (30);
\draw[line] (9) -- (12) -- (27) -- (33);
\draw[line] (9) -- (13) -- (22);
\draw[line] (9) -- (25) -- (29);
\draw[line] (9) -- (26) -- (34);
\draw[line] (10) -- (12) -- (15);
\draw[line] (10) -- (13) -- (30) -- (32) -- (33);
\draw[line] (10) -- (16) -- (24);
\draw[line] (10) -- (20) -- (27);
\draw[line] (10) -- (21) -- (34);
\draw[line] (11) -- (15) -- (33);
\draw[line] (11) -- (16) -- (21);
\draw[line] (11) -- (18) -- (22) -- (32);
\draw[line] (12) -- (13) -- (16) -- (23);
\draw[line] (12) -- (14) -- (26) -- (31);
\draw[line] (12) -- (19) -- (20);
\draw[line] (12) -- (21) -- (28);
\draw[line] (13) -- (28) -- (29);
\draw[line] (14) -- (16) -- (28) -- (32);
\draw[line] (14) -- (17) -- (27);
\draw[line] (14) -- (22) -- (33);
\draw[line] (15) -- (18) -- (27);
\draw[line] (15) -- (28) -- (34);
\draw[line] (16) -- (18) -- (26);
\draw[line] (16) -- (20) -- (31);
\draw[line] (17) -- (18) -- (25) -- (33) -- (34);
\draw[line] (19) -- (22) -- (24);
\draw[line] (20) -- (23) -- (30) -- (34);
\draw[line] (20) -- (29) -- (33);
\draw[line] (21) -- (24) -- (26) -- (29);
\draw[line] (22) -- (25) -- (28);
\draw[line] (24) -- (27) -- (31);
\draw[line] (25) -- (31) -- (32);

\node[point]  at (0) {};
\node[point]  at (1) {};
\node[point]  at (2) {};
\node[point]  at (3) {};
\node[point]  at (4) {};
\node[point]  at (5) {};
\node[point]  at (6) {};
\node[point]  at (7) {};
\node[point]  at (8) {};
\node[point]  at (9) {};
\node[point]  at (10) {};
\node[point]  at (11) {};
\node[point]  at (12) {};
\node[point]  at (13) {};
\node[point]  at (14) {};
\node[point]  at (15) {};
\node[point]  at (16) {};
\node[point]  at (17) {};
\node[point]  at (18) {};
\node[point]  at (19) {};
\node[point]  at (20) {};
\node[point]  at (21) {};
\node[point]  at (22) {};
\node[point]  at (23) {};
\node[point]  at (24) {};
\node[point]  at (25) {};
\node[point]  at (26) {};
\node[point]  at (27) {};
\node[point]  at (28) {};
\node[point]  at (29) {};
\node[point]  at (30) {};
\node[point]  at (31) {};
\node[point]  at (32) {};
\node[point]  at (33) {};
\node[point]  at (34) {};

\end{scope}

\end{tikzpicture}
\caption{The point configuration $\A(35,680)^*$. Three points are not shown and can be obtained by continuing the line segments.}
\label{fig:not_cn_35}
\end{figure}
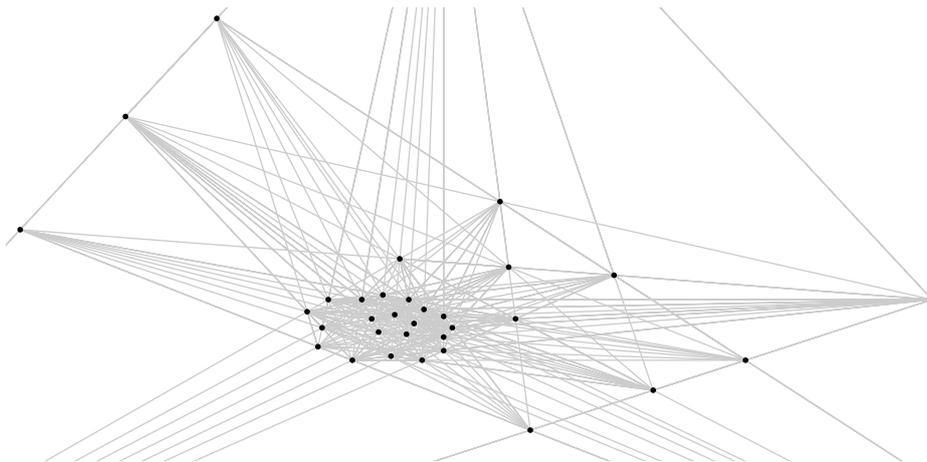

\subsection{Observations and consequences}
\label{ssec:discuss}

We make a few remarks on the verification and its implications.
The number of shards do not depend on the choice of base region:
indeed, \cite[Lemma~146]{padrol_shard_2020} says that in a simplicial arrangement, the number of shards is the number of rays in the arrangement minus the dimension. 
So, computing the number of shards leads to the number of facets of the corresponding simple zonotope.
For rank-$3$ simplicial arrangements, the number of shards is one less than half the number of regions.
For example, the arrangements $\A(30,480)$ and $\A(31,480)$ have different numbers of hyperplanes but the same number of shards and regions.

Finally, we end with questions that arose from this investigation.

\begin{description}
	\item[Question 1] What is the relationship between polygonal and semidistributive lattices?
	\item[Question 2] Is there a hyperplane arrangement with at most $8$ hyperplanes that yields a tight poset of regions which is not congruence normal?
        \item[Question 3] Is there a proof of congruence normality for $\A(31,480)$ using dualty with $H_3$?
	\item[Question 4] Is there a geometric explanation for the provenance of the three arrangements that are never congruence normal? Are the posets of regions all isomorphic? 
	\item[Question 5] Reading used ``signed subsets'' to describe when an edge occurs between two shards in type $A$ and $B$ \cite{reading_lattice_2004}.
	Can shard covectors be used in conjunction with positive roots to describe forcing on shards? 
        \item[Question 6] Apart from being dual to 2-neighborly, what can be said about the combinatorial types of the regions in a tight hyperplane arrangement?
\end{description}

\section{Invariants of rank-$3$ simplicial hyperplane arrangements}
\label{sec:invariants}

Table~\ref{tab:invariants} gives a list of invariants for the simplicial hyperplane arrangements of rank $3$ with at most $37$ hyperplanes (excluding the reducible near-pencil arrangements).

The $f$-vector consists of the numbers of $0$-, $1$-, and $2$-cells in the corresponding CW-complex, i.e.\ $f_3$ is the number of chambers.
The $t$-vector contains the numbers $t_i$ of vertices which lie on exactly $i$ lines;
the $r$-vector contains the numbers $r_i$ of lines on which exactly $i$ vertices lie.
The automorphism groups of the CW-complexes are listed in the column AG, the automorphism groups of the matroids are listed in the column AGM; a pair $(a,b)$ represents the $b$-th group of order $a$ in the database of small groups (as included for example in the system \texttt{GAP} \cite{GAP4}).
In the column EXP we list the roots of the characteristic polynomials of the arrangements when they are all integers.
The column ``domain'' contains the minimal field of definition for a realization of the matroid of the arrangement in characteristic zero (as computed in \cite{p-C10b}): $\mathbb{Z}$ stands for a crystallographic arrangement (which defines a Weyl groupoid). The other domains are
\begin{eqnarray*}
Q_i &:=& \mathbb{Q}(\zeta+\zeta^{-1}), \quad \zeta \text{ an } i\text{-th root of unity},\\
K_1 &:=& \mathbb{Q}[X]/(X^3-X+1), \\
K_2 &:=& \mathbb{Q}[X]/(X^4 - 3X^3 + 3X^2 - 3X + 1).
\end{eqnarray*}

\clearpage
{\tiny
\begin{longtable}{llllllll}
Name & $f$-vector & $t$-vector & $r$-vector & AG & AGM & EXP & domain \\\toprule[1pt]
$\A{(6,24)}=\AG(6,1)$ & (7,18,12) & $(3,4)$ & $(0,6)$ & (48,48) & (24,12) & [ 1, 2, 3 ] & $\mathbb{Z}$ \\
\hline
$\A{(7,32)}=\AG(7,1)$ & (9,24,16) & $(3,6)$ & $(0,4,3)$ & (48,48) & (24,12) & [ 1, 3, 3 ] & $\mathbb{Z}$ \\
\hline
$\A{(8,40)}=\AG(8,1)$ & (11,30,20) & $(4,6,1)$ & $(0,2,6)$ & (16,11) & (8,3) & [ 1, 3, 4 ] & $\mathbb{Z}$ \\
\hline
$\A{(9,48)}=\AG(9,1)$ & (13,36,24) & $(6,4,3)$ & $(0^{2},9)$ & (48,48) & (24,12) & [ 1, 3, 5 ] & $\mathbb{Z}$ \\
\hline
$\A{(10,60)_1}=\AG(10,3)$ & (16,45,30) & $(6,7,3)$ & $(0,1,3,6)$ & (24,14) & (12,4) & [ 1, 4, 5 ] & $\mathbb{Z}$ \\
$\A{(10,60)_2}=\AG(10,2)$ & (16,45,30) & $(6,7,3)$ & $(0^{2},6,3,1)$ & (12,4) & (6,1) & [ 1, 4, 5 ] & $\mathbb{Z}$ \\
$\A{(10,60)_3}=\AG(10,1)$ & (16,45,30) & $(5,10,0,1)$ & $(0^{2},5^{2})$ & (20,4) & (20,3) & [ 1, 4, 5 ] & $Q_{5}$ \\
\hline
$\A{(11,72)}=\AG(11,1)$ & (19,54,36) & $(7,8,4)$ & $(0^{2},4^{2},3)$ & (8,5) & (4,2) & [ 1, 5, 5 ] & $\mathbb{Z}$ \\
\hline
$\A{(12,84)_1}=\AG(12,2)$ & (22,63,42) & $(8,10,3,1)$ & $(0^{2},3^{2},6)$ & (8,5) & (4,2) & [ 1, 5, 6 ] & $\mathbb{Z}$ \\
$\A{(12,84)_2}=\AG(12,3)$ & (22,63,42) & $(9,7,6)$ & $(0^{2},3^{2},6)$ & (12,4) & (6,1) & [ 1, 5, 6 ] & $\mathbb{Z}$ \\
$\A{(12,84)_3}=\AG(12,1)$ & (22,63,42) & $(6,15,0^{2},1)$ & $(0^{2},3^{2},6)$ & (24,14) & (12,4) & [ 1, 5, 6 ] & $\mathbb{Q}$ \\
\hline
$\A{(13,96)_1}=\AG(13,1)$ & (25,72,48) & $(9,12,3,0,1)$ & $(0^{2},3,0,10)$ & (24,14) & (12,4) & [ 1, 5, 7 ] & $\mathbb{Z}$ \\
$\A{(13,96)_2}=\AG(13,3)$ & (25,72,48) & $(10^{2},3,2)$ & $(0^{2},1,4,8)$ & (8,5) & (4,2) & [ 1, 5, 7 ] & $\mathbb{Z}$ \\
$\A{(13,96)_3}=\AG(13,2)$ & (25,72,48) & $(12,4,9)$ & $(0^{2},3,0,10)$ & (48,48) & (24,12) & [ 1, 5, 7 ] & $\mathbb{Z}$ \\
$\A{(13,104)}=\AG(13,4)$ & (27,78,52) & $(6,18,3)$ & $(0^{4},13)$ & (24,13) & (24,12) & [] & $Q_{5}$ \\
\hline
$\A{(14,112)_1}=\AG(14,2)$ & (29,84,56) & $(11,12,4,2)$ & $(0^{2},1,4^{3},1)$ & (4,2) & (2,1) & [ 1, 6, 7 ] & $\mathbb{Z}$ \\
$\A{(14,112)_2}=\AG(14,1)$ & (29,84,56) & $(7,21,0^{3},1)$ & $(0^{3},7,0,7)$ & (28,3) & (42,1) & [ 1, 6, 7 ] & $Q_{7}$ \\
$\A{(14,112)_3}=\AG(14,4)$ & (29,84,56) & $(10,14,4,0,1)$ & $(0^{3},4,6,4)$ & (8,5) & (8,3) & [ 1, 6, 7 ] & $Q_{5}$ \\
$\A{(14,116)}=\AG(14,3)$ & (30,87,58) & $(9,16,4,1)$ & $(0^{4},11,3)$ & (4,2) & (4,2) & [] & $Q_{5}$ \\
\hline
$\A{(15,120)}=\AG(15,1)$ & (31,90,60) & $(15,10,0,6)$ & $(0^{4},15)$ & (120,35) & (120,34) & [ 1, 5, 9 ] & $Q_{5}$ \\
$\A{(15,128)_1}=\AG(15,2)$ & (33,96,64) & $(13,12,6,2)$ & $(0^{2},1,4,2,4^{2})$ & (16,11) & (8,3) & [ 1, 7, 7 ] & $\mathbb{Z}$ \\
$\A{(15,128)_2}=\AG(15,4)$ & (33,96,64) & $(12,14,6,0,1)$ & $(0^{4},10,4,1)$ & (8,5) & (8,3) & [ 1, 7, 7 ] & $Q_{5}$ \\
$\A{(15,132)_1}=\AG(15,5)$ & (34,99,66) & $(9,22,0,3)$ & $(0^{4},9,3^{2})$ & (12,4) & (6,1) & [] & $K_1$ \\
$\A{(15,132)_2}=\AG(15,3)$ & (34,99,66) & $(12,13,9)$ & $(0^{4},9,3^{2})$ & (12,4) & (6,1) & [] & $\mathbb{Q}$ \\
\hline
$\A{(16,140)}=\AG(16,4)$ & (36,105,70) & $(15^{2},0,6)$ & $(0^{4},10,5,0^{2},1)$ & (20,4) & (20,3) & [ 1, 6, 9 ] & $Q_{5}$ \\
$\A{(16,144)_1}=\AG(16,2)$ & (37,108,72) & $(14,15,6,1^{2})$ & $(0^{2},1,2,4,2,7)$ & (8,5) & (4,2) & [ 1, 7, 8 ] & $\mathbb{Z}$ \\
$\A{(16,144)_2}=\AG(16,3)$ & (37,108,72) & $(15,13,6,3)$ & $(0^{4},10,0,6)$ & (12,4) & (6,1) & [ 1, 7, 8 ] & $\mathbb{Z}$ \\
$\A{(16,144)_3}=\AG(16,1)$ & (37,108,72) & $(8,28,0^{4},1)$ & $(0^{3},4^{2},0,8)$ & (32,39) & (32,43) & [ 1, 7, 8 ] & $Q_{8}$ \\
$\A{(16,144)_4}=\AG(16,6)$ & (37,108,72) & $(15,12,9,0,1)$ & $(0^{4},7,6,3)$ & (12,4) & (6,1) & [ 1, 7, 8 ] & $\mathbb{Q}$ \\
$\A{(16,144)_5}=\AG(16,5)$ & (37,108,72) & $(14,16,3,4)$ & $(0^{3},2,4,8,0,2)$ & (8,5) & (8,3) & [ 1, 7, 8 ] & $Q_{5}$ \\
$\A{(16,148)}=\AG(16,7)$ & (38,111,74) & $(12,19,6,0,1)$ & $(0^{3},3^{2},2,8)$ & (8,5) & (4,2) & [] & $\mathbb{Q}$ \\
\hline
$\A{(17,160)_1}=\AG(17,2)$ & (41,120,80) & $(16^{2},7,0,2)$ & $(0^{2},1,0,6,0,10)$ & (16,11) & (8,3) & [ 1, 7, 9 ] & $\mathbb{Z}$ \\
$\A{(17,160)_2}=\AG(17,4)$ & (41,120,80) & $(16^{2},7,0,2)$ & $(0^{2},1,0,6,0,10)$ & (16,11) & (8,3) & [ 1, 7, 9 ] & $\mathbb{Z}$ \\
$\A{(17,160)_3}=\AG(17,3)$ & (41,120,80) & $(18,12,7,4)$ & $(0^{4},8,0,9)$ & (16,11) & (8,3) & [ 1, 7, 9 ] & $\mathbb{Z}$ \\
$\A{(17,160)_4}=\AG(17,1)$ & (41,120,80) & $(12,24,4,0^{3},1)$ & $(0^{4},8,0,9)$ & (32,39) & (32,43) & [ 1, 7, 9 ] & $Q_{8}$ \\
$\A{(17,160)_5}=\AG(17,5)$ & (41,120,80) & $(16,18,1,6)$ & $(0^{4},6,8,1,0,2)$ & (8,5) & (8,3) & [ 1, 7, 9 ] & $Q_{5}$ \\
$\A{(17,164)}=\AG(17,6)$ & (42,123,82) & $(16,15,10,0,1)$ & $(0^{4},6,3,7,0,1)$ & (4,2) & (2,1) & [] & $\mathbb{Q}$ \\
$\A{(17,168)_1}=\AG(17,7)$ & (43,126,84) & $(13,22,7,0,1)$ & $(0^{4},6,0,10,0,1)$ & (8,5) & (4,2) & [] & $\mathbb{Q}$ \\
$\A{(17,168)_2}=\AG(17,8)$ & (43,126,84) & $(14,20,7,2)$ & $(0^{4},1,8^{2})$ & (8,5) & (8,3) & [] & $Q_{8}$ \\
\hline
$\A{(18,180)_1}=\AG(18,7)$ & (46,135,90) & $(18^{2},6,3,1)$ & $(0^{3},3^{2},0,6^{2})$ & (12,4) & (6,1) & [ 1, 8, 9 ] & $\mathbb{Z}$ \\
$\A{(18,180)_2}=\AG(18,3)$ & (46,135,90) & $(19,16,6,5)$ & $(0^{4},6,2,6,3,1)$ & (4,2) & (2,1) & [ 1, 8, 9 ] & $\mathbb{Z}$ \\
$\A{(18,180)_3}=\AG(18,1)$ & (46,135,90) & $(9,36,0^{5},1)$ & $(0^{4},9,0^{2},9)$ & (36,4) & (54,6) & [ 1, 8, 9 ] & $Q_{9}$ \\
$\A{(18,180)_4}=\AG(18,2)$ & (46,135,90) & $(18^{2},6,3,1)$ & $(0^{4},3^{2},12)$ & (24,14) & (12,4) & [ 1, 8, 9 ] & $\mathbb{Q}$ \\
$\A{(18,180)_5}=\AG(18,4)$ & (46,135,90) & $(18,19,3,6)$ & $(0^{4},3,9,3,0,3)$ & (12,4) & (6,1) & [ 1, 8, 9 ] & $Q_{5}$ \\
$\A{(18,180)_6}=\AG(18,5)$ & (46,135,90) & $(18,19,3,6)$ & $(0^{4},3,9,3,0,3)$ & (12,4) & (6,1) & [ 1, 8, 9 ] & $Q_{5}$ \\
$\A{(18,184)_1}=\AG(18,6)$ & (47,138,92) & $(18,16,12,0,1)$ & $(0^{4},5,2,7,2^{2})$ & (4,2) & (2,1) & [] & $\mathbb{Q}$ \\
$\A{(18,184)_2}=\AG(18,8)$ & (47,138,92) & $(16,22,6,2,1)$ & $(0^{4},6,0,7,4,1)$ & (4,2) & (2,1) & [] & $\mathbb{Q}$ \\
\hline
$\A{(19,192)_1}=\AG(19,1)$ & (49,144,96) & $(21,18,6,0,4)$ & $(0^{4},4,0,15)$ & (24,14) & (12,4) & [ 1, 7, 11 ] & $\mathbb{Z}$ \\
$\A{(19,192)_2}=\AG(19,3)$ & (49,144,96) & $(24,12,6^{2},1)$ & $(0^{4},4,0,15)$ & (24,14) & (12,4) & [ 1, 7, 11 ] & $\mathbb{Z}$ \\
$\A{(19,200)_1}=\AG(19,4)$ & (51,150,100) & $(20^{2},6,4,1)$ & $(0^{4},4^{4},3)$ & (8,5) & (4,2) & [ 1, 9, 9 ] & $\mathbb{Z}$ \\
$\A{(19,200)_2}=\AG(19,5)$ & (51,150,100) & $(20^{2},6,4,1)$ & $(0^{4},4^{4},3)$ & (8,2) & (4,1) & [ 1, 9, 9 ] & $\mathbb{Z}$ \\
$\A{(19,200)_3}=\AG(19,6)$ & (51,150,100) & $(20^{2},6,4,1)$ & $(0^{4},6,0,6,4,3)$ & (8,5) & (4,2) & [ 1, 9, 9 ] & $\mathbb{Z}$ \\
$\A{(19,200)_4}=\AG(19,2)$ & (51,150,100) & $(21,18,6^{2})$ & $(0^{4},1,8,6,0,4)$ & (8,5) & (8,3) & [ 1, 9, 9 ] & $Q_{5}$ \\
$\A{(19,204)}=\AG(19,7)$ & (52,153,102) & $(21,15^{2},0,1)$ & $(0^{4},4,3^{2},6,3)$ & (12,4) & (6,1) & [] & $\mathbb{Q}$ \\
\hline
$\A{(20,216)}=\AG(20,5)$ & (55,162,108) & $(20,26,4^{2},0^{2},1)$ & $(0^{3},2^{2},0,4,12)$ & (16,11) & (8,3) & [ 1, 8, 11 ] & $\mathbb{Z}$ \\
$\A{(20,220)_1}=\AG(20,3)$ & (56,165,110) & $(21,24,6,4,0,1)$ & $(0^{4},4,2,4,6,3,1)$ & (4,2) & (2,1) & [ 1, 9, 10 ] & $\mathbb{Z}$ \\
$\A{(20,220)_2}=\AG(20,4)$ & (56,165,110) & $(23,20,7,5,1)$ & $(0^{4},5,1,4^{2},6)$ & (4,2) & (2,1) & [ 1, 9, 10 ] & $\mathbb{Z}$ \\
$\A{(20,220)_3}=\AG(20,1)$ & (56,165,110) & $(10,45,0^{6},1)$ & $(0^{4},5^{2},0^{2},10)$ & (40,13) & (40,12) & [ 1, 9, 10 ] & $Q_{5}$ \\
$\A{(20,220)_4}=\AG(20,2)$ & (56,165,110) & $(25,15,10,6)$ & $(0^{5},5,10,0,5)$ & (20,4) & (20,3) & [ 1, 9, 10 ] & $Q_{5}$ \\
\hline
$\A{(21,240)_1}=\AG(21,4)$ & (61,180,120) & $(22,28,6,4,0^{2},1)$ & $(0^{4},4,0,4,8,4,0,1)$ & (16,11) & (8,3) & [ 1, 9, 11 ] & $\mathbb{Z}$ \\
$\A{(21,240)_2}=\AG(21,5)$ & (61,180,120) & $(26,20,9,4,2)$ & $(0^{4},5,0,3,4,9)$ & (8,5) & (4,2) & [ 1, 9, 11 ] & $\mathbb{Z}$ \\
$\A{(21,240)_3}=\AG(21,3)$ & (61,180,120) & $(24^{2},9,0,4)$ & $(0^{4},6,0,3,0,12)$ & (48,48) & (24,12) & [ 1, 9, 11 ] & $\mathbb{Z}$ \\
$\A{(21,240)_4}=\AG(21,1)$ & (61,180,120) & $(15,40,5,0^{5},1)$ & $(0^{4},5,0,5,0,11)$ & (40,13) & (40,12) & [ 1, 9, 11 ] & $Q_{5}$ \\
$\A{(21,240)_5}=\AG(21,2)$ & (61,180,120) & $(30,10,15,6)$ & $(0^{6},15,0,6)$ & (120,35) & (120,34) & [ 1, 9, 11 ] & $Q_{5}$ \\
$\A{(21,248)}=\AG(21,6)$ & (63,186,124) & $(25,20,15,2,1)$ & $(0^{4},1,0,11,0,8,0,1)$ & (8,5) & (4,2) & [] & $\mathbb{Q}$ \\
$\A{(21,252)}=\AG(21,7)$ & (64,189,126) & $(24,22,15,3)$ & $(0^{6},12,0,6,3)$ & (12,4) & (6,1) & [] & $K_1$ \\
\hline
$\A{(22,264)_1}=\AG(22,4)$ & (67,198,132) & $(27,25,9,3^{2})$ & $(0^{4},4,0,6,0,6^{2})$ & (12,4) & (6,1) & [ 1, 10, 11 ] & $\mathbb{Z}$ \\
$\A{(22,264)_2}=\AG(22,1)$ & (67,198,132) & $(11,55,0^{7},1)$ & $(0^{5},11,0^{3},11)$ & (44,3) & (110,1) & [ 1, 10, 11 ] & $Q_{11}$ \\
$\A{(22,264)_3}=\AG(22,3)$ & (67,198,132) & $(27,28,0,12)$ & $(0^{6},12,0,9,0,1)$ & (12,4) & (12,4) & [ 1, 10, 11 ] & $Q_{5}$ \\
$\A{(22,276)}=\AG(22,2)$ & (70,207,138) & $(24,30,12,3,1)$ & $(0^{4},1,0,6,3,9,0,3)$ & (12,4) & (6,1) & [] & $\mathbb{Q}$ \\
$\A{(22,288)}=\AG(22,5)$ & (73,216,144) & $(12,58,0^{2},3)$ & $(0^{7},12,6,0,4)$ & (48,48) & (48,48) & [] & $Q_{5}$ \\
\hline
$\A{(23,296)}=\AG(23,1)$ & (75,222,148) & $(27,32,10,4,2)$ & $(0^{4},1,0,6,2,7,4,3)$ & (4,2) & (2,1) & [] & $\mathbb{Q}$ \\
$\A{(23,304)}=\AG(23,2)$ & (77,228,152) & $(16,56,2,0,1,2)$ & $(0^{6},1,8,10,0,4)$ & (16,11) & (16,11) & [] & $Q_{5}$ \\
\hline
$\A{(24,304)}=\AG(24,2)$ & (77,228,152) & $(32^{2},0,12,0^{2},1)$ & $(0^{5},4,0^{2},20)$ & (32,39) & (32,43) & [] & $Q_{8}$ \\
$\A{(24,312)}=\AG(24,1)$ & (79,234,156) & $(12,66,0^{8},1)$ & $(0^{5},6^{2},0^{3},12)$ & (48,36) & (48,38) & [ 1, 11, 12 ] & $Q_{12}$ \\
$\A{(24,316)}=\AG(24,3)$ & (80,237,158) & $(31,32,9,5,3)$ & $(0^{4},1,0,6,1,6^{2},4)$ & (4,2) & (2,1) & [] & $\mathbb{Q}$ \\
$\A{(24,320)}=\AG(24,4)$ & (81,240,160) & $(20,54,4,0^{2},2,1)$ & $(0^{6},2,4,14,0,4)$ & (16,11) & (16,11) & [] & $Q_{5}$ \\
\hline
$\A{(25,320)}=\AG(25,5)$ & (81,240,160) & $(36,32,0,8,4,0,1)$ & $(0^{6},5,0,20)$ & (32,39) & (32,43) & [ 1, 9, 15 ] & $Q_{8}$ \\
$\A{(25,336)_1}=\AG(25,7)$ & (85,252,168) & $(33,34,12,2,3,0,1)$ & $(0^{4},2,0,4^{3},0,11)$ & (8,5) & (4,2) & [ 1, 11, 13 ] & $\mathbb{Z}$ \\
$\A{(25,336)_2}=\AG(25,4)$ & (85,252,168) & $(36,30,9,6,4)$ & $(0^{4},1,0,9,0,3,0,12)$ & (24,14) & (12,4) & [ 1, 11, 13 ] & $\mathbb{Z}$ \\
$\A{(25,336)_3}=\AG(25,6)$ & (85,252,168) & $(36,30,9,6,4)$ & $(0^{4},1,0,6,0,6^{3})$ & (12,4) & (6,1) & [ 1, 11, 13 ] & $\mathbb{Z}$ \\
$\A{(25,336)_4}=\AG(25,2)$ & (85,252,168) & $(36,28,15,0,6)$ & $(0^{4},4,0,3,0,6,0,12)$ & (48,48) & (24,12) & [ 1, 11, 13 ] & $\mathbb{Z}$ \\
$\A{(25,336)_5}=\AG(25,1)$ & (85,252,168) & $(18,60,6,0^{7},1)$ & $(0^{6},12,0^{3},13)$ & (48,36) & (48,38) & [ 1, 11, 13 ] & $Q_{12}$ \\
$\A{(25,336)_6}=\AG(25,8)$ & (85,252,168) & $(24,52,6,0^{3},3)$ & $(0^{6},3,0,18,0,4)$ & (48,48) & (48,48) & [ 1, 11, 13 ] & $Q_{5}$ \\
$\A{(25,360)}=\AG(25,3)$ & (91,270,180) & $(30,40,15,6)$ & $(0^{8},15,0,10)$ & (120,35) & (120,34) & [] & $Q_{5}$ \\
\hline
$\A{(26,364)_1}=\AG(26,4)$ & (92,273,182) & $(35,39,10,4,3,0,1)$ & $(0^{4},1^{2},4^{2},2^{2},7,4,1)$ & (4,2) & (2,1) & [ 1, 12, 13 ] & $\mathbb{Z}$ \\
$\A{(26,364)_2}=\AG(26,3)$ & (92,273,182) & $(37,36,9,6,3,1)$ & $(0^{4},1,0,7,2^{2},1,8,4,1)$ & (4,2) & (2,1) & [ 1, 12, 13 ] & $\mathbb{Z}$ \\
$\A{(26,364)_3}=\AG(26,1)$ & (92,273,182) & $(13,78,0^{9},1)$ & $(0^{6},13,0^{4},13)$ & (52,4) & (156,7) & [ 1, 12, 13 ] & $Q_{13}$ \\
$\A{(26,380)}=\AG(26,2)$ & (96,285,190) & $(35,40,10,11)$ & $(0^{8},11,5,10)$ & (20,4) & (20,3) & [] & $Q_{5}$ \\
\hline
$\A{(27,392)_1}=\AG(27,4)$ & (99,294,196) & $(38,42,9,6,3,0,1)$ & $(0^{4},1,0,5,4,2,0,7,4^{2})$ & (8,5) & (4,2) & [ 1, 13, 13 ] & $\mathbb{Z}$ \\
$\A{(27,392)_2}=\AG(27,3)$ & (99,294,196) & $(39,40,10,6,2^{2})$ & $(0^{4},1,0,6,2^{3},5,6,3)$ & (4,2) & (2,1) & [ 1, 13, 13 ] & $\mathbb{Z}$ \\
$\A{(27,392)_3}=\AG(27,2)$ & (99,294,196) & $(39,40,10,6,2^{2})$ & $(0^{4},1,0,5,4,1,2,4,8,2)$ & (4,2) & (2,1) & [ 1, 13, 13 ] & $\mathbb{Z}$ \\
$\A{(27,400)}=\AG(27,1)$ & (101,300,200) & $(40^{2},6,14,1)$ & $(0^{8},8^{2},11)$ & (8,5) & (8,3) & [] & $Q_{5}$ \\
\hline
$\A{(28,420)_1}=\AG(28,4)$ & (106,315,210) & $(41,44,11,6,2,1^{2})$ & $(0^{4},1,0,4^{2},2,1,4,6^{2})$ & (2,1) & (1,1) & [ 1, 13, 14 ] & $\mathbb{Z}$ \\
$\A{(28,420)_2}=\AG(28,5)$ & (106,315,210) & $(42^{2},12,6,1,3)$ & $(0^{4},1,0,4^{2},1,3,1,10,4)$ & (4,2) & (2,1) & [ 1, 13, 14 ] & $\mathbb{Z}$ \\
$\A{(28,420)_3}=\AG(28,6)$ & (106,315,210) & $(42^{2},12,6,1,3)$ & $(0^{4},1,0,6,0,3^{3},6^{2})$ & (12,4) & (6,1) & [ 1, 13, 14 ] & $\mathbb{Z}$ \\
$\A{(28,420)_4}=\AG(28,1)$ & (106,315,210) & $(14,91,0^{10},1)$ & $(0^{6},7^{2},0^{4},14)$ & (56,12) & (84,7) & [ 1, 13, 14 ] & $Q_{7}$ \\
$\A{(28,420)_5}=\AG(28,2)$ & (106,315,210) & $(45,40,3,15,3)$ & $(0^{8},6,9,13)$ & (12,4) & (6,1) & [ 1, 13, 14 ] & $Q_{5}$ \\
$\A{(28,420)_6}=\AG(28,3)$ & (106,315,210) & $(45,40,3,15,3)$ & $(0^{8},6,9,13)$ & (12,4) & (6,1) & [ 1, 13, 14 ] & $Q_{5}$ \\
\hline
$\A{(29,440)}=\AG(29,2)$ & (111,330,220) & $(50,40,1,14,6)$ & $(0^{8},5,8,16)$ & (8,5) & (8,3) & [] & $Q_{5}$ \\
$\A{(29,448)_1}=\AG(29,3)$ & (113,336,224) & $(44,46,13,6,2,0,2)$ & $(0^{4},1,0,3,4,3,0,4^{2},10)$ & (8,5) & (4,2) & [ 1, 13, 15 ] & $\mathbb{Z}$ \\
$\A{(29,448)_2}=\AG(29,4)$ & (113,336,224) & $(45,44,14,6,1,2,1)$ & $(0^{4},1,0,3,4,2^{2},1,8^{2})$ & (4,2) & (2,1) & [ 1, 13, 15 ] & $\mathbb{Z}$ \\
$\A{(29,448)_3}=\AG(29,5)$ & (113,336,224) & $(45,44,14,6,1,2,1)$ & $(0^{4},1,0,4,2,3,2^{2},6,9)$ & (4,2) & (2,1) & [ 1, 13, 15 ] & $\mathbb{Z}$ \\
$\A{(29,448)_4}=\AG(29,1)$ & (113,336,224) & $(21,84,7,0^{9},1)$ & $(0^{6},7,0,7,0^{3},15)$ & (56,12) & (84,7) & [ 1, 13, 15 ] & $Q_{7}$ \\
\hline
$\A{(30,460)}=\AG(30,2)$ & (116,345,230) & $(55,40,0,11,10)$ & $(0^{8},5^{2},20)$ & (20,4) & (20,3) & [] & $Q_{5}$ \\
$\A{(30,476)}=\AG(30,3)$ & (120,357,238) & $(49,44,17,6,1^{2},2)$ & $(0^{4},1,0,3,2,4,1,2,4,13)$ & (4,2) & (2,1) & [ 1, 13, 16 ] & $\mathbb{Z}$ \\
$\A{(30,480)}=\AG(30,1)$ & (121,360,240) & $(15,105,0^{11},1)$ & $(0^{7},15,0^{5},15)$ & (60,12) & (120,36) & [ 1, 14, 15 ] & $Q_{15}$ \\
\hline
$\A{(31,480)}=\AG(31,1)$ & (121,360,240) & $(60,40,0,6,15)$ & $(0^{8},6,0,25)$ & (120,35) & (120,34) & [ 1, 11, 19 ] & $Q_{5}$ \\
$\A{(31,504)_1}=\AG(31,2)$ & (127,378,252) & $(54,42,21,6,1,0,3)$ & $(0^{4},1,0^{3},9,0,6,0,15)$ & (24,14) & (12,4) & [ 1, 13, 17 ] & $\mathbb{Z}$ \\
$\A{(31,504)_2}=\AG(31,3)$ & (127,378,252) & $(54,42,21,6,1,0,3)$ & $(0^{4},1,0,3,0,6,0,3,0,18)$ & (24,14) & (12,4) & [ 1, 13, 17 ] & $\mathbb{Z}$ \\
\hline
$\A{(32,544)}=\AG(32,1)$ & (137,408,272) & $(16,120,0^{12},1)$ & $(0^{7},8^{2},0^{5},16)$ & (64,186) & (128,913) & [ 1, 15, 16 ] & $Q_{16}$ \\
\hline
$\A{(33,576)}=\AG(33,1)$ & (145,432,288) & $(24,112,8,0^{11},1)$ & $(0^{8},16,0^{5},17)$ & (64,186) & (128,913) & [ 1, 15, 17 ] & $Q_{16}$ \\
\hline
$\A{(34,612)_1}=\AG(34,2)$ & (154,459,306) & $(60,63,18,6,4,0,3)$ & $(0^{6},3^{3},0,4,0,6,0,9,6)$ & (12,4) & (6,1) & [ 1, 16, 17 ] & $\mathbb{Z}$ \\
$\A{(34,612)_2}=\AG(34,1)$ & (154,459,306) & $(17,136,0^{13},1)$ & $(0^{8},17,0^{6},17)$ & (68,4) & (272,50) & [ 1, 16, 17 ] & $Q_{17}$ \\
\hline
$\A{(35,680)}=\AG(35,1)$ & (171,510,340) & $(70,55,25,21)$ & $(0^{12},25,0,10)$ & (20,4) & (20,4) & [] & $K_2$ \\
\hline
$\A{(36,684)}=\AG(36,1)$ & (172,513,342) & $(18,153,0^{14},1)$ & $(0^{8},9^{2},0^{6},18)$ & (72,17) & (108,26) & [ 1, 17, 18 ] & $Q_{9}$ \\
\hline
$\A{(37,720)_1}=\AG(37,3)$ & (181,540,360) & $(72^{2},24,0,10,0,3)$ & $(0^{6},3,0,6,0,4,0^{3},12,0,12)$ & (48,48) & (24,12) & [ 1, 17, 19 ] & $\mathbb{Z}$ \\
$\A{(37,720)_2}=\AG(37,1)$ & (181,540,360) & $(27,144,9,0^{13},1)$ & $(0^{8},9,0,9,0^{5},19)$ & (72,17) & (108,26) & [ 1, 17, 19 ] & $Q_{9}$ \\
$\A{(37,720)_3}=\AG(37,2)$ & (181,540,360) & $(72^{2},12,24,0^{6},1)$ & $(0^{10},13,0^{3},24)$ & (48,36) & (48,17) & [ 1, 17, 19 ] & $Q_{12}$ \\
\hline
& & \\
\caption{Invariants for the known simplicial arrangements of rank $3$ with up to $37$ lines.}
\label{tab:invariants}
\end{longtable}}

\clearpage

\newcommand{\etalchar}[1]{$^{#1}$}
\providecommand{\bysame}{\leavevmode\hbox to3em{\hrulefill}\thinspace}
\providecommand{\MR}{\relax\ifhmode\unskip\space\fi MR }
\providecommand{\MRhref}[2]{%
  \href{http://www.ams.org/mathscinet-getitem?mr=#1}{#2}
}
\providecommand{\href}[2]{#2}

\addresseshere

\appendix

\newpage
\section{Tables of normals}
\label{app:lists}

Tables~\ref{tab:normals1} to \ref{tab:normals9} give the normals for the simplicial hyperplane arrangements of rank $3$ with at most $37$ hyperplanes, excluding the reducible near-pencil arrangements.
We use the notation $\A(m,r)_i$ to indicate the $i$-th hyperplane arrangement with $m$ hyperplanes and $r$ regions.
We use the following algebraic numbers:
\begin{align*}
\tau    & = (1+\sqrt5)/2\\
\rho    & = \text{real root of } x^3-3x-25\text{ at }\approx 3.26463299874008)\\
\gamma  & = \text{real root of } x^4-3x^3+3x^2-3x+1\text{ at }\approx 0.4643126132 \\
\{r_{i} & = \zeta_i+\zeta_i^{-1}~:~\text{ $\zeta_i$ a primitive $i$-th root of unity}, i \in [18] \} \\   
\end{align*}
The normals for the crystallographic arrangements and their orderings are chosen in a canonical way: we chose the root system of the chamber which is lexicographically the smallest one including those with permuted coordinates. 
The normals for each arrangement are also ordered lexicographically. 
This is why we need permutations in the list of wiring diagrams in Appendix \ref{app:wd}.
Although most of these simplicial arrangements are presented in Gr\"unbaum's catalogue \cite{grunbaum_2009}, it requires some work to extract the normals from his pictures. 
One can for example obtain a realization by using the matroid structure as performed in \cite{p-C10b}.
The reader interested to use the normals is invited to use the \LaTeX\ source to extract the data.
Alternatively, one may use the \texttt{Sage}-package \texttt{CN-HyperArr} \cite{cn_hyperarr}.

% \resizebox{\textwidth}{!}{
\begin{table}[H]
\footnotesize
\centering
\begin{tabular}{l p{0.8\linewidth}}
Name & Normals \\\toprule[1pt]
$\A{(6,24)}$        & ${(0, 0, 1)}$, ${(0, 1, 0)}$, ${(0, 1, 1)}$, ${(1, 0, 0)}$, ${(1, 0, 1)}$, ${(1, 1, 1)}$ \\\hline
$\A{(7,32)}$        & ${(0, 0, 1)}$, ${(0, 1, 0)}$, ${(0, 1, 1)}$, ${(1, 0, 0)}$, ${(1, 0, 1)}$, ${(1, 1, 0)}$, ${(1, 1, 1)}$ \\\hline
$\A{(8,40)}$        & ${(0, 0, 1)}$, ${(0, 1, 0)}$, ${(0, 1, 1)}$, ${(0, 1, 2)}$, ${(1, 0, 0)}$, ${(1, 0, 1)}$, ${(1, 1, 1)}$, ${(1, 1, 2)}$ \\\hline
$\A{(9,48)}$        & ${(0, 0, 1)}$, ${(0, 1, 0)}$, ${(0, 1, 1)}$, ${(0, 1, 2)}$, ${(1, 0, 0)}$, ${(1, 1, 0)}$, ${(1, 1, 1)}$, ${(1, 1, 2)}$, ${(1, 2, 2)}$ \\\hline
$\A{(10,60)_1}$     & ${(0, 0, 1)}$, ${(0, 1, 0)}$, ${(0, 1, 1)}$, ${(0, 1, 2)}$, ${(1, 0, 0)}$, ${(1, 0, 1)}$, ${(1, 0, 2)}$, ${(1, 1, 1)}$, ${(1, 1, 2)}$, ${(1, 1, 3)}$ \\
$\A{(10,60)_2}$     & ${(0, 0, 1)}$, ${(0, 1, 0)}$, ${(0, 1, 1)}$, ${(0, 1, 2)}$, ${(1, 0, 0)}$, ${(1, 0, 1)}$, ${(1, 1, 0)}$, ${(1, 1, 1)}$, ${(1, 1, 2)}$, ${(1, 2, 2)}$ \\
$\A{(10,60)_3}$     & ${(2\tau + 1, 2\tau, \tau)}$, ${(2\tau + 2, 2\tau + 1, \tau + 1)}$, ${(1, 1, 1)}$, ${(\tau + 1, \tau + 1, \tau)}$, ${(2\tau, 2\tau, \tau)}$, ${(\tau + 1, \tau + 1, 1)}$, ${(1, 1, 0)}$, ${(0, 1, 0)}$, ${(1, 0, 0)}$, ${(\tau + 1, \tau, \tau)}$ \\
\end{tabular}%}
\caption{The normals of all known simplicial hyperplane arrangements of rank $3$ with $6$ to $14$ hyperplanes.}
\label{tab:normals1}
 \end{table}
\begin{table}[!ht]
\footnotesize
\centering
\begin{tabular}{l p{0.8\linewidth}}
Name & Normals \\\toprule[1pt]
$\A{(15,120)}$      & ${(2\tau + 1, \tau + 1, \tau)}$, ${(2\tau + 2, \tau + 2, \tau + 1)}$, ${(3\tau + 1, 2\tau + 1, \tau + 1)}$, ${(2\tau + 1, 2\tau, \tau)}$, ${(2\tau + 2, 2\tau + 1, \tau + 1)}$, ${(1, 1, 1)}$, ${(\tau + 1, \tau + 1, \tau)}$, ${(2\tau, 2\tau, \tau)}$, ${(\tau + 1, \tau + 1, 1)}$, ${(1, 1, 0)}$, ${(0, 1, 0)}$, ${(1, 0, 0)}$, ${(\tau + 1, 1, 1)}$, ${(2\tau, \tau, \tau)}$, ${(\tau + 1, \tau, \tau)}$ \\
$\A{(15,128)_1}$    & ${(0, 0, 1)}$, ${(0, 1, 0)}$, ${(0, 1, 1)}$, ${(0, 1, 2)}$, ${(0, 1, 3)}$, ${(1, 0, 0)}$, ${(1, 0, 1)}$, ${(1, 0, 2)}$, ${(1, 1, 1)}$, ${(1, 1, 2)}$, ${(1, 1, 3)}$, ${(1, 1, 4)}$, ${(1, 2, 3)}$, ${(1, 2, 4)}$, ${(1, 2, 5)}$ \\
$\A{(15,128)_2}$    & ${(2\tau + 1, \tau + 1, 1)}$, ${(2\tau + 2, \tau + 2, \tau + 1)}$, ${(4\tau + 2, 3\tau + 1, \tau + 1)}$, ${(2\tau + 2, \tau + 2, 1)}$, ${(3\tau + 1, 2\tau + 1, \tau)}$, ${(2\tau + 1, 2\tau, \tau)}$, ${(2\tau + 2, 2\tau + 1, \tau)}$, ${(2\tau, 2\tau, \tau)}$, ${(\tau + 1, \tau + 1, 1)}$, ${(1, 1, 0)}$, ${(0, 1, 0)}$, ${(1, 0, 0)}$, ${(\tau + 1, 1, 0)}$, ${(2\tau + 2, \tau + 1, 1)}$, ${(3\tau + 1, 2\tau, \tau)}$ \\
$\A{(15,132)_1}$    & ${(0, 0, 1)}$, ${(0, 1, 1/9\rho^2 - 2/9\rho - 8/9)}$, ${(1, -1, -1)}$, ${(1, -1, 0)}$, ${(1, 0, -1)}$, ${(1, -1/9\rho^2 - 4/9\rho - 13/9, -1)}$, ${(1, 1/9\rho^2 + 4/9\rho + 4/9, -1/9\rho^2 - 4/9\rho - 13/9)}$, ${(1, 0, 1/9\rho^2 - 2/9\rho - 17/9)}$, ${(1, -1/9\rho^2 - 4/9\rho - 13/9, 1/9\rho^2 + 1/9\rho - 2/9)}$, ${(1, 1/9\rho^2 + 4/9\rho + 4/9, -1/3\rho - 5/3)}$, ${(1, 1/9\rho^2 + 4/9\rho + 4/9, -1/3\rho - 2/3)}$, ${(1, -1/9\rho^2 - 1/9\rho - 7/9, 1/9\rho^2 - 2/9\rho - 8/9)}$, ${(0, 1, 1/9\rho^2 + 1/9\rho - 20/9)}$, ${(1, 1/9\rho^2 + 4/9\rho + 4/9, -1/9\rho^2 - 7/9\rho - 19/9)}$, ${(1, -1/9\rho^2 - 1/9\rho - 7/9, 1/9\rho^2 - 2/9\rho + 1/9)}$ \\
\end{tabular}%}
\caption{The normals of all known simplicial hyperplane arrangements of rank three with $15$ to $17$ hyperplanes.}
\label{tab:normals2}
\end{table}

\begin{table}[!ht]
\footnotesize
\centering
\begin{tabular}{l p{0.8\linewidth}}
Name & Normals \\\toprule[1pt]
$\A{(18,180)_1}$     & ${(0, 0, 1)}$, ${(0, 1, 0)}$, ${(0, 1, 1)}$, ${(0, 1, 2)}$, ${(0, 1, 3)}$, ${(0, 2, 3)}$, ${(1, 0, 0)}$, ${(1, 0, 1)}$, ${(1, 0, 2)}$, ${(1, 1, 1)}$, ${(1, 1, 2)}$, ${(1, 1, 3)}$, ${(1, 1, 4)}$, ${(1, 2, 3)}$, ${(1, 2, 4)}$, ${(1, 2, 5)}$, ${(1, 3, 5)}$, ${(1, 3, 6)}$ \\
$\A{(18,180)_2}$     & ${(0, 0, 1)}$, ${(0, 1, 0)}$, ${(0, 1, 1)}$, ${(0, 1, 2)}$, ${(0, 1, 3)}$, ${(1, 0, 0)}$, ${(1, 0, 1)}$, ${(1, 1, 0)}$, ${(1, 1, 1)}$, ${(1, 1, 2)}$, ${(1, 1, 3)}$, ${(1, 2, 1)}$, ${(1, 2, 2)}$, ${(1, 2, 3)}$, ${(1, 2, 4)}$, ${(1, 3, 3)}$, ${(1, 3, 4)}$, ${(2, 3, 4)}$ \\
$\A{(18,180)_3}$     & ${(0, 1, 0)}$, ${(1, 0, 0)}$, ${(1, -1, 0)}$, ${(1, -2/3r_{9}^2 - 1/3r_{9} + 1/3, 0)}$, ${(1, -r_{9}^2 - r_{9}, 0)}$, ${(1, r_{9}, 0)}$, ${(1, -1/3r_{9}^2 + 1/3r_{9} - 1/3, 0)}$, ${(1, -r_{9}^2 + 1, 0)}$, ${(1, -r_{9}^2, 0)}$, ${(1, 1, -1)}$, ${(0, 1, -1)}$, ${(1, r_{9}^2 - 4, -r_{9}^2 + 3)}$, ${(1, -2r_{9}^2 + 3, r_{9}^2 - 2)}$, ${(1, -r_{9}^2 - 2/3r_{9} + 4/3, 1/3r_{9}^2 + 1/3r_{9} - 1)}$, ${(1, -2r_{9}^2 + 2, r_{9}^2 - 2)}$, ${(1, -r_{9} - 3, -r_{9}^2 + 3)}$, ${(0, 1, 3/19r_{9}^2 + 2/19r_{9} - 14/19)}$, ${(1, r_{9} + 1, -1)}$ \\
$\A{(18,180)_4}$     & ${(1, 2, 4)}$, ${(2, 3, 6)}$, ${(1, 1, 2)}$, ${(1, 0, 0)}$, ${(1, 2, 3)}$, ${(1, 1, 1)}$, ${(1, 2, 2)}$, ${(1, 1, 0)}$, ${(0, 1, 2)}$, ${(0, 1, 3)}$, ${(0, 1, 4)}$, ${(0, 1, 0)}$, ${(0, 1, 1)}$, ${(1, 2, 6)}$, ${(1, 1, 4)}$, ${(1, 2, 5)}$, ${(1, 3, 6)}$, ${(1, 1, 3)}$ \\
\end{tabular}%}
\caption{The normals of all known simplicial hyperplane arrangements of rank three with $18$ to $20$ hyperplanes.}
\label{tab:normals3}
\end{table}

\begin{table}[!ht]
\footnotesize
\centering
\begin{tabular}{l p{0.8\linewidth}}
Name & Normals \\\toprule[1pt]
$\A{(21,240)_1}$     & ${(0, 0, 1)}$, ${(0, 1, 0)}$, ${(0, 1, 1)}$, ${(0, 1, 2)}$, ${(0, 1, 3)}$, ${(0, 1, 4)}$, ${(0, 2, 3)}$, ${(0, 2, 5)}$, ${(1, 0, 0)}$, ${(1, 0, 1)}$, ${(1, 1, 1)}$, ${(1, 1, 2)}$, ${(1, 1, 3)}$, ${(1, 1, 4)}$, ${(1, 2, 3)}$, ${(1, 2, 4)}$, ${(1, 2, 5)}$, ${(1, 2, 6)}$, ${(1, 3, 6)}$, ${(1, 3, 7)}$, ${(2, 3, 7)}$ \\
$\A{(21,240)_2}$     & ${(0, 0, 1)}$, ${(0, 1, 0)}$, ${(0, 1, 1)}$, ${(0, 1, 2)}$, ${(0, 1, 3)}$, ${(0, 2, 3)}$, ${(1, 0, 0)}$, ${(1, 0, 1)}$, ${(1, 0, 2)}$, ${(1, 1, 1)}$, ${(1, 1, 2)}$, ${(1, 1, 3)}$, ${(1, 1, 4)}$, ${(1, 2, 2)}$, ${(1, 2, 3)}$, ${(1, 2, 4)}$, ${(1, 2, 5)}$, ${(1, 3, 4)}$, ${(1, 3, 5)}$, ${(1, 3, 6)}$, ${(2, 3, 6)}$ \\
$\A{(21,240)_3}$     & ${(0, 0, 1)}$, ${(0, 1, 0)}$, ${(0, 1, 1)}$, ${(0, 1, 2)}$, ${(0, 1, 3)}$, ${(0, 2, 3)}$, ${(1, 0, 0)}$, ${(1, 0, 1)}$, ${(1, 0, 2)}$, ${(1, 1, 1)}$, ${(1, 1, 2)}$, ${(1, 1, 3)}$, ${(1, 1, 4)}$, ${(1, 2, 3)}$, ${(1, 2, 4)}$, ${(1, 2, 5)}$, ${(1, 3, 5)}$, ${(1, 3, 6)}$, ${(2, 2, 5)}$, ${(2, 3, 6)}$, ${(2, 3, 7)}$ \\
$\A{(21,240)_4}$     & ${(0, 1, 0)}$, ${(1, 0, 0)}$, ${(1, -1, 0)}$, ${(1, -\tau, 0)}$, ${(1, -\tau - 1, 0)}$, ${(1, -\tau + 1, 0)}$, ${(1, -1/2\tau - 1/2, 0)}$, ${(1, -2, 0)}$, ${(1, -2\tau - 1, 0)}$, ${(1, \tau, 0)}$, ${(1, 1, -1)}$, ${(0, 1, -1)}$, ${(1, -\tau, \tau - 1)}$, ${(1, -\tau, 1/2\tau - 1/2)}$, ${(1, -3\tau + 3, 2\tau - 3)}$, ${(1, -2\tau + 1, 2\tau - 3)}$, ${(1, -3/2\tau - 1/2, 1/2\tau - 1/2)}$, ${(1, -3\tau, \tau - 1)}$, ${(0, 1, 1/2\tau - 1)}$, ${(1, \tau + 1, -1)}$, ${(1, -2\tau, \tau - 1)}$ \\
\end{tabular}%}
\caption{The normals of all known simplicial hyperplane arrangements of rank three with $21$ to $23$ hyperplanes.}
\label{tab:normals4}
\end{table}

\begin{table}[!ht]
\footnotesize
\centering
\begin{tabular}{l p{0.8\linewidth}}
Name & Normals \\\toprule[1pt]
$\A{(24,304)}$       & ${(1, -\sqrt2 - 2, \sqrt2 + 1)}$, ${(1, -\sqrt2 - 1, 2\sqrt2 + 2)}$, ${(1, -\sqrt2 - 2, 2\sqrt2 + 2)}$, ${(0, 1, 0)}$, ${(0, 1, -1)}$, ${(0, 1, -1/2\sqrt2 - 1)}$, ${(0, 1, -\sqrt2 - 1)}$, ${(0, 1, -\sqrt2 - 2)}$, ${(0, 1, -2\sqrt2 - 3)}$, ${(0, 0, 1)}$, ${(0, 1, \sqrt2 + 1)}$, ${(1, 0, -\sqrt2 - 2)}$, ${(1, -1, -\sqrt2 - 2)}$, ${(1, 0, -1)}$, ${(1, -1, -1)}$, ${(1, -\sqrt2, -1)}$, ${(1, -1, 0)}$, ${(1, -\sqrt2 - 1, -1)}$, ${(1, -2, 0)}$, ${(1, -\sqrt2, \sqrt2)}$, ${(1, -1, \sqrt2 + 1)}$, ${(1, -\sqrt2 - 1, \sqrt2)}$, ${(1, -2, \sqrt2 + 1)}$, ${(1, -\sqrt2 - 1, \sqrt2 + 1)}$ \\
$\A{(24,312)}$       & ${(0, 1, 0)}$, ${(1, 0, 0)}$, ${(1, -1, 0)}$, ${(1, -3/2, 0)}$, ${(1, -2, 0)}$, ${(1, -3, 0)}$, ${(1, 1/2\sqrt3 - 3/2, 0)}$, ${(1, \sqrt3 - 3, 0)}$, ${(1, -\sqrt3, 0)}$, ${(1, -1/2\sqrt3 - 3/2, 0)}$, ${(1, -\sqrt3 - 3, 0)}$, ${(1, \sqrt3, 0)}$, ${(1, 1, -1)}$, ${(0, 1, -1)}$, ${(1, -1/3\sqrt3 - 1, 1/3\sqrt3)}$, ${(1, 2\sqrt3 - 5, -\sqrt3 + 2)}$, ${(1, -1/4\sqrt3 - 5/4, 1/4\sqrt3 - 1/4)}$, ${(1, -5/3\sqrt3 + 1, 2/3\sqrt3 - 1)}$, ${(1, -2/3\sqrt3 - 1, 2/3\sqrt3 - 1)}$, ${(1, -3/4\sqrt3 - 5/4, 1/4\sqrt3 - 1/4)}$, ${(1, \sqrt3 - 5, -\sqrt3 + 2)}$, ${(1, -4/3\sqrt3 - 3, 1/3\sqrt3)}$, ${(0, 1, 1/13\sqrt3 - 4/13)}$, ${(1, \sqrt3 + 1, -1)}$ \\
$\A{(24,316)}$       & ${(0, 2, 1)}$, ${(0, 0, 1)}$, ${(0, 1, 3)}$, ${(0, 1, 2)}$, ${(1, 2, 1)}$, ${(1, 3, 3)}$, ${(1, 2, 2)}$, ${(1, 1, 1)}$, ${(1, 0, 0)}$, ${(2, 3, 4)}$, ${(1, 3, 4)}$, ${(2, 2, 3)}$, ${(1, 2, 3)}$, ${(2, 3, 5)}$, ${(1, 1, 2)}$, ${(1, 0, 1)}$, ${(2, 4, 7)}$, ${(2, 3, 6)}$, ${(1, 3, 5)}$, ${(1, 2, 4)}$, ${(1, 1, 3)}$, ${(1, 0, 2)}$, ${(1, 1, 4)}$, ${(0, 1, 1)}$ \\
\end{tabular}%}
\caption{The normals of all known simplicial hyperplane arrangements of rank three with $24$ to $25$ hyperplanes.}
\label{tab:normals5}
 \end{table}
\begin{table}[!ht]
\footnotesize
\centering
\begin{tabular}{l p{0.8\linewidth}}
Name & Normals \\\toprule[1pt]
$\A{(26,364)_1}$     & ${(0, 0, 1)}$, ${(0, 1, 0)}$, ${(0, 1, 1)}$, ${(0, 1, 2)}$, ${(0, 1, 3)}$, ${(0, 1, 4)}$, ${(0, 2, 3)}$, ${(0, 2, 5)}$, ${(1, 0, 0)}$, ${(1, 0, 1)}$, ${(1, 0, 2)}$, ${(1, 1, 1)}$, ${(1, 1, 2)}$, ${(1, 1, 3)}$, ${(1, 1, 4)}$, ${(1, 1, 5)}$, ${(1, 2, 3)}$, ${(1, 2, 4)}$, ${(1, 2, 5)}$, ${(1, 2, 6)}$, ${(1, 2, 7)}$, ${(1, 3, 6)}$, ${(1, 3, 7)}$, ${(1, 3, 8)}$, ${(2, 3, 7)}$, ${(2, 3, 8)}$ \\
$\A{(26,364)_2}$     & ${(0, 0, 1)}$, ${(0, 1, 0)}$, ${(0, 1, 1)}$, ${(0, 1, 2)}$, ${(0, 1, 3)}$, ${(0, 1, 4)}$, ${(0, 2, 3)}$, ${(1, 0, 0)}$, ${(1, 0, 1)}$, ${(1, 0, 2)}$, ${(1, 1, 1)}$, ${(1, 1, 2)}$, ${(1, 1, 3)}$, ${(1, 1, 4)}$, ${(1, 1, 5)}$, ${(1, 2, 3)}$, ${(1, 2, 4)}$, ${(1, 2, 5)}$, ${(1, 2, 6)}$, ${(1, 2, 7)}$, ${(1, 3, 6)}$, ${(1, 3, 7)}$, ${(1, 3, 8)}$, ${(2, 3, 7)}$, ${(2, 3, 8)}$, ${(2, 3, 9)}$ \\
\end{tabular}%}
\caption{The normals of all known simplicial hyperplane arrangements of rank three with $26$ to $28$ hyperplanes.}
\label{tab:normals6}
\end{table}

\begin{table}[!ht]
\footnotesize
\centering
\begin{tabular}{l p{0.8\linewidth}}
Name & Normals \\\toprule[1pt]
$\A{(29,440)}$       & ${(2\tau, \tau, 0)}$, ${(4\tau + 2, 2\tau + 2, \tau + 1)}$, ${(2\tau + 2, \tau + 2, \tau + 1)}$, ${(3\tau + 3, 2\tau + 2, \tau + 1)}$, ${(4\tau + 2, 3\tau + 1, 2\tau)}$, ${(3\tau + 1, 2\tau + 1, \tau + 1)}$, ${(4\tau + 2, 3\tau + 1, \tau + 1)}$, ${(\tau + 1, \tau, 0)}$, ${(2\tau + 2, \tau + 2, 1)}$, ${(3\tau + 1, 2\tau + 1, \tau)}$, ${(2\tau + 1, 2\tau, \tau)}$, ${(2\tau + 2, 2\tau + 1, \tau + 1)}$, ${(2\tau + 2, 2\tau + 1, \tau)}$, ${(1, 1, 1)}$, ${(\tau + 1, \tau + 1, \tau)}$, ${(2\tau, 2\tau, \tau)}$, ${(\tau + 1, \tau + 1, 1)}$, ${(1, 1, 0)}$, ${(0, 0, 1)}$, ${(0, 1, 0)}$, ${(0, 1, 1)}$, ${(1, 0, 0)}$, ${(\tau + 1, 1, 1)}$, ${(2\tau, \tau, \tau)}$, ${(\tau + 1, 1, 0)}$, ${(2\tau + 2, \tau + 1, \tau)}$, ${(\tau + 1, \tau, \tau)}$, ${(2\tau + 2, \tau + 1, 1)}$, ${(3\tau + 1, 2\tau, \tau)}$ \\
$\A{(29,448)_1}$     & ${(0, 0, 1)}$, ${(0, 1, 0)}$, ${(0, 1, 1)}$, ${(0, 1, 2)}$, ${(0, 1, 3)}$, ${(0, 1, 4)}$, ${(0, 2, 3)}$, ${(0, 2, 5)}$, ${(1, 0, 0)}$, ${(1, 0, 1)}$, ${(1, 0, 2)}$, ${(1, 1, 1)}$, ${(1, 1, 2)}$, ${(1, 1, 3)}$, ${(1, 1, 4)}$, ${(1, 1, 5)}$, ${(1, 2, 3)}$, ${(1, 2, 4)}$, ${(1, 2, 5)}$, ${(1, 2, 6)}$, ${(1, 2, 7)}$, ${(1, 3, 6)}$, ${(1, 3, 7)}$, ${(1, 3, 8)}$, ${(2, 2, 5)}$, ${(2, 2, 7)}$, ${(2, 3, 7)}$, ${(2, 3, 8)}$, ${(2, 3, 9)}$ \\
$\A{(29,448)_2}$     & ${(0, 0, 1)}$, ${(0, 1, 0)}$, ${(0, 1, 1)}$, ${(0, 1, 2)}$, ${(0, 1, 3)}$, ${(0, 1, 4)}$, ${(0, 2, 3)}$, ${(0, 2, 5)}$, ${(1, 0, 0)}$, ${(1, 0, 1)}$, ${(1, 0, 2)}$, ${(1, 1, 1)}$, ${(1, 1, 2)}$, ${(1, 1, 3)}$, ${(1, 1, 4)}$, ${(1, 1, 5)}$, ${(1, 2, 3)}$, ${(1, 2, 4)}$, ${(1, 2, 5)}$, ${(1, 2, 6)}$, ${(1, 2, 7)}$, ${(1, 3, 6)}$, ${(1, 3, 7)}$, ${(1, 3, 8)}$, ${(2, 2, 5)}$, ${(2, 3, 7)}$, ${(2, 3, 8)}$, ${(2, 3, 9)}$, ${(2, 4, 9)}$ \\
\end{tabular}%}
\caption{The normals of all known simplicial hyperplane arrangements of rank three with $29$ to $32$ hyperplanes}
\label{tab:normals7}
\end{table}

\begin{table}[!ht]
\footnotesize
\centering
\begin{tabular}{l p{0.8\linewidth}}
Name & Normals \\\toprule[1pt]
$\A{(33,576)}$     & ${(0, 1, 0)}$, ${(1, 0, 0)}$, ${(1, -1, 0)}$, ${(1, -r_{16}^2 + 2, 0)}$, ${(1, -1/2r_{16}^2, 0)}$, ${(1, -2, 0)}$, ${(1, -r_{16}^2 + 1, 0)}$, ${(1, -r_{16}^2, 0)}$, ${(1, r_{16}^3 - r_{16}^2 - 3r_{16} + 2, 0)}$, ${(1, r_{16}^3 - 3r_{16} - 2, 0)}$, ${(1, -r_{16}^2 + r_{16}, 0)}$, ${(1, -r_{16}, 0)}$, ${(1, -r_{16}^3 - r_{16}^2 + 3r_{16} + 2, 0)}$, ${(1, -r_{16}^3 + 3r_{16} - 2, 0)}$, ${(1, -r_{16}^2 - r_{16}, 0)}$, ${(1, r_{16}, 0)}$, ${(1, 1, -1)}$, ${(0, 1, -1)}$, ${(1, 1/2r_{16}^3 - 2r_{16} - 1, -1/2r_{16}^3 + 2r_{16})}$, ${(1, 2r_{16}^3 - 6r_{16} - 3, -r_{16}^3 + 3r_{16} + 1)}$, ${(1, 2r_{16}^3 - 3r_{16}^2 - 7r_{16} + 9, -2r_{16}^3 + 2r_{16}^2 + 7r_{16} - 7)}$, ${(1, 1/2r_{16}^3 - 2r_{16}^2 + 2, -1/2r_{16}^3 + r_{16}^2 + r_{16} - 2)}$, $(1, -1/2r_{16}^3 - r_{16}^2 + 3/2r_{16} + 2, 1/2r_{16}^3 + 1/2r_{16}^2 - 3/2r_{16} - 2)$, ${(1, r_{16}^3 - 5r_{16} + 1, -r_{16}^3 + 4r_{16} - 1)}$, ${(1, r_{16}^3 - 4r_{16} - 1, -r_{16}^3 + 4r_{16} - 1)}$, $(1, -3/2r_{16}^3 - 3/2r_{16}^2 + 9/2r_{16} + 4, 1/2r_{16}^3 + 1/2r_{16}^2 - 3/2r_{16} - 2)$, ${(1, 1/2r_{16}^3 - 2r_{16}^2 - r_{16} + 3, -1/2r_{16}^3 + r_{16}^2 + r_{16} - 2)}$, ${(1, r_{16}^3 - 2r_{16}^2 - 4r_{16} + 5, -2r_{16}^3 + 2r_{16}^2 + 7r_{16} - 7)}$, ${(1, r_{16}^3 - r_{16}^2 - 3r_{16} - 1, -r_{16}^3 + 3r_{16} + 1)}$, ${(1, 1/2r_{16}^3 - r_{16}^2 - 3r_{16}, -1/2r_{16}^3 + 2r_{16})}$, ${(0, 1, -5/31r_{16}^3 + 6/31r_{16}^2 + 19/31r_{16} - 29/31)}$, ${(1, r_{16} + 1, -1)}$, ${(1, -2r_{16}^3 - 2r_{16}^2 + 6r_{16} + 6, r_{16}^3 + r_{16}^2 - 3r_{16} - 4)}$ \\
\hline
\end{tabular}%}
\caption{The normals of all known simplicial hyperplane arrangements of rank three with $33$ to $36$ hyperplanes.}
\label{tab:normals8}
\end{table}
\begin{table}[!ht]
\footnotesize
\centering
\begin{tabular}{l p{0.8\linewidth}}
Name & Normals \\\toprule[1pt]
$\A{(37,720)_1}$     & ${(0, 0, 1)}$, ${(0, 1, 0)}$, ${(0, 1, 1)}$, ${(0, 1, 2)}$, ${(0, 1, 3)}$, ${(0, 1, 4)}$, ${(0, 2, 3)}$, ${(0, 2, 5)}$, ${(1, 0, 0)}$, ${(1, 0, 1)}$, ${(1, 0, 2)}$, ${(1, 1, 1)}$, ${(1, 1, 2)}$, ${(1, 1, 3)}$, ${(1, 1, 4)}$, ${(1, 1, 5)}$, ${(1, 2, 3)}$, ${(1, 2, 4)}$, ${(1, 2, 5)}$, ${(1, 2, 6)}$, ${(1, 2, 7)}$, ${(1, 3, 5)}$, ${(1, 3, 6)}$, ${(1, 3, 7)}$, ${(1, 3, 8)}$, ${(1, 3, 9)}$, ${(1, 4, 8)}$, ${(1, 4, 9)}$, ${(1, 4, 10)}$, ${(2, 3, 7)}$, ${(2, 3, 8)}$, ${(2, 3, 9)}$, ${(2, 4, 9)}$, ${(2, 4, 11)}$, ${(2, 5, 11)}$, ${(2, 5, 12)}$, ${(2, 5, 13)}$ \\
\end{tabular}%}
\caption{The normals of all known simplicial hyperplane arrangements of rank three with $37$ hyperplanes.}
\label{tab:normals9}
\end{table}

\clearpage
\section{Wiring diagrams}
\label{app:wd}

In this section, we reproduce wiring diagrams that correspond to the oriented matroids defined by the irreducible simplicial arrangements of rank 3.
A wiring diagram consists of a sequence of moves between wires in such a way that each pair of wires meets exactly once in some move (see for example \cite{oriented_matroids} for details). 
Figure \ref{wirsim} shows an example: the wires in this picture correspond to the lines of the arrangement $\A(31,480)=\AG(31,1)$. 
The moves are the intersection points. 
To encode the information of the moves it suffices to list for each intersection point the first and last label of the bundle of wires which meet. 
Thus in this example, the moves begin with $(11,12)$, $(3,4)$, $(21,22)$, $(12,16)$, $(16,18)$ and so on. 
The tables in the appendix contain such a description for each known irreducible simplicial arrangement with up to $37$ lines.
The first sequence is a bijection between these labels and the normals of Appendix \ref{app:lists}.

\begin{figure}[h]
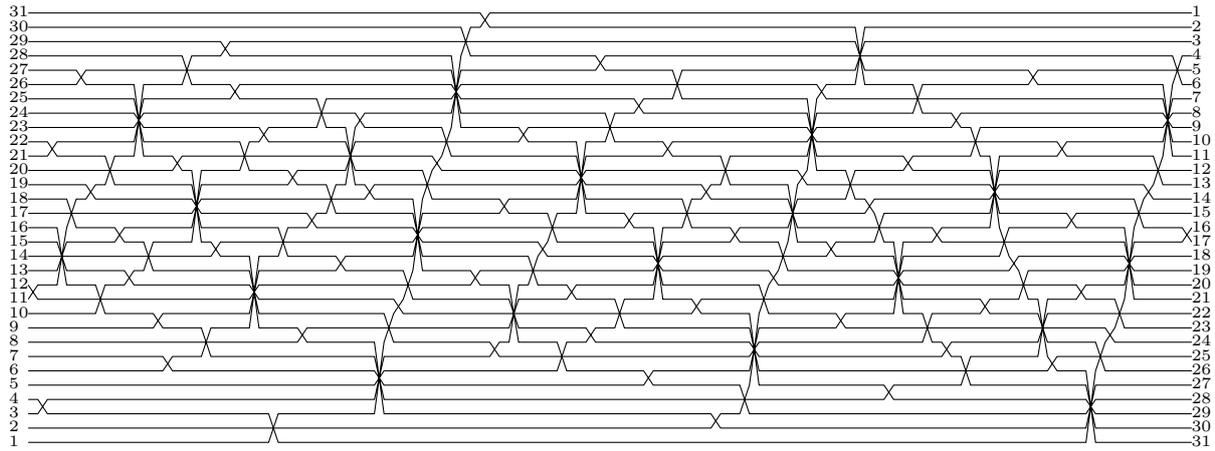

\vspace{20pt}
\begin{center}
\setlength{\unitlength}{0.45pt}

\end{center}
\caption{A wiring diagram for the arrangement $\AG(31,1)$.\label{wirsim}}
\end{figure}

{\tiny
\begin{longtable}{l p{0.82\linewidth}}
% \begin{table}[H]
% \footnotesize
% \centering
% \begin{tabular}{l p{0.8\linewidth}}
Name & $\pi$ --- Crossings\\\toprule[1pt]
$\A{(6,24)}$     & (1, 6, 2, 5, 3, 4) --- ((1, 2), (4, 5), (2, 4), (4, 6), (3, 4), (1, 3), (3, 5)) \\\hline
$\A{(7,32)}$     & (1, 6, 7, 2, 5, 3, 4) --- ((1, 3), (5, 6), (3, 5), (5, 7), (2, 3), (3, 5), (5, 6), (1, 3), (3, 5)) \\\hline
$\A{(8,40)}$     & (4, 7, 1, 2, 8, 6, 3, 5) --- ((5, 7), (1, 2), (2, 5), (5, 6), (6, 8), (4, 6), (3, 4), (1, 3), (3, 5), (5, 7), (4, 5)) \\\hline
$\A{(9,48)}$     & (7, 4, 6, 1, 8, 2, 9, 3, 5) --- ((5, 7), (4, 5), (2, 4), (7, 8), (4, 7), (3, 4), (7, 9), (6, 7), (4, 6), (1, 4), (4, 5), (5, 8), (8, 9)) \\\hline
$\A{(10,60)_1}$  & (10, 1, 8, 9, 2, 7, 3, 6, 4, 5) --- ((1, 4), (4, 5), (5, 7), (7, 9), (9, 10), (6, 7), (3, 6), (6, 9), (2, 3), (5, 6), (3, 5), (5, 7), (7, 8), (1, 3), (3, 5), (5, 7)) \\
$\A{(10,60)_2}$  & (4, 8, 7, 1, 9, 2, 10, 6, 3, 5) --- ((1, 2), (5, 7), (4, 5), (7, 9), (2, 4), (4, 7), (7, 8), (8, 10), (3, 4), (6, 8), (4, 6), (6, 7), (1, 4), (4, 6), (6, 9), (3, 4)) \\
$\A{(10,60)_3}$  & (3, 7, 6, 10, 2, 8, 5, 9, 1, 4) --- ((5, 7), (4, 5), (9, 10), (7, 9), (5, 7), (1, 5), (5, 6), (6, 8), (4, 6), (3, 4), (8, 10), (6, 8), (4, 6), (2, 4), (8, 9), (6, 8)) \\
\end{longtable}}

\end{document}